\documentclass[12pt,a4paper]{amsart}

\usepackage{euscript,amsfonts,amssymb,amsmath,amscd,amsthm,enumerate,hyperref}

\usepackage{comment}

\usepackage{mathrsfs}

\usepackage{graphicx}

\usepackage{color}

\usepackage[margin=1.1in]{geometry}

\usepackage{bbm}

\newtheorem{theorem}{Theorem}[section]
\newtheorem*{theorem*}{Theorem}
\newtheorem{lemma}[theorem]{Lemma}
\newtheorem{definition}[theorem]{Definition}
\newtheorem{proposition}[theorem]{Proposition}

\newtheorem{corollary}[theorem]{Corollary}
\newtheorem{assumption}[theorem]{Assumption}

\theoremstyle{remark}
\newtheorem{rmk}[theorem]{Remark}

\newcommand{\eps}{\varepsilon}
\newcommand{\U}{\mathcal U}
\newcommand{\E}{\mathbb E}

\newcommand{\Co}{\mathcal S}

\renewcommand{\aa}{\mathfrak a}
\newcommand{\bb}{\mathfrak b}

\newcommand{\A}{\mathcal A}

\newcommand{\e}{\mathfrak e}

\newcommand{\pp}{\mathbb{P}}
\newcommand{\qq}{\mathbb{Q}}
\newcommand{\rr}{\mathbb{R}}
\newcommand{\zz}{\mathbb{Z}}
\newcommand{\nn}{\mathbb{N}}
\newcommand{\ep}{\hfill \ensuremath{\Box}}
\newcommand{\eq}{\begin{equation}}
\newcommand{\en}{\end{equation}}

\numberwithin{equation}{section}

\newcommand{\ii}{\mathbf i}

\newcommand{\vv}{\mathbf v}
\newcommand{\G}{\mathcal G}

\title{Stochastic Airy Semigroup through tridiagonal matrices}

\author{Vadim Gorin}
\address{Department of Mathematics, Massachusetts Institute of Technology, MA, USA and Institute for Information Transmission Problems of Russian Academy of Sciences,  Russia}
\email{vadicgor@gmail.com}
\author{Mykhaylo Shkolnikov}
\address{ORFE Department, Princeton University, Princeton, NJ, USA}
\email{mshkolni@gmail.com}

\begin{document}

\begin{abstract}
We determine the operator limit for large powers of random tridiagonal matrices as the size of the matrix grows. The result provides a novel expression in terms of functionals of Brownian motions for the Laplace transform of the Airy$_\beta$ process, which describes the largest eigenvalues in the $\beta$ ensembles of random matrix theory. Another consequence is a Feynman-Kac formula for the stochastic Airy operator of Ram\'{i}rez, Rider, and Vir\'{a}g.

As a side result, we find that the difference between the area underneath a standard Brownian excursion and one half of the integral of its squared local times is a Gaussian random variable.
\end{abstract}

\maketitle

\section{Introduction}

This article is about spectral properties of random matrices, and we refer to \cite{AGZ}, \cite{Pastur_Shcherbina}, \cite{Forrester}, \cite{ABF} for modern general reviews. In random matrix ensembles one distinguishes a parameter $\beta$, which is typically equal to $1$, $2$, or $4$ in full Hermitian matrix models (such as Wigner
or Wishart ensembles) and corresponds to real, complex, or quaternion matrix elements. More generally, $\beta$ can be taken to be an arbitrary positive number, in relation with Coulomb log-gases, the Calogero-Sutherland quantum many-body system, random
tridiagonal matrices, Heckman-Opdam and Macdonald processes, see e.g.
\cite[Chapter 20 ``Beta Ensembles'']{ABF}, \cite{Dum}, \cite{BG_GFF} for the details.

\medskip

Here, we concentrate on \emph{edge limits} of random matrix ensembles describing the asymptotic behavior of the largest eigenvalues (and the corresponding eigenvectors). At $\beta=2$, i.e. for complex Hermitian matrices, there are many deep results in this direction. In particular, the properly centered and rescaled largest eigenvalue
converges to the Tracy-Widom law $F_2$ \cite{TW-U}, the point process describing all largest eigenvalues converges to the Airy point process, which is a part of $2$D Airy line ensemble \cite{Prah_Spohn}, \cite{CH} (the latter can be obtained by
considering the largest eigenvalues of corners of random matrices). All these results are very robust and have been proved rigorously in great generality, see e.g. \cite{Soshnikov}, \cite{Peche}, \cite{Sodin_edge}, \cite{Pastur_Shcherbina_edge}, \cite{DG}, \cite{BEY}, \cite{KRV}, \cite{BFG}. Furthermore, the universality of these objects extends far beyond random matrix theory, see e.g. \cite{BG_Spb}, \cite{BP_Lectures}, \cite{Johansson_lectures},
\cite{Corwin_ICM} and references therein. Several descriptions of the limiting objects are known: tracktable expressions of their correlation functions (see e.g. \cite{TW-U}, \cite{Prah_Spohn}) and Laplace transforms (see e.g. \cite{Soshnikov}, \cite{Okounkov}, \cite{Sodin_ICM}), and a conjectural description through the so-called Brownian Gibbs property (see \cite[Section 3.2]{CH}).

\medskip

While for $\beta=1$ there are still many results parallel to the $\beta=2$ case, there is much less understanding for general values of $\beta>0$. For the general $\beta$ analogues of $F_2$ and the Airy point process, the only known identification is via the spectrum of the stochastic Airy operator \cite{Edelman_Sutton}, \cite{RRV}, and no analytic formulas for correlation functions or Laplace transforms are known. Moreover, even the existence of the Airy line ensemble for general $\beta$ has not been established.

\medskip

From the analytic point of view, the main difficulty for general $\beta$ is that the determinantal/Pfaffian formulas for the correlation functions available for $\beta=1,\,2,\,4$ are not known to extend to other values of $\beta$. A recent alternative approach producing explicit formulas through \emph{Macdonald processes} \cite{BigMac}, \cite{BCGS} does work for $\beta$ ensembles (see \cite{BG_GFF}, \cite{BG_Appendix}), but the edge limits are
not yet accessible through these techniques.

\medskip

Another approach, which has proved to be very successful for $\beta=1,\,2,\,4$, is the \emph{moments method}. In the present article we prove that the latter approach \emph{can} be adapted to the study of the edge limits of general $\beta$ ensembles. This leads to several outcomes. First, we prove that the Laplace transform of the point process of the rescaled largest eigenvalues in the
Gaussian $\beta$ ensemble (and more general random matrices) converges to the Laplace transform of the Airy$_\beta$ point process and establish a novel formula for the latter in terms of
a functional of Brownian motion. This is closely related to our second result: the identification and proof of a Feynman-Kac formula for the \emph{stochastic Airy semigroup}, the semigroup associated with the stochastic Airy operator.

\medskip

It is known that Laplace transforms can be used to study various properties of the underlying point processes. For instance, by integrating one should be able to access \emph{linear statistics} of the Airy$_\beta$ point process (in addition to their intrinsic interest, the latter are also important for the study of rigidity properties, cf. \cite{Ghosh_Peres}). On the other hand, by sending the parameter of the Laplace transform to infinity, one should be able to find the Tracy-Widom laws $F_\beta$. We postpone the discussion of these possible applications to future papers. Instead, we present a rather unexpected consequence: comparing our results with the literature for $\beta=2$ we find a novel identity involving the Brownian excursion area and the local times of the same excursion. 

\medskip

From the technical point of view, our main result is the computation of the asymptotics of matrix elements of \emph{large} powers of random tridiagonal matrices. More precisely, for a matrix of size $N\times N$ we deal with powers of the order $N^{2/3}$. In the
case when the powers do not grow or grow slower than $N^{2/3}$, such asymptotics has been previously analyzed in \cite{DE2}, \cite{Wong}, \cite{Duy}, but for the analysis of the fast growing powers (directly related to the edge asymptotics of $\beta$ ensembles) many new ideas are necessary. In particular, in our proofs we heavily rely
on \emph{strong} invariance principles, i.e. statements about the convergence of the trajectories of (conditioned) random walks to those of Brownian motions (or bridges) with a very precise control of errors, as in \cite{Kh}, \cite{LTF}, \cite{BK2}. In addition, we
use \emph{path transformations} linking discrete local times of random walks to time-changed versions of the same random walk, see \cite{AFP}, and also \cite{Jeulin}, \cite{Biane-Yor}, \cite{CSY} for the continuous analogues.

\medskip

We proceed to a detailed exposition of our results.

\medskip

\noindent {\bf Notation.} In what follows $C$ stands for a positive constant whose exact value is not important for us and might change from line to line.

\medskip

\noindent {\bf Acknowledgement.}  We would like to thank Simone Warzel for suggesting the strategy for the proof of the trace formula \eqref{trace_formula}, and Sasha Sodin for helpful discussions. V.~G. was partially supported by the NSF grant DMS-1407562. M.~S. was partially supported by the NSF grant DMS-1506290.
\section{Setup and results}
\label{Section_setup}


Given two sequences of independent random variables $\aa(m)$, $m\in\nn$ and
$\bb(m)$, $m\in\nn$, we define for each $N\in\nn$ the $N\times N$ symmetric
tridiagonal matrix $M_N=(M_N[m,n])_{m,n=1}^N$ by setting $M_N[m,m]=\aa(m)$,
$m=1,2,\ldots,N$ and $M[m,m+1]=\bb(m)$, $m=1,2,\ldots,N-1$: \eq
M_N=\left(\begin{array}{ccccc}
\aa(1) & \bb(1) & 0 & \cdots & 0 \\
\bb(1) & \aa(2) & \bb(2) & \ddots & \vdots \\
0 & \bb(2) & \aa(3) & \ddots & 0 \\
\vdots & \ddots & \ddots & \ddots & \bb(N-1) \\
0 & \cdots & 0 & \bb(N-1) & \aa(N)
\end{array}\right)
\en In this paper, we study $M_N$  in the asymptotic regime
$N\to\infty$. \footnote{All our arguments extend to the case when the entries $\aa(m)$, $m=1,2,\ldots,N$ and $\bb(m)$, $m=1,2,\ldots,N-1$ vary with $N$. However, to keep the notation reasonable we have
decided to work in the less general setup of no dependence on $N$.}

\medskip

The case when, for a fixed $\beta>0$, all $\aa(m)$, $m\in\nn$ have the normal
distribution $N(0,2/\beta)$, and the $\bb(m)$, $m\in\nn$ are $\beta^{-1/2}$
multiples of $\chi$-distributed random variables with parameters $\beta m$,
$m\in\nn$ is of particular interest. Here, the density of the $\chi$ distribution
with parameter $a$ on $\rr_{\ge 0}$ is
\[
\frac{2^{1-k/2}}{\Gamma(a/2)}\,x^{a-1}\,e^{-x^2/2}, \quad x>0.
\]
In this situation, the joint density of the $N$ eigenvalues
$\lambda_1\le \lambda_2\le\dots\le \lambda_N$ of $M_N$ is proportional to
\[
\prod_{i<j} (\lambda_j-\lambda_i)^{\beta}\,
\prod_{i=1}^N e^{-\frac{\beta}{4} \lambda_i^2},
\]
and the corresponding joint distribution is usually referred to as the \textit{Gaussian $\beta$ ensemble}, see e.g. \cite{DE}. More generally, we work with arbitrary sequences of independent random variables $\aa(m)$, $m\in\nn$ and $\bb(m)$, $m\in\nn$ satisfying the following assumption.

\begin{assumption} \label{Assumptions}
The sequences $\aa(m)$, $m\in\nn$ and $\bb(m)$, $m\in\nn$ of independent random variables satisfy, with the notation
\eq
\bb(m)=\sqrt{m}+\xi(m),\quad m\in\nn,
\en
\begin{enumerate}[(a)]
\item as $m\to\infty$, $\big|\E[\aa(m)]\big|=o(m^{-1/3})$, $\big|\E[\xi(m)]\big|=o(m^{-1/3})$;
\item there exist nonnegative constants $s_\aa$, $s_\xi$ such that $\frac{s_a^2}{4}+s_\xi^2=\frac{1}{\beta}$ and, as $m\to\infty$, $\E[\aa(m)^2]=s_\aa^2+o(1)$, $\E[\xi(m)^2]=s_\xi^2+o(1)$;
\item there exist constants $C>0$ and $0<\gamma<2/3$ such that
\[
\E\big[|\aa(m)|^\ell\big]\le C^\ell \ell^{\gamma\ell}\quad\text{and}\quad \E\big[|\xi(m)|^\ell\big]\le C^\ell \ell^{\gamma \ell}\quad\text{for all }\,m,\ell\in\nn.
\]
\end{enumerate}
\end{assumption}

\smallskip

In particular, we have the following simple lemma.

\begin{lemma}
If all $\aa(m)$, $m\in\nn$ are $N(0,2/\beta)$-distributed and the
$\sqrt{\beta}\,\bb(m)$, $m\in\nn$ are $\chi$-distributed with parameters $\beta m$,
$m\in\nn$, respectively, then Assumption \ref{Assumptions} holds with
$s_{\aa}/2=s_\xi=\frac{1}{\sqrt{2\beta}}$.
\end{lemma}

\begin{proof} 
The result is immediate for $\aa(m)$, $m\in\nn$. For $\xi(m)$, $m\in\nn$ it follows from known tail estimates for $\chi$ random variables, see e.g. \cite[Section 4.1, Lemma 1]{LM}.
\end{proof}

We start the study of the $N\to\infty$ limit of $M_N$ by recalling the \emph{semicircle law}. For each $N\in\nn$, let $\lambda_N^1\ge\lambda_N^2\ge\cdots\ge\lambda_N^N$ denote the ordered eigenvalues of $M_N$. Consider the random probability measure
\begin{equation}
\label{eq_empirical}
\rho_N=\frac{1}{N}\sum_{i=1}^N \delta_{\lambda_N^i/\sqrt{N}}\,.
\end{equation}
\begin{proposition} \label{Prop_semicircle}
Under Assumption \ref{Assumptions}, as $N\to\infty$, the random measure $\rho_N$ converge weakly, in probability, to the deterministic measure $\mu$ with density
 $$
  \frac{1}{2\pi} \sqrt{4-x^2},\quad -2<x<2.
 $$
\end{proposition}

\smallskip

For the Gaussian $\beta$ ensemble, Proposition \ref{Prop_semicircle} is well-known and can be proven in several ways, cf. \cite{AGZ}, \cite{Dum}. For the sake of completeness, we provide a proof
of our more general statement in the appendix.

\medskip

Proposition \ref{Prop_semicircle} gives the leading order asymptotics of the
normalized spectral measure $\rho_N$. Refinements of this statement are available in
at least three different directions. The first one studies the higher order
asymptotics of $\rho_N$ in the same coordinates, that is, the fluctuations of
$\rho_N$ around the semicircle distribution $\mu$. This is referred to as
\emph{global} asymptotics. In this direction, we prove in the appendix a Central
Limit Theorem (CLT) for the joint fluctuations of multiple corners of $M_N$. Note
that the Gaussian nature of the fluctuations (and the corresponding covariance structure) for a single matrix is well-known, cf. \cite{Johansson}, \cite{AGZ}, \cite{Dum}.

\medskip

The second refinement is the study of the \emph{local} asymptotics of the eigenvalues in the bulk of the spectrum. A typical question in this direction is the asymptotic distribution of the rescaled spacing $\sqrt{N}(\lambda^{\lfloor N/2\rfloor}_N-\lambda^{\lfloor N/2\rfloor+1}_N)$. We do not address this limiting regime in the present paper and instead refer to \cite{VV} for results of this type
for random tridiagonal matrices.

\medskip

The third refinement is the investigation of the asymptotics of the \emph{extreme}
eigenvalues of $M_N$ and the corresponding eigenvectors, which is known as
\emph{edge} asymptotics. In this direction, it is shown in \cite{RRV} (see also
\cite{KRV}) that the random variable $N^{1/6}(\lambda^i_N-2\sqrt{N})$ converges in
distribution for every fixed $i\in\nn$. The limit of the corresponding eigenvector
is also studied therein. Our main results are closely related to this work.

\medskip

Let us now present the main results of this paper. Fix a (possibly unbounded) interval $\A\subset\rr_{\ge0}$, consider a probability space which supports a standard Brownian motion $W$, and consider for each $T>0$ the following (random) kernel on $\mathbb R_{\ge 0} \times\mathbb R_{\ge 0}$:
\begin{equation}\label{def:kernel}
\begin{split}
& K_\A(x,y;T)=\frac{1}{\sqrt{2\pi T}} \exp\left(-\frac{(x-y)^2}{2T}\right) \\ 
& \qquad\qquad\quad \E_{B^{x,y}} \Biggl[\mathbf{1}_{\{\forall t:\,B^{x,y}(t)\in\A\}} \exp\left(
-\frac{1}{2}\int_0^T B^{x,y}(t)\,\mathrm{d}t+\frac{1}{\sqrt{\beta}}\,\int_0^\infty
L_a(B^{x,y})\,\mathrm{d}W(a)\right)\Biggr].
\end{split}
\end{equation}
Here, $B^{x,y}$ is a standard Brownian bridge starting at $x$ at time $0$ and ending at $y$ at time $T$ which is independent of $W$; the $L_a(B^{x, y})$ are the local times accumulated by $B^{x,y}$ at level $a$ on $[0,T]$; and the expectation $\E_{B^{x,y}}$ is taken only with respect to $B^{x,y}$. We define $\U_\A(T)$, $T>0$
as the integral operators on $\rr_{\ge0}$ with kernels $K_\A(x,y;T)$, $T>0$, respectively. In order to be able to make statements about \textit{multiple} operators $\U_\A(T)$, we use the \textit{same} path of $W$ in \eqref{def:kernel} and define the stochastic integral with respect to $W$ therein according to the almost sure procedure described in \cite{Ka}. For notational convenience, we let $\U_\A(0)$ be the the orthogonal projector from $L^2(\rr_{\ge 0})$ onto $L^2(\A)$; in particular, when $\A=\rr_{\ge 0}$, then $\U_\A(0)$ is the identity operator.

\begin{proposition} \label{Proposition_trace_class}
For each $T>0$, $\U_\A(T)$ is almost surely a symmetric non-negative trace class
operator on $L^2(\rr_{\ge0})$ satisfying the trace formula
\begin{equation}\label{trace_formula}
\mathrm{Trace}(\U_\A(T))=\int_{\rr_{\ge0}} K_\A(x,x;T)\,\mathrm{d}x.
\end{equation}
\end{proposition}
\begin{proposition} \label{Proposition_semigroup}
Operators $\U_\A(T)$, $T\ge 0$ have the almost sure semigroup property: for any
$T_1,T_2\ge0$, it holds $\U_\A(T_1)\,\U_\A(T_2)=\U_\A(T_1+T_2)$ with probability
one.
\end{proposition}
\begin{proposition} \label{Proposition_continuous}
The semigroup $\U_\A(T)$, $T\ge0$ is $L^2$-strongly continuous, that is, for any $T\ge0$ and $f\in L^2(\mathbb R_{\ge 0})$, it holds $\lim_{t\to T} \E\big[\|\U_\A(T)f-\U_\A(t)f\|^2\big]=0$.
\end{proposition}

\begin{proposition} \label{Proposition_spectrum_of_semigroup}
 There exists an orthonormal basis of random vectors $\vv^1_\A,\vv^2_\A,\ldots\in L^2(\mathbb \A)\subset L^2(\mathbb R_{\ge 0})$ and random variables $\eta^1_\A \ge \eta^2_\A \ge \dots$ defined on the same probability space as $\,\U_\A(T)$, $T>0$ such that, for each $T> 0$, the spectrum of $\U_\A(T)$ (as an operator on $L^2(\A)$) consists of eigenvalues $\exp(T\eta^i_\A/2)$, $i\in\nn$ corresponding to the eigenvectors $\vv^i_\A$, $i\in\nn$, respectively.
\end{proposition}

\noindent The proofs of Propositions \ref{Proposition_trace_class},
\ref{Proposition_semigroup}, and \ref{Proposition_continuous} are given in Section \ref{Section_properties}, and Proposition
\ref{Proposition_spectrum_of_semigroup} is established in Section \ref{Section_properties_2}.

\medskip

Our interest in the operators $\U_\A(T)$, $T>0$ is based on their appearance in the $N\to\infty$ edge limit of the matrix $M_N$ and its submatrices. More specifically, let $\Co$ denote the set of all locally integrable functions $f$ on $\rr_{\ge0}$ which grow subexponentially fast at infinity (that is, for which there exists a
$\delta>0$ such that $f(x)=O(\exp(x^{1-\delta}))$ as $x\to\infty$). Further, for any $N\in\nn$ and $f\in\Co$, write $\pi_N f$ for the vector in $\rr^N$ with components $N^{1/6}\,\int_{N^{-1/3}(N-i)}^{N^{-1/3}(N-i+1)} f(x)\,\mathrm{d}x$, $i=1,2,\ldots,N$ and $(\pi_N f)'$ for its transpose. In addition, define the $N\times N$ matrix
\begin{equation*}
\mathcal M(T,\A,N)=\frac{1}{2}\left( \left(\frac{M_{N;\A}}{2\sqrt N}\right)^{\lfloor T N^{2/3} \rfloor}+
\left(\frac{M_{N;\A}}{2\sqrt{N}}\right)^{\lfloor T N^{2/3}\rfloor-1}\right),
\end{equation*}
where $M_{N;\A}$ is the restriction of $M_N$ onto $\A$, so that the $(i,j)$-th entry of $M_{N,\A}$ is equal to that of $M_N$ if $\frac{N-i+1/2}{N^{1/3}},\frac{N-j+1/2}{N^{1/3}}\in\A$ and zero otherwise. In particular, $M_{N;[0,\infty)}=M_N$.

\begin{theorem} \label{Theorem_Main}
Under the Assumption \ref{Assumptions} we have
\begin{equation*}
\lim_{N\to\infty} \, \mathcal M(T,\A,N) =\U_\A(T),\quad T\ge 0
\end{equation*}
in the following senses.
\begin{enumerate}[(a)]
\item Weak convergence: For any $f,g\in \Co$ and $T\ge 0$, we have
\begin{equation*}
\lim_{N\to\infty} \, (\pi_N f)'\,\mathcal M(T,\A,N)\,(\pi_N g)=\int_{\rr_{\ge0}} \big(\U_A(T)f\big)(x)\,g(x)\,\mathrm{d}x
\end{equation*}
in distribution and in the sense of moments.
\item Convergence of traces: For any $T\ge 0$ we have
\begin{equation*}
\lim_{N\to\infty} {\rm Trace}\big(\mathcal M(T,\A,N)\big)={\rm Trace}\big(\U_\A(T)\big)
\end{equation*}
in distribution and in the sense of moments.
\item The convergences in parts (a) and (b) also hold jointly for any finite collection of $T$'s, $\A$'s, $f$'s, and $g$'s.
\end{enumerate}
The Brownian motion $W$ in the definition of $\,\U_\A(T)$ arises hereby as the following limit in distribution with respect to the Skorokhod topology:
\begin{equation}\label{eq_limit_BM_appearance}
W(a)=\sqrt{\beta}  \lim_{N\to\infty}  N^{-1/6} \sum_{n=N-\lfloor N^{1/3} a \rfloor}^N \left(\xi(n)+\frac{\aa(n)}{2}\right), \quad a\ge 0.
\end{equation}
\end{theorem}

\noindent The proof of Theorem \ref{Theorem_Main} is given in Section \ref{Section_complete_proof}.

\begin{rmk}
We recall that, for deterministic operators, weak convergence together with the convergence of their traces imply other stronger forms of convergence, in particular, the convergence in the
trace-class norm, see e.g. \cite[Section 2]{Simon}. Yet, when we speak about the convergence of finite-dimensional distributions of random operators, sticking to the statements $(a)$ and $(b)$ seems
quite natural.
\end{rmk}

The convergence of the traces ${\rm Trace}\bigl(\mathcal M(T,\A,N)\bigr)$ as
$N\to\infty$ implies the convergence of the eigenvalues of $M_{N;\A}$ in the same
limit. Let $\lambda^1_{N;\A}\ge\lambda^2_{N;\A}\ge\cdots\ge\lambda^N_{N;\A}$ denote
the eigenvalues of the matrix $N^{1/6}(M_{N;\A}-2\sqrt{N})$.

\begin{corollary} \label{Corollary_spectrum}
In the notations of Proposition \ref{Proposition_spectrum_of_semigroup}, one has the
convergence in distribution
\begin{equation}\label{eq_limit_trace}
\sum_{i=1}^N e^{T\lambda^i_{N;\A}/2}
\longrightarrow_{N\to\infty} \sum_{i=1}^\infty e^{T\eta^i_\A/2}
={\rm Trace}(\U_\A(T))
\end{equation}
jointly for any finitely many $T$'s and $\A$'s. Therefore, one also has
\begin{equation}\label{eq_edge_limit}
\lambda^i_{N;\A}\longrightarrow_{N\to\infty} \eta^i_\A
\end{equation}
jointly for any finitely many $i$'s and $\A$'s.
\end{corollary}

\begin{rmk}
If we replace $\mathcal M(T,\A,N)$ by $-\mathcal M(T,\A,N)$, then limit theorems similar to Theorem \ref{Theorem_Main} and Corollary \ref{Corollary_spectrum} will hold for this new object (see Remark \ref{Remark_even_and_odd} for more details). The latter give the asymptotics of the smallest eigenvalues of $M_N$. Interestingly, while for the variance constants $s_\aa$, $s_\xi$ corresponding to the Gaussian $\beta$ ensemble the limits of the largest and the smallest eigenvalues are independent, this is not true in general.
\end{rmk}


\noindent The proof of Corollary \ref{Corollary_spectrum} is given in Section
\ref{Section_extreme_convergence}.

\medskip

In order to compare with the previous work on the subject, we take $\A=\rr_{\ge0}$
and omit it from the notations. In other words, we consider only the spectrum of the
original matrix $M_N$. In this case, an alternative derivation of the edge limit
theorem and another interpretation of the limits $\eta^i$ were given in \cite{RRV}.
There, the authors make sense of the stochastic Airy operator $SAO_\beta$
\begin{equation*}
SAO_\beta=-\frac{d^2}{da^2}+a+\frac{2}{\sqrt{\beta}}\,W'(a)
\end{equation*}
on $L^2(\rr_{\ge0})$ with a Dirichlet boundary condition at zero by appropriately
defining an orthnormal basis of its eigenfunctions and the corresponding eigenvalues
$-\eta^i$, $i\in\nn$ (see \cite[Section 2]{RRV}, and also \cite{Bl_thesis},
\cite{Minami}). In addition, they show that the leading eigenvalues of $M_N$ (and
the corresponding eigenvectors) converge to the leading eigenvalues (eigenvectors)
of $SAO_\beta$. Note that, since the white noise $W'(a)$, $a\ge0$ is a generalized
function, special care is required in defining the operator $SAO_\beta$.

\begin{corollary} \label{Corollary_generator}
Let $\A=\rr_{\ge0}$ and, for any $T\ge 0$, define $e^{-\frac{T}{2}SAO_\beta}$ as the
unique operator on $L^2(\rr_{\ge0})$ with the same orthonormal basis of
eigenfunctions as $SAO_\beta$ and the corresponding eigenvalues $e^{T\eta^1/2}\ge
e^{T\eta^2/2}\ge\cdots$. If one couples $e^{-\frac{T}{2}SAO_\beta}$ with $\U(T)$ by identifying the Brownian motions $W$ in their respective definitions, then for
each $T\ge 0$, the operators $e^{-\frac{T}{2}SAO_\beta}$ and $\U(T)$ coincide with probability one.
\end{corollary}
\noindent The proof of Corollary \ref{Corollary_generator} is given in Section \ref{Section_properties_2}. Proposition \ref{Proposition_semigroup} and Corollary \ref{Corollary_generator} lead to the name \emph{stochastic Airy semigroup} for the
operators $\U(T)$, $T\ge0$.

\medskip

The relationship between $SAO_\beta$ and the operators $\U(T)$, $T>0$ can be viewed
as a variant of the Feynman-Kac formula for Schroedinger operators (see e.g.
\cite[Section 6, equation (6.6)]{Si1} for the case when the potential is a
deterministic function). However, since in the case of $SAO_\beta$ the potential is
given by the generalized function $a+\frac{2}{\sqrt{\beta}}W'(a)$, such a result
seems to be beyond the scope of the previous literature. A notable exception is the
``zero temperature case'' $\beta=\infty$ which falls into the framework of the usual
Feynman-Kac formula. In that case, a path transformation argument allows to recast
the Feynman-Kac identity
\begin{equation*}
{\rm Trace}(\U(T))=\frac{1}{\sqrt{2\pi T}}\,\int_0^\infty \E_{B^{x,x}} \bigg[\mathbf{1}_{\{\forall t:\,B^{x,x}(t)\ge 0\}}\,\exp\left(-\frac{1}{2}\int_0^T B^{x,x}(t)\,\mathrm{d}t\right)\bigg]\,\mathrm{d}x
\end{equation*}
for the trace of $\U(T)$ as
\begin{equation*}
{\rm Trace}(\U(T))=\sqrt{\frac{2}{\pi}}\,T^{-3/2}\,\E\left[\exp\Big(-{\frac{T^{3/2}}{2}\,\int_0^1 \e(t)\,\mathrm{d}t}\Big)\right]
\end{equation*}
where $\e$ is a standard Brownian excursion on the time interval $[0,1]$ (see the proof of Proposition \ref{Proposition_expectation} for the details). Since $SAO_\infty$ is the deterministic Airy operator, the latter formula is the well-known series representation for the Laplace transform of the Brownian excursion area $\int_0^1 \e(t)\,\mathrm{d}t$ (see e.g. \cite[Section 13]{Ja}).

\medskip

We now turn to the special value $\beta=2$, in which case the edge asymptotics is
much better understood. In particular, there exist formulas for the moments of the
limiting traces of the form \eqref{eq_limit_trace}. The first moment admits a
particularly simple formula and is given (in our notation) by the following
proposition from \cite{Okounkov}.

\begin{proposition}[{\cite[Section 2.6.1]{Okounkov}}] \label{Proposition_expectation_Ok}
Take $\beta=2$ and $\A=\rr_{\ge0}$. Then, for all $T>0$,
\begin{equation*}
\E \big[{\rm Trace}(\U(T))\big]
=\E\bigg[\sum_{i\ge 1} e^{T\,\eta^i/2}\bigg]
=\sqrt{\frac{2}{\pi}} \, T^{-3/2}\,e^{T^3/96}.
\end{equation*}
\end{proposition}
\noindent On the other hand, our expression for the kernel $K(x,y;T)$ and a suitable path tranformation allow to write the same trace in terms of a functional of a standard Brownian excursion $\e$ on the time interval $[0,1]$.

\begin{proposition} \label{Proposition_expectation}
Take $\A=\rr_{\ge0}$. Then, for all $T>0$,
\begin{equation} \label{eq_expectation_of_moment}
\begin{split}
\E \big[{\rm Trace}(\U(T))\big]
&=\E\bigg[\sum_{i\ge 1} e^{T\,\eta^i/2}\bigg] \\
&=\sqrt{\frac{2}{\pi}}\, T^{-3/2}\, \E\bigg[\exp\bigg(-\frac{T^{3/2}}{2}\int_0^1 \e(t)\,\mathrm{d}t + \frac{T^{3/2}}{2\beta}\int_0^\infty (l_y)^2 \,\mathrm{d}y\bigg)\bigg],
\end{split}
\end{equation}
where $\e$ is a standard Brownian excursion on the time interval $[0,1]$, and each $l_y$ is the total local time of $\e$ at level $y$.
\end{proposition}
\noindent The proof of Proposition \ref{Proposition_expectation} is given in Section
\ref{Section_properties_2}.

\medskip

Comparing Propositions \ref{Proposition_expectation_Ok} and
\ref{Proposition_expectation} one obtains the following corollary of independent interest.

\begin{corollary} \label{Cor_excursion_identity}
Let $\e$ be a standard Brownian excursion on the time interval
$[0,1]$ and, for each $y\ge0$, let $l_y$ be the total local time of $\e$ at level $y$. Then,
\begin{equation}
\label{eq_excursion_functional}
\int_0^1 \e(t)\,\mathrm{d}t -
\frac{1}{2}\,\int_0^\infty (l_y)^2\,\mathrm{d}y
\end{equation}
is a Gaussian random variable of mean $0$ and variance $\frac{1}{12}$.
\end{corollary}

\noindent The proof of Corollary \ref{Cor_excursion_identity} is given in Section
\ref{Section_properties_2}. To the best of our knowledge, Corollary
\ref{Cor_excursion_identity} is new and did not appear previously in the path
transformation literature. However, it has been established in that literature (see
\cite{Jeulin}, \cite{Biane-Yor}, \cite[Theorem 2.1]{CSY}) that the two terms in
\eqref{eq_excursion_functional} have the same distribution. In particular, this
implies that the expectation of the random variable in
\eqref{eq_excursion_functional} indeed equals to $0$. It would be interesting to
find an independent proof of Corollary \ref{Cor_excursion_identity}, which does not
rely on the random matrix theory.

\medskip

Proposition \ref{Proposition_expectation} also gives a partial explanation for the special role that the value $\beta=2$ plays. Indeed, expanding the last exponential function in \eqref{eq_expectation_of_moment} into a power series we get
\begin{equation} \label{eq_expansion}
\begin{split}
\E \big[{\rm Trace}(\U(T))\big]
= &\sqrt{\frac2\pi}\,T^{-3/2}  - \frac{1}{2\sqrt{2\pi}}\,\E\bigg[
\int_0^1 \e(t)\,\mathrm{d}t - \frac{1}{\beta}\,\int_0^\infty (l_y)^2\,\mathrm{d}y \bigg] \\
& +\frac{1}{4\sqrt{2\pi}}\,T^{3/2}\,\E\bigg[\bigg(\int_0^1 \e(t)\,\mathrm{d}t - \frac{1}{\beta}\,\int_0^\infty (l_y)^2\,\mathrm{d}y\bigg)^2\bigg]+\cdots.
\end{split}
\end{equation}
In particular, $\beta=2$ is the only case in which the second term in the expansion \eqref{eq_expansion} vanishes.

\section{Combinatorics of high powers of tridiagonal matrices}\label{Section_combi}

In this section, we give a sketch of the proof of Theorem \ref{Theorem_Main} and, in particular, explain how the Brownian bridges $B^{x,y}$ and the Brownian motion $W$ in the definition of the kernel $K_\A(x,y;T)$ arise in the study of high powers of the matrix $M_N$. The technical estimates required to justify the steps of this sketch are then presented in Section \ref{Section_estimates} below, culminating in the complete proof of the theorem in Section \ref{Section_complete_proof}.

\medskip

Our aim is to study the matrix elements and the trace of a high power of the matrix $M_N$ and its principal submatrices. For the sake of a cleaner notation, we consider only the full matrix $M_N$. To study its submatrices one only needs to restrict the (scaled) indices to the corresponding set $\A$.

\medskip

By definition,
\begin{equation} \label{eq_element_as_sum}
(M_N)^k[i,i']=\sum M_N[i_0,i_1]\,M_N[i_1,i_2]\,\cdots \,M_N[i_{k-2},i_{k-1}]\,M_N[i_{k-1},i_{k}],
\end{equation}
where the sum is taken over all sequences of integers $i_0,i_1,\ldots,i_k$ in $\{1,2,\ldots,N\}$ such that $i_0=i$, $i_k=i'$, and $|i_j-i_{j-1}|\le 1$ for all $j=1,2,\ldots,k$. Hereby, the factors of the form $M_N[m,m+1]$ or $M_N[m,m-1]$ (given by $\bb(m)=\sqrt{m}+\xi(m)$) are given by $\sqrt{m}$ at the leading order in $m$, whereas factors of the form $M_N[m,m]$ (given by $\aa(m)$) are of order $1$ in $m$.

\medskip

We are interested in $\mathcal M(T,\A,N)$ and take first $k=\lfloor TN^{2/3}\rfloor$. Throughout the argument we assume that $k$ is even, with the odd case being very similar. Let us consider the sequences in \eqref{eq_element_as_sum} without ``horizontal'' segments
$i_{j-1}=i_j$. Note that we need to assume that $i-i'$ is even, as otherwise the sum is empty. With the notation $a\wedge b$ for $\min(a,b)$, the corresponding part of the sum in \eqref{eq_element_as_sum} is
\begin{equation}\label{eq_high_moment_base}
\big(2\sqrt{N}\big)^k \cdot\frac{1}{2^k}\sum_{\substack{1\le i_0,i_1,\ldots,i_{k} \le N \\ |i_j-i_{j-1}|=1\text{ for all } j\\
 i_0=i,\; i_{k}=i'}}
\prod_{l=1}^k \frac{\sqrt{i_l\wedge i_{l-1}}}{\sqrt{N}} \left(1+\frac{\xi( i_l\wedge i_{l-1})}{\sqrt{i_l\wedge i_{l-1}}}\right).
\end{equation}
The prefactor $\big(2\sqrt{N}\big)^k$ corresponds to the scaling under which the limiting spectral interval is $[-1,1]$. It is also precisely the normalization of $M_N$ used in the definition of $\mathcal M(T,\A,N)$, and we need to identify the $N\to\infty$ limit of the rest of the expression in \eqref{eq_high_moment_base}.

\medskip

Write $i^*$ for $\min(i_0,i_1,\ldots,i_k)$. It is not hard to see that the contribution of the sequences with $\frac{N-i^*}{N^{1/3}}\to\infty$ to the sum in \eqref{eq_high_moment_base} becomes negligible in the limit, so that one can restrict the attention to sequences with $\limsup_{N\to\infty} \frac{N-i^*}{N^{1/3}}<\infty$. In particular, we choose $i$, $i'$ such that
\begin{equation*}
x:=\lim_{N\to\infty} \frac{N-i}{N^{1/3}}<\infty,\quad
y:=\lim_{N\to\infty} \frac{N-i'}{N^{1/3}}<\infty.
\end{equation*}
Note further that we are summing over the trajectories of a simple random walk bridge with $k$ steps connecting $i$ to $i'$ (that is, a simple random walk with $k$ steps conditioned on having the prescribed endpoints). Our aim is to prove that the normalized sum converges to an integral with respect to the law of the Brownian bridge connecting $x$ to $y$.

\medskip

Each summand in \eqref{eq_high_moment_base} can be trivially rewritten as
\begin{equation}
\label{eq_summand_high_moment}
 \exp\bigg(\frac{1}{2}\,\sum_{l=1}^k \log\bigg(1-\frac{N-i_l\wedge i_{l-1}}{N}\bigg) +\sum_{l=1}^k \log\bigg(1+\frac{\xi( i_l\wedge
i_{l-1})}{\sqrt{i_l\wedge i_{l-1}}} \bigg)\bigg).
\end{equation}
For terms with $\limsup_{N\to\infty} \frac{N-i^*}{N^{1/3}}<\infty$, the arguments of the logarithms in the first sum are close to $1$, and one can use the formula $\log\,(1+z)\approx z$ (here, $z$ is of the order $N^{-2/3}$, and there are order $N^{2/3}$ summands). Similarly, for the logarithms in the second sum, consider the Taylor expansion
\begin{equation*}
\log\bigg(1+\frac{\xi( i_l\wedge i_{l-1})}{\sqrt{i_l\wedge i_{l-1}}}\bigg) = \frac{\xi( i_l\wedge i_{l-1})}{\sqrt{i_l\wedge i_{l-1}}}
- \frac{1}{2}\,\frac{\xi( i_l\wedge i_{l-1})^2}{i_l\wedge i_{l-1}} + \cdots
\end{equation*}
and note that already the second term is of order $N^{-1}$ (in expectation). Since there are order $N^{2/3}$ summands, only the first term can contribute to the $N\to\infty$ limit. Consequently, in that limit the expression from \eqref{eq_summand_high_moment} can be replaced by
\begin{equation}
\label{eq_summand_high_moment_2}
\exp\bigg(-\frac{1}{2N}\,\sum_{l=1}^k (N-i_l\wedge i_{l+1})
+ \sum_{l=1}^k \frac{\xi(i_l\wedge i_{l-1})}{\sqrt{i_l\wedge i_{l-1}}} \bigg).
\end{equation}

\smallskip

At this stage, we observe that
\begin{equation} \label{eq_summand_high_moment_3}
\sum_{l=1}^k \frac{\xi(i_l\wedge i_{l-1})}{\sqrt{i_l\wedge i_{l-1}}}=\sum_{h=i^*}^N \frac{\xi(h)}{\sqrt{h}}\,\big|\{l:\,i_l\wedge i_{l+1}=h\}\big|.
\end{equation}
A typical trajectory of a simple random walk bridge with $k$ steps connecting $i$ to $i'$ visits an order of $k^{1/2}$ sites, and the corresponding ``occupation times'' $\big|\{l:\,i_l\wedge i_{l-1}=h\}\big|$ are of the order $k^{1/2}$. Therefore, for every such trajectory, the sum on the right-hand side of \eqref{eq_summand_high_moment_3} is a sum of independent random variables with means of orders $o(h^{-2/3}k^{1/2})=o(N^{-1/3})$ and variances of orders $O(h^{-1}k)=O(N^{-1/3})$. Since there are an order of $N^{1/3}$ summands, the limit of the sum is given by the Central Limit Theorem. More specifically, the described random walk bridge converges in the limit $N\to\infty$ to a standard Brownian bridge on $[0,T]$ connecting $x$ to $y$, its occupation times normalized by $N^{1/3}$ converge to the local times of the Brownian bridge, and so the variance of the limiting centered Gaussian random variable comes out to $s_\xi^2\,\int_0^\infty L_a(B^{x,y})^2\,\mathrm{d}a$ (see Assumption \ref{Assumptions} (b) and \eqref{def:kernel} for the notations). That random variable can be written more explicitly as
\begin{equation*}
s_\xi \int_0^\infty L_a(B^{x,y})\,\mathrm{d}W_\xi(a),
\end{equation*}
where the Brownian motion $s_\xi\,W_\xi(a)$ is the limit of $N^{-1/6}\,\sum_{h=N-\lfloor N^{1/3} a\rfloor}^N \xi(h)$.

\medskip

Next, by a standard application of Stirling's formula we find that the number of random walk bridges of length $k=\lfloor TN^{2/3}\rfloor$ connecting $i$ to $i'$ behaves asymptotically as $2^k N^{-1/3}\sqrt{\frac{2}{\pi T}}\,e^{-(x-y)^2/(2T)}$. Since the expression in \eqref{eq_high_moment_base} can be viewed as a multiple of the expectation of a functional with respect to the law of such a random walk bridge, its asymptotic behavior is given by the same multiple of the corresponding functional of the Brownian bridge $B^{x,y}$:
\begin{equation*} \label{eq_high_power_base}
\begin{split}
& \big(2\sqrt{N}\big)^k N^{-1/3}\sqrt{\frac{2}{\pi T}}\,e^{-(x-y)^2/(2T)}\\
& \cdot\E_{B^{x,y}}\bigg[{\mathbf
1}_{\{\forall t:\,B^{x,y}(t) \ge 0\}}
\exp\bigg( -\frac12 \int_0^T B^{x,y}(t)\,\mathrm{d}t
+ s_\xi\,\int_0^\infty L_a(B^{x,y})\,\mathrm{d}W(a)\bigg)\bigg].
\end{split}
\end{equation*}

\smallskip

Next, we turn to the sequences in \eqref{eq_element_as_sum} which have
horizontal segments. We still work with an even $k$ and write $2n$ for the number of horizontal segments. The corresponding sequences can be thought of as follows: take a sequence of length $k-2n$ with no horizontal segments and insert $2n$ horizontal segments at arbitrary spots.
To analyze the effect of such insertion we start with the case $n=1$. The corresponding part of the sum in \eqref{eq_element_as_sum}, normalized by
$\big(2\sqrt{N}\big)^k$, is given by
\begin{equation*}
\frac{1}{2^{k-2}}\sum_{\substack{1\le i_0,i_1,\ldots,i_{k-2} \le N \\ |i_j-i_{j-1}|=1\text{ for all } j \\
 i_0=i,\; i_{k-2}=i'}}
\prod_{l=1}^{k-2} \frac{\sqrt{i_l\wedge i_{l-1}}}{\sqrt{N}} \left(1+\frac{\xi( i_l\wedge i_{l-1})}{\sqrt{i_l\wedge i_{l-1}}}\right)
\cdot\bigg(\frac{1}{(2\sqrt{N})^2} \sum_{0\le j\le l\le k-2} \aa(i_j)\,\aa(i_l)\bigg).
\end{equation*}

\smallskip

Note that the last factor can be written as the sum of the terms
$\frac{1}{2}\,\frac{1}{(2\sqrt{N})^2}\,\big(\sum_{j=0}^{k-2} \aa(i_j)\big)^2$ and $\frac{1}{2}\,\frac{1}{(2\sqrt{N})^2}\,\sum_{j=0}^{k-2} \aa(i_j)^2$. An analysis as for the left-hand side of \eqref{eq_summand_high_moment_3} shows that the first term tends to $\frac{1}{2}$ times the square of a Gaussian random variable with mean $0$ and variance $\frac{s_\aa^2}{4}\,\int_0^\infty L_a(B^{x,y})^2\,\mathrm{d}a$, which we write as
\begin{equation*}
\frac{s_\aa}{2}\,\int_0^\infty L_a(B^{x,y})\,\mathrm{d}W_\aa(a).
\end{equation*}
Here, the Brownian motion $s_\aa\,W_\aa(a)$ is the limit of $N^{-1/6}\,\sum_{h=N-\lfloor N^{1/3} a\rfloor}^N \aa(h)$. The second term is of order $O(N^{-1}k)=O(N^{-1/3})$ (in expectation), and so negligible in the limit $N\to\infty$.


\medskip

Similarly, for any number $2n$ of horizontal segments, their leading order contribution is a factor of
\begin{equation*}
\frac{1}{(2n)!\,(2\sqrt{N})^{2n}}\bigg(\sum_{j=0}^{k-2n} \aa(i_j) \bigg)^{2n}\sim \frac{1}{(2n)!}\,\bigg(\frac{s_\aa}{2}\,\int_0^\infty L_a(B^{x,y})\,\mathrm{d}W_\aa(a)\bigg)^{2n}.
\end{equation*}
At this point, we add $(M_N)^{k-1}[i,i']$ (as in the definition of $\mathcal M(T,\A,N)$) and recall that $k-1$ is odd. Treating $(M_N)^{k-1}[i,i']$ exactly in the same way as $(M_N)^k[i,i']$, only changing the even number of horizontal segments to an odd number $2n-1$ of horizontal segments, we conclude that the leading order contributions of the latter combine to
\begin{equation*}
\begin{split}
& \sum_{n=0}^\infty \frac{1}{(2n)!}\bigg(\frac{s_\aa}{2}\,\int_0^\infty L_a(B^{x,y})\,\mathrm{d}W_\aa(a)\bigg)^{2n}
+\sum_{n=1}^\infty \frac{1}{(2n-1)!}\bigg(\frac{s_\aa}{2}\,\int_0^\infty L_a(B^{x,y})\,\mathrm{d}W_\aa(a)\bigg)^{2n-1} \\
&=\exp\bigg(\frac{s_\aa}{2}\,\int_0^\infty L_a(B^{x,y})\,\mathrm{d}W_\aa(a)\bigg).
\end{split}
\end{equation*}

\smallskip

Putting everything together we obtain the asymptotics
\begin{equation} \label{eq_asymptotic_expansion_main}
\begin{split}
&\frac{1}{2}\left(\left(\frac{M_N}{2\sqrt N}\right)^k+ \left(\frac{M_N}{2\sqrt
N}\right)^{k-1} \right)[i,i'] \sim N^{-1/3}\sqrt{\frac{1}{2\pi T}}\,
e^{-\frac{(x-y)^2}{2T}} \\
&\E_{B^{x,y}}\bigg[{\mathbf 1}_{\{
\forall t:B^{x,y}(t)\ge 0\}}\exp\bigg(-\frac12 \int_0^T B^{x,y}(t)\mathrm{d}t + \int_0^\infty L_a(B^{x,y})\Big(s_\xi\mathrm{d}W_\xi(a)
+\frac{s_\aa}{2}\mathrm{d}W_\aa(a)\Big)\bigg)\bigg].
\end{split}
\end{equation}
To complete the derivation of the kernel $K(x,y;T)$ it remains to recall from Assumption \ref{Assumptions} that $s_\xi^2+\frac{s_\aa^2}{4}=\frac{1}{\beta}$ and to set $W=\sqrt{\beta}\,\big(s_\xi\,W_\xi+\frac{s_\aa}{2}\,W_\aa\big)$. Theorem \ref{Theorem_Main} now follows by summing  \eqref{eq_asymptotic_expansion_main} over the relevant indices $i$,
$i'$ and replacing the sums by the integrals they approximate. Finally,
for odd $i-i'$ and for general $\A\subset\rr_{\ge0}$, one only needs to interchange the roles of even and odd powers of $M_N$ and to consider sequences of $i_j$'s with $\frac{N-i_j+1/2}{N^{1/3}}\in\A$, respectively.

\section{Towards a rigorous proof} \label{Section_estimates}

\subsection{Convergence of random walk bridges and their local times}

In this subsection, we present the main technical ingredients needed to justify the arguments of Section \ref{Section_combi}.

\medskip

We use the following set of notations. Let $x,y\in\rr$, $N\in\nn$, and $\tilde{T}>0$ be such that $\tilde{T}N^{2/3}$ is an integer of the same parity as $\lfloor N-N^{1/3} x\rfloor-\lfloor N-N^{1/3} y\rfloor$. We write
\begin{equation*}
X^{x,y;N,\tilde{T}}:=\big(X^{x,y;N,\tilde{T}}(0),\,X^{x,y;N,\tilde{T}}(N^{-2/3}),\,X^{x,y;N,\tilde{T}}(2N^{-2/3}),\,\ldots,\,X^{x,y;N,\tilde{T}}(\tilde{T})\big)
\end{equation*}
for the simple random walk bridge connecting $\lfloor N-N^{1/3} x\rfloor$ to $\lfloor N-N^{1/3} y\rfloor$ in $\tilde{T}N^{2/3}$ steps of size $\pm 1$. More specifically, the trajectory of $X^{x,y;N,\tilde{T}}$ is chosen uniformly at random among all trajectories of the type described. We further define $X^{x,y;N,\tilde{T}}(t)$ for all $t\in[0,\tilde{T}]$ by linear interpolation. Finally, for $h\in\rr$, we  introduce the (normalized) occupation times
\begin{equation}
L_h(X^{x,y;N,\tilde{T}})=N^{-1/3}\big|\{t\in[0,\tilde{T}]:\,
X^{x,y;N,\tilde{T}}(t)=N-N^{1/3}h\}\big|,\quad h\in\rr.
\end{equation}

\medskip

The following proposition provides a coupling of the quantities $L_h(X^{x,y;N,\tilde{T}})$, $h\in\rr$ with their continuous analogues, the local times of a Brownian bridge.

\begin{proposition}\label{lemma_coupling}
Fix $x,y\in\rr$. Let $T_N$, $N\in\nn$ be a sequence of reals such that $\sup_N |T_N-T|\,N^{2/3}<\infty$ for some $T>0$, $T_N\,N^{2/3}\in\nn$ for all $N$, and $T_N\,N^{2/3}$ has the same parity as $\lfloor N-N^{1/3} x\rfloor-\lfloor N-N^{1/3} y\rfloor$. Then, there exists a probability space supporting a sequence of random walk bridges $X^{x,y;N,T_N}$, $N\in\nn$, a standard Brownian bridge $B^{x,y}$ on $[0,T]$ connecting $x$ to $y$, and a real random variable $\mathcal C$ such that
\begin{eqnarray}
&& \sup_{h\in\rr} \Big|L_h\big(X^{x,y;N,T_N}\big)-L_h\big(B^{x,y}\big)\Big| \le \mathcal C\,N^{-1/16},\quad N\in\nn \label{XtoB0}\\
&& \sup_{0\le t\le T_N\wedge T} \, \Big|N^{-1/3}(N-X^{x,y;N,T_N}(t))-B^{x,y}(t)\Big|\le \mathcal C\,N^{-1/3}\,\log N ,\quad N\in\nn \label{XtoB}.
\end{eqnarray}
\end{proposition}

Strong invariance principles similar to Proposition \ref{lemma_coupling} can be found in the literature, cf. \cite[Remark 1.3]{Bo}, \cite[Theorem 2.2.4]{CR} for the case of random walks and \cite[Theorem 2]{Kh} for the case of bridges of compensated Poisson processes. The proof of Proposition \ref{lemma_coupling} is given in Appendix \ref{Section_appendix_coupling} and is based on a coupling constructed in \cite{LTF} and a lemma from \cite{BK2}.

\medskip

Proposition \ref{lemma_coupling} describes the typical behavior of the random walk bridges $X^{x,y;N,T_N}$, $N\in\nn$ and the associated occupation times $L_h\big(X^{x,y;N,T_N}\big)$, $h\in\rr$, $N\in\nn$.
Next, we complement it by some estimates on the tail behavior of these quantities. To this end, for each $x,y\in\rr$, $N\in\nn$, and $T_0>0$, we write ${\mathcal T}(x,y;N,T_0)$ for the set of $\tilde{T}\in[0,T_0)$ such that $\tilde{T}N^{2/3}$ is an integer of the same parity as $\lfloor N-N^{1/3} x\rfloor-\lfloor N-N^{1/3} y\rfloor$.

\begin{proposition}\label{Proposition_exponential_bound_1}
For every $T_0>0$ and $\theta\in\rr$, one has the uniform integrability estimate
\begin{equation}
\sup_{N\in\nn} \; \sup_{\tilde{T}\in {\mathcal T}(x,y;N,T_0)} \;
\E\bigg[\exp\bigg(\theta\,N^{-2/3} \sum_{i=0}^{\tilde{T}N^{2/3}}
N^{-1/3} \Big(N-X^{x,y;N,\tilde{T}}\big(iN^{-2/3}\big)\Big)\bigg)\bigg] <\infty \label{UImax}
\end{equation}
uniformly on compact sets in $x,y$.
\end{proposition}

\begin{proposition}\label{Proposition_exponential_bound_2}
For every $T_0>0$, $1\le p<3$, and $\theta>0$, one has the uniform integrability estimate
\begin{equation}
\sup_{N\in\nn} \; \sup_{\tilde{T}\in {\mathcal T}(x,y;N,T_0)} \;
\E\bigg[\exp\bigg(\theta\,N^{-1/3} \sum_{h\in N^{-1/3}\zz}
L_h\big(X^{x,y;N,\tilde{T}}\big)^p\,\bigg)\bigg] < \infty\, \label{UIlt}
\end{equation}
uniformly on compact sets in $x$, $y$.
\end{proposition}

\begin{rmk}\label{rmk_loc_time}
The trivial estimate
\begin{equation*}
L_h\big(X^{x,y;N,\tilde{T}}\big) \le L_{N^{2/3}-N^{-1/3}\lceil N-N^{1/3}h\rceil}\big(X^{x,y;N,\tilde{T}}\big)+L_{N^{2/3}-N^{-1/3}\lfloor N-N^{1/3}h\rfloor}\big(X^{x,y;N,\tilde{T}}\big)
\end{equation*}
for $h\in\rr$ shows that \eqref{UIlt} continues to hold if one replaces the summation over $h\in N^{-1/3}\zz$ by the summation over $h\in c+N^{-1/3} \zz$ for some $c\in\rr$.
\end{rmk}

\begin{proof}[Proof of Proposition \ref{Proposition_exponential_bound_1}]
Define the random walk bridge
\begin{equation*}
\tilde{X}^{x,y;N,\tilde{T}}(t):=N^{-1/3}\big(N-X^{x,y;N,\tilde{T}}(t)\big),\quad t\in[0,\tilde{T}]
\end{equation*}
and note that its endpoints $x_N$, $y_N$ satisfy $|x_N-x|\le N^{-1/3}$, $|y_N-y|\le N^{-1/3}$. Moreover, due to the symmetry properties of random walk bridges, we may assume without loss of generality that $\theta>0$ and $x_N\le y_N$. Observe further that the random variable inside
the expectation in \eqref{UImax} can be bounded from above by $e^{\theta\,(T_0+N^{-2/3})\,\tilde{M}(N,\tilde{T})}$ where $\tilde{M}(N,\tilde{T})$ is the maximum of $\tilde{X}^{x,y;N,\tilde{T}}$, that is, $\tilde{M}(N,\tilde{T})=\max_{t\in[0,\tilde{T}]} \tilde{X}^{x,y;N,\tilde{T}}(t)$.

\medskip

In order to upper bound the exponential moments of $\tilde{M}(N,\tilde{T})$, we introduce an auxiliary Markov chain $Y(t)$, $t=0,\,N^{-2/3},\,2N^{-2/3},\,\ldots,\,\tilde{T}$ (on an extension of the probability space on which $\tilde{X}^{x,y;N,\tilde{T}}$ is defined) by the following conditions:
\begin{enumerate}[(a)]
\item $Y(0)=y_N$;
\item the increments $Y(t+N^{-2/3})-Y(t)$ belong to the set $\{-N^{-1/3},0,N^{-1/3}\}$;
\item if $\tilde{X}^{x,y;N,\tilde{T}}(t)<y_N$, then $Y(t+N^{-2/3})=Y(t)$, otherwise $Y(t+N^{-2/3})-Y(t)\in\{-N^{-1/3},N^{-1/3}\}$;
\item in the latter case, $Y(t+N^{-2/3})-Y(t)$, conditional on $Y(0),\,\ldots,\,Y(t)$ and $\tilde{X}^{x,y;N,\tilde{T}}(0),\,\ldots,\,\tilde{X}^{x,y;N,\tilde{T}}(t)$, equals to $-N^{-1/3}$ or $N^{-1/3}$ with probability $\frac{1}{2}$ each;
\item if $\tilde{X}^{x,y;N,\tilde{T}}(t+N^{-2/3})-\tilde{X}^{x,y;N,\tilde{T}}(t)=N^{-1/3}$ and $\tilde{X}^{x,y;N,\tilde{T}}(t)\ge y_N$, then
$Y(t+N^{-2/3})-Y(t)=N^{-1/3}$.
\end{enumerate}
In words: $Y$ is a symmetric random walk which moves only when
$\tilde{X}^{x,y;N,\tilde{T}}(t)\ge y_N$ and, in that case, moves up whenever $\tilde{X}^{x,y;N,\tilde{T}}$ moves up. The existence of $Y$ follows from the observation that $\tilde{X}^{x,y;N,\tilde{T}}(t)$, $t=0,\,N^{-2/3},\,2N^{-2/3},\,\ldots,\,\tilde{T}$ is a Markov chain whose probabilities of upward jumps at sites above $y_N$ are less than $\frac{1}{2}$ (see e.g. \cite[Lemma 1.1]{LSS} and \cite[Lemma 1]{Russo} for similar statements).

\medskip

The definition of $Y$ implies $Y(t)\ge X(t)$, $t=0,\,N^{-2/3},\,2N^{-2/3},\,\ldots,\,\tilde{T}$,
and so the maximum of $Y$ is greater or equal than $\tilde{M}(N,\tilde{T})$. Since the non-zero
increments of $Y$ are those of a simple symmetric random walk, the desired estimate follows from
the corresponding statement for the latter. Indeed, \cite[Theorem 6.2.1, inequality (6.2.3)]{Ch}
reveals that the maximum $J_{\tilde{T}N^{2/3}}$ of a simple symmetric random walk with
$\tilde{T}N^{2/3}$ steps satisfies
\begin{equation} \label{eq_SRW_max_bound}
\E\big[(J_{\tilde{T}N^{2/3}})^n\big] \le \sqrt{n!}\,\big(C\,\tilde{T}^{1/2}N^{1/3}\big)^n,\quad n\in\nn,\;\tilde{T}\in{\mathcal T}(x,y;N,\infty),\;N\in\nn
\end{equation}
with some constant $C<\infty$. The desired uniform bound on the corresponding exponential moments
now follows from the series expansion of the exponential function.
\end{proof}

\smallskip

Our proof of Proposition \ref{Proposition_exponential_bound_2} relies on an identity in distribution from \cite{AFP} relating the maximal occupation time of $X^{x,y;N,\tilde{T}}$ to the range of $X^{x,y;N,\tilde{T}}$. The latter builds on the concepts of quantile and Vervaat transforms from \cite{AFP} and \cite{Ver}, respectively.

\begin{definition}[Quantile transform]
For every realization of the random walk bridge $X^{x,y;N,\tilde{T}}$, let $\kappa$ be the unique permutation of $\{1,\,2,\,\ldots,\,\tilde{T}N^{2/3}\}$ such that, for all $1\le l_1<l_2\le\tilde{T}N^{2/3}$, either
$X^{x,y;N,\tilde{T}}\big((\kappa(l_1)-1)N^{-2/3}\big)<X^{x,y;N,\tilde{T}}\big((\kappa(l_2)-1)N^{-2/3}\big)$, or
$X^{x,y;N,\tilde{T}}\big((\kappa(l_1)-1)N^{-2/3}\big)=X^{x,y;N,\tilde{T}}\big((\kappa(l_2)-1)N^{-2/3}\big)$ and $\kappa(l_1)<\kappa(l_2)$. Then, the quantile transform $Q^{N,\tilde{T}}$ of $X^{x,y;N,\tilde{T}}$ is
defined by $Q^{N,\tilde{T}}(0)=0$,
\begin{equation}
\begin{split}
Q^{N,\tilde{T}}\big(lN^{-2/3}\big)=\sum_{l_1=1}^l
\Big(X^{x,y;N,\tilde{T}}\big(\kappa(l_1)N^{-2/3}\big)-X^{x,y;N,\tilde{T}}\big((\kappa(l_1)-1)N^{-2/3}\big)\Big), \\
l=1,\,2,\,\ldots,\,\tilde{T}N^{2/3}.
\end{split}
\end{equation}
In words: the quantile transform $Q^{N,\tilde{T}}$ is obtained by starting at $0$ and consecutively adding the increments of $X^{x,y;N,\tilde{T}}$ in the non-decreasing order of the sites they originate from in $X^{x,y;N,\tilde{T}}$, with the increments originating from the same site in $X^{x,y;N,\tilde{T}}$ being added in chronological order. We refer to \cite[Figure 1]{AFP} for an illustration.
\end{definition}

\begin{definition}
For every realization of a random walk bridge $X^{x,y;N,\tilde{T}}$, let
\begin{equation}
l^*=\min\big\{l_1\in\{1,\,2,\,\ldots,\,\tilde{T}N^{2/3}\}:\;X^{x,y;N,\tilde{T}}(l_2 N^{-2/3})\ge X^{x,y;N,\tilde{T}}(l_1 N^{-2/3})\;\;\text{for all }\;l_2\big\}
\end{equation}
and write $\vartheta$ for the permutation of $\{1,2,\ldots,\tilde{T}N^{2/3}\}$ mapping each element $l_1$ to $l_1+l^*\,\mod\,\tilde{T}N^{2/3}$. Then, the Vervaat transform $V^{N,\tilde{T}}$ of $X^{x,y;N,\tilde{T}}$ is defined by $V^{N,\tilde{T}}(0)=0$,
\begin{equation}
\begin{split}
V^{N,\tilde{T}}\big(lN^{-2/3}\big)=\sum_{l_1=1}^l
\Big(X^{x,y;N,\tilde{T}}\big(\vartheta(l_1)N^{-2/3}\big)-X^{x,y;N,\tilde{T}}\big((\vartheta(l_1)-1)N^{-2/3}\big)\Big), \\
l=1,\,2,\,\ldots,\,\tilde{T}N^{2/3}.
\end{split}
\end{equation}
In words: the Vervaat transform $V^{N,\tilde{T}}$ is obtained from $X^{x,y;N,\tilde{T}}$ by splitting the path of $X^{x,y;N,\tilde{T}}$ at its first global minimum, attaching the first part of the path to the endpoint of the second part of the path, and shifting the resulting path so that it starts at $0$. We refer to \cite[Figure 17]{AFP} for an illustration.
\end{definition}

\begin{proof}[Proof of Proposition \ref{Proposition_exponential_bound_2}]
For each $h\in N^{-1/3}\zz$, we write $u^{N,\tilde{T}}_h$ and $d^{N,\tilde{T}}_h$ for the numbers of up and down steps of $X^{x,y;N,\tilde{T}}$ originating from $N-N^{1/3}h$, respectively, and define
\begin{equation}\label{whatist}
t^{N,\tilde{T}}_h=\sum_{N^{-1/3}\zz\,\ni h_1>h} \big(u^{N,\tilde{T}}_{h_1}+d^{N,\tilde{T}}_{h_1}\big).
\end{equation}
Using these and the elementary inequality
\begin{equation} \label{eq_powers_ineq}
(a+b)^c \le 2^c\,\big(a^c+b^c\big),\quad a,\,b,\,c\ge 0
\end{equation}
with $c=p-1$ we obtain
\begin{equation} \label{eq_occup_time_transform}
\begin{split}
& \sum_{h\in N^{-1/3}\zz} L_h\big(X^{x,y;N,\tilde{T}}\big)^p
=N^{-p/3}\,\sum_{h\in N^{-1/3}\zz} \big(u^{N,\tilde{T}}_h+d^{N,\tilde{T}}_h\big)^p \\
& \le N^{-p/3}\,2^{p-1}\,\sum_{h\in N^{-1/3}\zz}
\Big(\big(u^{N,\tilde{T}}_h+d^{N,\tilde{T}}_h\big)\,\big(u^{N,\tilde{T}}_h\big)^{p-1}+\big(u^{N,\tilde{T}}_h+d^{N,\tilde{T}}_h\big)\,\big(d^{N,\tilde{T}}_h\big)^{p-1}\Big).
\end{split}
\end{equation}

\smallskip

Next, we employ the following combinatorial identity (see \cite[equation (5.3)]{AFP}):
\begin{equation*}
\begin{split}
Q^{N,\tilde{T}}(t^{N,\tilde{T}}_{h-N^{-1/3}}N^{-2/3})
=u^{N,\tilde{T}}_h+\big(N-N^{1/3}h-\lfloor N-N^{1/3}x\rfloor\big)_+
\qquad\qquad\qquad\;\; \\
-\big(N-N^{1/3}h-\lfloor N-N^{1/3}y\rfloor\big)_+,\quad h\in N^{-1/3}\zz
\end{split}
\end{equation*}
(note that in \cite{AFP} the height of the bridge is computed with respect to the minimum of the bridge rather than zero, but this is not important). The latter gives the bound
\begin{equation}\label{ubound}
u^{N,\tilde{T}}_h \le Q^{N,\tilde{T}}(t^{N,\tilde{T}}_{h-N^{-1/3}}N^{-2/3})+N^{1/3}|x-y|+1,\quad h\in N^{-1/3}\zz.
\end{equation}
Moreover, by combinatorial constraints $|d^{N,\tilde{T}}_{h-N^{-1/3}}-u^{N,\tilde{T}}_h|\le 1$ for all $h\in N^{-1/3}\zz$, so
\begin{equation}\label{dbound}
d^{N,\tilde{T}}_h\le Q^{N,\tilde{T}}\big(t^{N,\tilde{T}}_h N^{-2/3}\big)+N^{1/3}|x-y|+2,\quad h\in N^{-1/3}\zz.
\end{equation}
Combining \eqref{eq_occup_time_transform} with $u^{N,\tilde{T}}_h+d^{N,\tilde{T}}_h=t^{N,\tilde{T}}_{h-N^{-1/3}}-t^{N,\tilde{T}}_h$, $h\in N^{-1/3}\zz$ and \eqref{ubound}, \eqref{dbound} we find
\begin{equation*}
\begin{split}
& \sum_{h\in N^{-1/3}\zz} L_h\big(X^{x,y;N,\tilde{T}}\big)^p \\
& \le N^{-p/3}\,2^{p-1}\,\sum_{h\in N^{-1/3}\zz}
\big(t^{N,\tilde{T}}_{h-N^{-1/3}}-t^{N,\tilde{T}}_h\big)\,\Big(Q^{N,\tilde{T}}\big(t^{N,\tilde{T}}_{h-N^{-1/3}}N^{-2/3}\big)+N^{1/3}|x-y|+1\Big)^{p-1} \\
&\quad\; + N^{-p/3}\,2^{p-1}\,\sum_{h\in N^{-1/3}\zz}
\big(t^{N,\tilde{T}}_{h-N^{-1/3}}-t^{N,\tilde{T}}_h\big)\,\Big(Q^{N,\tilde{T}}\big(t^{N,\tilde{T}}_h N^{-2/3}\big)+N^{1/3}|x-y|+2\Big)^{p-1}.
\end{split}
\end{equation*}
Estimating the values of $Q^{N,\tilde{T}}$ by $N^{1/3}$ times the maximal value $M(N,\tilde{T})$ taken by $N^{-1/3}Q^{N,\tilde{T}}$, noting $\sum_{h\in N^{-1/3}\zz} \big(t^{N,\tilde{T}}_{h-N^{-1/3}}-t^{N,\tilde{T}}_h\big)=\tilde{T} N^{2/3}$, and applying \eqref{eq_powers_ineq} with $c=p-1$ we get
\begin{equation}\label{SILTubd}
N^{-1/3}\sum_{h\in N^{-1/3}\zz} L_h\big(X^{x,y;N,\tilde{T}}\big)^p
\le 2^{2p-1}\,\tilde{T}\,\Big(M(N,\tilde{T})^{p-1}+\big(|x-y|+2N^{-1/3}\big)^{p-1}\Big).
\end{equation}
It remains to bound the exponential moments of the right-hand side of \eqref{SILTubd}. To this end, we use the fact that the distribution of $M(N,\tilde{T})$ is the same as the distribution of the maximum of the normalized Vervaat transform $N^{-1/3}V^{N,\tilde{T}}$ (see \cite[Corollary 7.4, equation (7.3)]{AFP}). However, by the definition of $V^{N,\tilde{T}}$, its maximum equals to the width of the range of $X^{x,y;N,\tilde{T}}$, that is, the difference between the maximum and the minimum of $X^{x,y;N,\tilde{T}}$. Therefore, by arguing as in the proof of Proposition \ref{Proposition_exponential_bound_1} we can reduce the problem at hand to that of bounding the exponential moments
of the $(p-1)$-st power of such width for a simple symmetric random walk with $\tilde{T}N^{2/3}$ steps, normalized by $N^{-1/3}$. For $1\le p<3$, these moments are bounded uniformly in $N\in\nn$, $\tilde{T}\in{\mathcal T}(x,y;N,T_0)$, as can be seen easily from \eqref{eq_SRW_max_bound}.
\end{proof}

\smallskip

We continue with some auxiliary convergence results that serve as building blocks in the proof of Theorem \ref{Theorem_Main}. Consider a sequence of independent random variables $\zeta(m)$, $m\in\nn$ which satisfies the following four conditions:
\begin{eqnarray}
&& \forall\,z>0:\quad \lim_{N\to\infty} N^{-1/6} \sum_{m=N-\lfloor N^{1/3} z \rfloor}^N \big|\E[\zeta(m)]\big| = 0, \label{eq_zeta_condition_1} \\
&& \exists\,s_\zeta\ge0:\quad \lim_{N\to\infty} N^{-1/3} \sum_{m=N-\lfloor N^{1/3} z \rfloor}^N \E\big[\zeta(m)^2\big] =s_\zeta^2\,z,\quad z>0, \label{eq_zeta_condition_2} \\
&& \forall z>0:\quad \lim_{N\to\infty} N^{-1/2} \sum_{m=N-\lfloor N^{1/3} z \rfloor}^{N} \E[|\zeta(m)|^3] = 0, \label{eq_zeta_condition_3} \\
&& \exists\,C>0,\;\;0<\gamma<2/3:\quad \E\big[|\zeta(m)|^\ell\big]\le C^\ell \ell^{\gamma\ell},\quad m,\ell\in\nn. \label{eq_zeta_tail_condition}
\end{eqnarray}
Note that \eqref{eq_zeta_condition_1}-\eqref{eq_zeta_tail_condition} automatically hold for the choices $\zeta(m)=\aa(m)$, $m\in\nn$ and $\zeta(m)=\xi(m)$, $m\in\nn$ by Assumption \ref{Assumptions}. In other words, \eqref{eq_zeta_condition_1}-\eqref{eq_zeta_tail_condition} are somewhat weaker than conditions (a)-(c) of Assumption \ref{Assumptions}. Before proceeding further let us observe the following simple consequences of \eqref{eq_zeta_tail_condition}.

\begin{lemma} \label{Lemma_tail_estimates} Suppose that for a random variable $\zeta$ there are constants $C>0$ and $0<\gamma<2/3$ such that $\E\big[|\zeta|^\ell\big]\le C^\ell \ell^{\gamma\ell}$ for all $\ell\in\nn$. Then, there exist $C'>0$ and $2<\gamma'<3$ depending
only on $C$ and $\gamma$ such that
\begin{eqnarray}
&& \quad \E\big[e^{v\zeta}\big] \le \exp\Big(v\,\E[\zeta]+C'\big(v^2+|v|^{\gamma'}\big)\Big),\;\;v\in\rr, \label{new_tail_bound1} \\
&& \quad \E\big[|1+v\zeta|^\ell\big] \le \exp\Big( |v|\,\ell\,|\E[\zeta]|+C'\big(v^2\ell^2+|v|^{\gamma'}\ell^{\gamma'}\big)\Big),\;\;v\in\rr,\;\ell\in\nn. \label{new_tail_bound2}
\end{eqnarray}
\end{lemma}

\begin{proof}
The proof of \eqref{new_tail_bound1} follows closely that of \cite[Lemma 5.5, implication 2.\,$\Rightarrow$\,4.]{Ve}. We make the following assumptions without loss of generality: $\gamma>1/2$; $\E[\zeta]=0$  (otherwise we can apply \eqref{new_tail_bound1} for $\zeta-\E[\zeta]$ which satisfies the same moment bounds as $\zeta$, just with $C$ replaced by $2C$); and $C=1$ (since we can use \eqref{new_tail_bound1} for $\zeta/C$ when $C>1$). Under these assumptions, we expand the exponential function on the left-hand side of \eqref{new_tail_bound1} and use the moment bounds for $\zeta$ together with the estimate $\ell!\ge(\ell/e)^\ell$ to get
\begin{equation}\label{LHSseries}
\E\big[e^{v\zeta}\big]\le 1+\sum_{\ell=2}^\infty \bigg(\frac{e|v|}{\ell^{1-\gamma}}\bigg)^\ell.
\end{equation}
On the other hand, an expansion of the exponential function on the right-hand side of \eqref{new_tail_bound1} and the bound $\hat{\ell}!\le\hat{\ell}^{\hat{\ell}}$ yield
\begin{equation}\label{RHSseries}
\exp\Big(C'\big(v^2+|v|^{\gamma'}\big)\Big)
\ge 1+\sum_{\hat{\ell}=1}^\infty \bigg(\frac{C'\big(v^2+|v|^{\gamma'}\big)}{\hat{\ell}}\bigg)^{\hat{\ell}}.
\end{equation}
Next, for every $\ell\ge\lceil \frac{3}{2-3\gamma} \rceil=:\ell_*$, we consider $\hat{\ell}:=\lfloor (1-\gamma)\ell \rfloor$ and note that each value of $\hat{\ell}$ arises for at most three different values of $\ell$ thanks to $1-\gamma>1/3$. We pick $\frac{\ell_*}{(1-\gamma)\ell_*-1}<\gamma'<3$ arbitrarily and observe that, in particular, $\gamma'>2$ due to $\gamma>1/2$. Moreover, with $\alpha_\ell:=\frac{\gamma'\hat{\ell}-\ell}{(\gamma'-2)\hat{\ell}}\in(0,1)$ (recall the lower bound on $\gamma'$, as well as $\gamma>1/2$) and for any $C'\ge1$, we have the estimates
\begin{equation*}
\begin{split}
\bigg(\frac{C'\big(v^2+|v|^{\gamma'}\big)}{\hat{\ell}}\bigg)^{\hat{\ell}}
& \ge \frac{(C')^{\hat{\ell}}}{(1-\gamma)^{(1-\gamma)\ell}\,\ell^{(1-\gamma)\ell}}\big(\alpha_\ell\,|v|^{2\hat{\ell}}+(1-\alpha_\ell)\,|v|^{\gamma'\hat{\ell}}\big) \\
& \ge \frac{(C')^{(1-\gamma)\ell-1}}{(1-\gamma)^{(1-\gamma)\ell}\,\ell^{(1-\gamma)\ell}}\,|v|^{2\hat{\ell}\alpha_\ell+\gamma'\hat{\ell}(1-\alpha_\ell)}
= \bigg(\frac{(C')^{1-\gamma-1/\ell}|v|}{(1-\gamma)^{1-\gamma}\ell^{1-\gamma}}\bigg)^\ell.
\end{split}
\end{equation*}
In particular, we see that, as soon as $\frac{(C')^{1-\gamma-1/\ell_*}}{(1-\gamma)^{1-\gamma}}>e$, the $\hat{\ell}$-th term of the sum in \eqref{RHSseries} is greater than the $\ell$-th term of the sum in \eqref{LHSseries} for all $\ell\ge\ell_*$. Therefore, by increasing the value of $C'$ sufficiently, we can make sure that \eqref{new_tail_bound1} holds.

\medskip

We now turn to the proof of \eqref{new_tail_bound2}. We may again assume without
loss of generality that $\E[\zeta]=0$, since otherwise we can use the Binomial
Theorem and the bound \eqref{new_tail_bound2} for $\bar{\zeta}:=\zeta-\E[\zeta]$ to
derive the chain of estimates
\begin{equation*}
\begin{split}
& \E\big[|1+v\zeta|^\ell\big] \le \sum_{{\ell}_1=0}^\ell \binom{\ell}{\ell_1}\E\big[|1+v\bar{\zeta}|^{\ell-\ell_1}\big]\big(|v|\,|\E[\zeta]|\big)^{\ell_1} \\
& \le \sum_{{\ell}_1=0}^\ell \binom{\ell}{\ell_1}e^{C'(v^2(\ell-\ell_1)^2+|v|^{\gamma'}(\ell-\ell_1)^{\gamma'})}\big(|v|\,|\E[\zeta]|\big)^{\ell_1}
\le e^{C'(v^2\ell^2+|v|^{\gamma'}\ell^{\gamma'})}\big(1+|v|\,|\E[\zeta]|\big)^\ell,
\end{split}
\end{equation*}
where the last expression is less or equal to the right-hand side of \eqref{new_tail_bound2}. When $\E[\zeta]=0$, we may also assume without loss of generality that $\ell$ is even, since for odd $\ell=2\ell_1+1$ we can apply Jensen's inequality and the bound \eqref{new_tail_bound2} with $\ell=2\ell_1+2$ to find
\begin{equation*}
\begin{split}
\E\big[|1+v\zeta|^{2\ell_1+1}\big]\le \E\big[|1+v\zeta|^{2\ell_1+2}\big]^{\frac{2\ell_1+1}{2\ell_1+2}}
\le e^{C'(v^2(2\ell_1+2)(2\ell_1+1)+|v|^{\gamma'}(2\ell_1+2)^{\gamma'-1}(2\ell_1+1))} \\
\le \exp\Big(2^{\gamma'-1} C'\big(v^2(2\ell_1+1)^2+|v|^{\gamma'}(2\ell_1+1)^{\gamma'}\big)\Big).
\end{split}
\end{equation*}
In particular, it suffices to increase the constant $C'$ in \eqref{new_tail_bound2} for even $\ell$ by a factor of $2^{\gamma'-1}$ to obtain \eqref{new_tail_bound2} for all $\ell\in\nn$.

\medskip

For even $\ell$, we may drop the absolute value on the left-hand side of
\eqref{new_tail_bound2} and use the Binomial Theorem together with $\E[\zeta]=0$,
$\binom{\ell}{\ell_1}\le\frac{\ell^{\ell_1}}{\ell_1!}\le
\ell^{\ell_1}/(\frac{\ell_1}{e})^{\ell_1}$, and the moment bounds for $\zeta$ to get
\begin{equation*}
\E\big[|1+v\zeta|^\ell\big] \le 1+\sum_{\ell_1=2}^\ell \frac{\ell^{\ell_1}}{\big(\frac{\ell_1}{e}\big)^{\ell_1}}\,|v|^{\ell_1}C^{\ell_1}\ell_1^{\gamma\ell_1}
\le 1+\sum_{\ell_1=2}^\infty \bigg(\frac{e\ell|v|C}{\ell_1^{1-\gamma}}\bigg)^{\ell_1}.
\end{equation*}
The latter expression equals to the right-hand side of \eqref{LHSseries} with $|v|$ replaced by $\ell|v|C$, so that we can bound it from above by the right-hand side of \eqref{new_tail_bound1} with $v$ replaced by $\ell vC$. As a result, we obtain the estimate \eqref{new_tail_bound2} with the same $\gamma'$ as in \eqref{new_tail_bound1} once we suitably enlarge $C'$.
\end{proof}

We proceed to the first auxiliary convergence result.

\begin{proposition}\label{Proposition_basic_conv}
Let $x,y\in\rr$ and $T_N>0$, $N\in\nn$ be such that $\sup_N |T_N-T|N^{2/3}<\infty$ for some $T>0$, $T_NN^{2/3}\in\nn$ for all $N$, and $T_NN^{2/3}$ has the same parity as $\lfloor N-N^{1/3}x\rfloor-\lfloor N-N^{1/3}y\rfloor$. Then, under \eqref{eq_zeta_condition_1}-\eqref{eq_zeta_condition_3}, the sequence
\begin{equation}\label{eq_full_seq}
\bigg(N^{-1/3}\big(N-X^{x,y;N,T_N}\big),\sum_{h\in N^{-1/3}(\zz_{\ge0}+1/2)}
L_h\big(X^{x,y;N,T_N}\big)\,\frac{\zeta(\lfloor N-N^{1/3}h \rfloor)}{N^{1/6}}\bigg),
\quad N\in\nn
\end{equation}
converges in the limit $N\to\infty$ in distribution to
\begin{equation}\label{lim_rv}
\bigg(B^{x,y},\,s_\zeta\,\int_0^\infty L_a\big(B^{x,y}\big)\,\mathrm{d}W_\zeta(a)\bigg).
\end{equation}
Here, $W_\zeta$ is a standard Brownian motion, and the convergence of the first argument is with respect to the topology of uniform convergence on $[0,T]$ (if $T_N<T$, we extend $X^{x,y;N,T_N}$ to $[0,T]$ by extending its last linear segment).
\end{proposition}

\begin{proof}
We work throughout under the coupling of Proposition \ref{lemma_coupling} (note that we can enlarge the probability space there further to support copies of $\zeta(m)$, $m\in\nn$ and $W_\zeta$). For every fixed $N\in\nn$, we consider the \emph{conditional} characteristic function
\begin{equation} \label{eq_cond_char_funct}
\E_\zeta\bigg[\exp\bigg(\ii\,u\sum_{h\in N^{-1/3}(\zz_{\ge0}+1/2)}
L_h\big(X^{x,y;N,T_N}\big)\,\frac{\zeta\lfloor N-N^{1/3}h\rfloor}{N^{1/6}}\bigg)\bigg], \quad u\in\rr,
\end{equation}
where the expectation is taken with respect to the $\zeta(m)$, $m\in\nn$ (so that the randomness in the result is coming only from $X^{x,y;N,T_N}$). At this point, it suffices to show that, for each $u\in\rr$, the difference between
\eqref{eq_cond_char_funct} and the conditional characteristic function
\begin{equation} \label{eq_lim_cond_char}
\E_{W_\zeta}\bigg[\exp\bigg(\ii\,u\,s_\zeta\,\int_0^\infty
L_a\big(B^{x,y}\big)\,\mathrm{d}W_\zeta(a)\bigg)\bigg],\quad u\in\rr
\end{equation}
tends to $0$ almost surely (in \eqref{eq_lim_cond_char}, the expectation is taken with respect to $W_\zeta$, and the randomness in the result is coming only from $B^{x,y}$).
Indeed, the convergence in distribution of \eqref{eq_full_seq} is equivalent to the pointwise convergence of the characteristic functions
\begin{equation}\label{eq_cfconv}
\begin{split}
&\E\bigg[\!\exp\bigg(\!\ii rF\big(N^{-1/3}(N-X^{x,y;N,T_N})\big)\!+\!\ii u\!\!\!\sum_{h\in N^{-1/3}(\zz_{\ge0}+1/2)} \!\!\!\!
L_h\big(X^{x,y;N,T_N}\big)\frac{\zeta(\lfloor N-N^{1/3}h\rfloor)}{N^{1/6}}\bigg)\!\bigg] \\
&\longrightarrow\E\bigg[\exp\bigg(\ii\,r\,F(B^{x,y})+\ii\,u\,s_\zeta\,\int_0^\infty L_a\big(B^{x,y}\big)\,\mathrm{d}W_\zeta(a)\bigg)\bigg], \quad (r,u)\in\rr^2
\end{split}
\end{equation}
for all bounded continuous functionals $F$ on the space of continuous functions on $[0,T]$ endowed with the topology of uniform convergence. The latter convergence follows directly from the Dominated Convergence Theorem, \eqref{XtoB}, and the almost sure convergence to $0$ of the difference between \eqref{eq_cond_char_funct} and \eqref{eq_lim_cond_char}.

\medskip

The rest of the proof is devoted to showing that the difference between \eqref{eq_cond_char_funct}
and \eqref{eq_lim_cond_char} tends to $0$ almost surely in the limit $N\to\infty$. Given a path of
$B^{x,y}$, the random variable $s_\zeta\,\int_0^\infty L_a\big(B^{x,y}\big)\,\mathrm{d}W(a)$ is
simply a normal random variable with mean zero and variance $s_\zeta^2\,\int_0^\infty
L_a\big(B^{x,y}\big)^2\,\mathrm{d}a$. At the same time, given a path of $X^{x,y;N,T_N}$, the
summands in $\sum_{h\in N^{-1/3}(\zz_{\ge0}+1/2)} L_h\big(X^{x,y;N,T_N}\big)\,\frac{\zeta(\lfloor
N-N^{1/3}h\rfloor)}{N^{1/6}}$ are independent. Therefore, the desired convergence will be a
consequence of the Central Limit Theorem in the form of the upper bound of \cite[Theorem 8.4]{BR},
provided we can prove the following three almost sure convergences as $N\to\infty$:
\begin{eqnarray}
&& \sum_{h\in N^{-1/3}(\zz_{\ge0}+1/2)}
\E_\zeta\bigg[L_h\big(X^{x,y;N,T_N}\big)\,\frac{\zeta\big(\lfloor N-N^{1/3}h\rfloor\big)}{N^{1/6}}\bigg]\longrightarrow 0, \label{CLTcond_mean}\\
 && \sum_{h\in N^{-1/3}(\zz_{\ge0}+1/2)}
\E_\zeta\bigg[\Big(L_h\big(X^{x,y;N,T_N}\big)\frac{\zeta\big(\lfloor N-N^{1/3}h \rfloor\big)}{N^{1/6}}\bigg)^2\bigg] \longrightarrow
s_\zeta^2\int_0^\infty \left[L_a\big(B^{x,y}\big)\right]^2 \mathrm{d}a, \label{CLTcond1}  \\
&& \sum_{h\in N^{-1/3}(\zz_{\ge0}+1/2)}
\E_\zeta\bigg[\bigg|L_h\big(X^{x,y;N,T_N}\big)\,\frac{\zeta\big(\lfloor N-N^{1/3}h\rfloor\big)}{N^{1/6}}\bigg|^3\,\bigg]\longrightarrow 0.
\label{CLTcond2}
\end{eqnarray}
As before, $\E_\zeta$ denotes the expectation with respect to the $\zeta(m)$, $m\in\nn$. The
convergences \eqref{CLTcond_mean} and \eqref{CLTcond1} pin down the limiting mean and variance,
whereas the convergence \eqref{CLTcond2} is the well-known Lyapunov condition for the Central Limit
Theorem.

\medskip

The coupling of Proposition \ref{lemma_coupling} reveals that, with probability one, the
$L_h\big(X^{x,y;N,T_N}\big)$ are uniformly bounded in $h\in\rr$, $N\in\nn$ and vanish for all $h$
large enough, so that \eqref{CLTcond_mean} follows from \eqref{eq_zeta_condition_1}. Similarly,
\eqref{CLTcond2} follows from \eqref{eq_zeta_condition_3}. To obtain \eqref{CLTcond1}, we combine
the simple inequality $\big|r_1^2-r_2^2\big|\le 2\max(r_1,r_2)\,|r_1-r_2|$, $r_1,r_2>0$ with
Proposition \ref{lemma_coupling} to conclude that
\begin{eqnarray*}
&& \Bigg|\sum_{h\in N^{-1/3}(\zz_{\ge0}+1/2)}
\E_\zeta\bigg[\bigg(L_h\big(X^{x,y;N,T_N}\big)\frac{\zeta\big(\lfloor N-N^{1/3}h\rfloor\big)}{N^{1/6}}\bigg)^2\bigg] \\
&& \;\; - \sum_{h\in N^{-1/3}(\zz_{\ge0}+1/2)} \frac{L_h(B^{x,y})^2\,\E\big[\zeta\big(\lfloor N-N^{1/3}h\rfloor\big)^2\big]}{N^{1/3}}\,\Bigg| \\
&& \le
2\,\max_{h\in N^{-1/3}(\zz_{\ge0}+1/2)} \max\big(L_h\big(X^{x,y;N,T_N}\big),L_h\big(B^{x,y}\big)\big)\,
{\mathcal C}\,N^{-1/16} \\
&& \quad\cdot\,\frac{1}{N^{1/3}}\,\sum_{N^{-1/3}(\zz_{\ge0}+1/2)\ni h\le\max(\|X^{x,y;N,T_N}\|_\infty,\|B^{x,y}\|_\infty)}\E\big[\zeta\big(\lfloor N-N^{1/3}h\rfloor\big)^2\big].
\end{eqnarray*}
The latter expression tends to $0$ almost surely in the limit $N\to\infty$, since $h\mapsto
L_h\big(X^{x,y;N,T_N}\big)$ and $X^{x,y;N,T_N}$ are uniformly bounded in $N\in\nn$ (Proposition
\ref{lemma_coupling}) and \eqref{eq_zeta_condition_2} holds. It remains to observe the almost sure
convergence
\begin{equation*}
\sum_{h\in N^{-1/3}(\zz_{\ge0}+1/2)} \frac{\left[L_h(B^{x,y})\right]^2\,\E\big[\zeta\big(\lfloor
N-N^{1/3}h\rfloor\big)^2\big]}{N^{1/3}} \longrightarrow s_\zeta^2\,\int_0^\infty
\left[L_a\big(B^{x,y}\big)\right]^2\,\mathrm{d}a
\end{equation*}
as $N\to\infty$, which follows from  \eqref{eq_zeta_condition_2} and the uniform
continuity of $a\mapsto L_a\big(B^{x,y}\big)^2$.
\end{proof}

\smallskip

Next, we present an extension of Proposition \ref{Proposition_basic_conv} which allows the endpoints $x,y$ to be random. Given probability measures $\lambda,\mu$ on $[0,\infty)$ and $\tilde{T}\in N^{-2/3}\nn$, we write $X^{\lambda,\mu;N,\tilde{T}}$ for the random walk bridge whose endpoints $x,y$ are chosen independently according to the images of $\lambda,\mu$ under the map $z\mapsto\lfloor N-N^{1/3}z\rfloor$ and which makes $\tilde{T}N^{2/3}$ steps of size $\pm 1$ if $\tilde{T}N^{2/3}$ has the same parity as $\lfloor N-N^{1/3}y\rfloor-\lfloor N-N^{1/3}x\rfloor$ and $\tilde{T}N^{2/3}-1$ steps of size $\pm 1$ otherwise. For $x,y$ for which there are no bridges of the type described (because the endpoints are too far apart), we instead let $X^{\lambda,\mu;N,\tilde{T}}$ be a straigh line connecting $x$ to $y$ in time $\tilde{T}N^{2/3}$. Similarly, we write $B^{\lambda,\mu}$ for a Brownian bridge whose endpoints $x, y$ are selected independently according to $\lambda, \mu$, respectively. We further extend the setup of Proposition \ref{Proposition_basic_conv} by allowing for multiple bridges, considering local times at both (scaled) integer and half-integer points, and adding a second sequence $\hat{\zeta}(m)$, $m\in\nn$.

\begin{proposition} \label{Proposition_basic_convergence_upgrade}
Fix an $R\in\nn$, probability measures $\lambda_1,\lambda_2,\ldots,\lambda_R,\mu_1,\mu_2,\ldots,\mu_R$ on $[0,\infty)$, and $T_N$, $N\in\nn$ such that $\sup_N |T_N-T|N^{2/3}<\infty$ for some $T>0$ and $T_N N^{2/3}\in\nn$, $N\in\nn$.
For each $N\in\nn$, let $X^{\lambda_1,\mu_1;N,T_N},X^{\lambda_2,\mu_2;N,T_N},\ldots,X^{\lambda_R,\mu_R;N,T_N}$ be \emph{independent} random walk bridges as described above. Then, for any two \emph{independent} sequences $\zeta(m)$, $m\in\nn$ and $\hat{\zeta}(m)$, $m\in\nn$ satisfying the assumptions \eqref{eq_zeta_condition_1},
\eqref{eq_zeta_condition_2}, and \eqref{eq_zeta_condition_3}, the random vector
\begin{equation}\label{eq_sequence_randomized}
\begin{split}
\bigg(N^{-1/3}\big(N-X^{\lambda_r,\mu_r;N,T_N}\big),\sum_{N^{-1/3}(\zz+1/2)\ni h\le N^{-1/3}(N-1)} L_h\big(X^{\lambda_r,\mu_r;N,T_N}\big)\,\frac{\zeta(\lfloor N-N^{1/3}h\rfloor)}{N^{1/6}}, \\
\sum_{N^{-1/3}\zz\ni h\le N^{-1/3}(N-1)}
L_h\big(X^{\lambda_r,\mu_r;N,T_N}\big)
\,\frac{\hat{\zeta}(\lfloor N-N^{1/3}h\rfloor)}{N^{1/6}} \bigg)_{r=1}^R
\end{split}
\end{equation}
converges in the limit $N\to\infty$ in distribution to
\begin{equation}\label{eq_limit_randomized}
\bigg(B^{\lambda_r,\mu_r},\,s_\zeta\,\int_0^\infty L_a\big(B^{\lambda_r,\mu_r}\big)\,\mathrm{d}W_\zeta(a),\,s_{\hat \zeta} \, \int_0^\infty L_a\big(B^{\lambda_r,\mu_r}\big)\,\mathrm{d}W_{\hat\zeta}(a)\bigg)_{r=1}^R,
\end{equation}
where $B^{\lambda_1,\mu_1},B^{\lambda_2,\mu_2},\ldots,B^{\lambda_R,\mu_R}$ are independent Brownian
bridges on $[0,T]$ as described above, and $W_\zeta$, $W_{\hat\zeta}$ are independent standard
Brownian motions. In addition, the following convergences in distribution with respect to the
Skorokhod topology occur jointly with the convergence of \eqref{eq_sequence_randomized} to
\eqref{eq_limit_randomized}:
\begin{eqnarray}
&& \frac{1}{s_\zeta\,N^{1/6}} \sum_{m=N-\lfloor a N^{1/3}\rfloor}^N
 \zeta(m)\longrightarrow W_\zeta(a),\quad N\to\infty, \label{eq_limit_BM} \\
&& \frac{1}{s_{\hat \zeta}\,N^{1/6}} \sum_{m=N-\lfloor a N^{1/3}\rfloor}^N \hat{\zeta}(m)\longrightarrow W_{\hat{\zeta}}(a),\quad N\to\infty. \label{eq_limit_BM'}
\end{eqnarray}
\end{proposition}

\begin{proof}
For $R=1$, the replacement of deterministic endpoints of Proposition \ref{Proposition_basic_conv}
by random ones can be dealt with by simply integrating with respect to $\lambda_1$ and $\mu_1$.
Note that, since convergence in distribution amounts to convergence of expectations of bounded
continuous functionals, we can use the Dominated Convergence Theorem to justify the convergence of
the integrals of such. The sum over (scaled) integers in \eqref{eq_sequence_randomized} can be
analyzed in the same way as the sum over (scaled) half-integers, as Proposition
\ref{lemma_coupling} does not distinguish between integers and half-integers. For $R>1$, one can
repeat the proof of Proposition \ref{Proposition_basic_conv}, replacing the Central Limit Theorem
invoked there by its multidimensional version. The same Central Limit Theorem yields the additional
convergences \eqref{eq_limit_BM}, \eqref{eq_limit_BM'} in the sense of convergence of
finite-dimensional distributions, and the tightness result needed to upgrade the latter to
convergence of processes is standard (see e.g. \cite[Problem 8.4 and proof of Theorem 8.1]{Bi}).
\end{proof}

\subsection{Leading order terms in Theorem \ref{Theorem_Main}} \label{subsec_leading}

We fix a $T>0$, an interval $\mathcal{A}\subset\rr_{\ge0}$, and
functions $f,g\in\Co$. Moreover, for $N\in\nn$ and $\tilde{T}>0$ such that $\tilde{T}N^{2/3}\in\nn$ has the same parity as $\lfloor N-N^{1/3}x\rfloor-\lfloor N-N^{1/3}y\rfloor$, we let
\begin{equation*}
\Xi(x,y;N,\tilde{T}):=\frac{N^{1/3}}{2^{\tilde{T}N^{2/3}}}
\binom{\tilde{T}N^{2/3}}{\frac{1}{2}\big(\tilde{T}N^{2/3}+\lfloor N-N^{1/3}y\rfloor-\lfloor N-N^{1/3}x\rfloor\big)}.
\end{equation*}
Note that $\Xi(x,y;N,\tilde{T})$ gives the (normalized) number of possible trajectories of $X^{x,y;N,\tilde{T}}$. For each $N\in\nn$, we further define the random variables ${\mathrm Sc}(N)$ (for ``scalar product'') and ${\mathrm Tr}(N)$ (for ``trace'') by
\begin{equation}\label{eq_high_moment_2}
\begin{split}
&\mathrm{Sc}(N)=\frac{1}{2}\int_0^{N^{2/3}} \int_0^{N^{2/3}} f(x)\,g(y)\, \sum_{j=0}^{\lfloor TN^{2/3}\rfloor} \Bigg(
\Xi\big(x,y;N,(\lfloor T N^{2/3}\rfloor-j-\eps_{j,x,y}) N^{-2/3}\big) \\
&\E_X
\bigg[\mathbf{1}_{\{\forall t:\,N^{-1/3}(N-X(t))\in{\mathcal A}\}} \\
&\prod_{i=1}^{\lfloor T N^{2/3}\rfloor-j-\eps_{j,x,y}} \frac{\sqrt{X(iN^{-2/3})\wedge X((i-1)N^{-2/3})}}{\sqrt{N}} \bigg(1+\frac{\xi\big(X(iN^{-2/3})\wedge X((i-1)N^{-2/3})\big)}{\sqrt{X(iN^{-2/3})\wedge X((i-1)N^{-2/3})}}\bigg) \\
&\qquad\frac{1}{(2\sqrt{N})^j} \sum_{0\le i_1 \le\cdots\le i_j\le \lfloor T
N^{2/3}\rfloor-j-\eps_{j,x,y}}\;\prod_{j'=1}^j \aa\big(X(i_{j'}N^{-2/3})\big)\bigg]\Bigg)\,\mathrm{d}x\,\mathrm{d}y,
\end{split}
\end{equation}
\begin{equation}\label{eq_trace_approx}
\begin{split}
&\mathrm{Tr}(N)=\frac{1}{2}\int_0^{N^{2/3}} \; \sum_{j=0}^{\lfloor T N^{2/3}\rfloor} \bigg(
\Xi\big(x,x;N,(\lfloor T N^{2/3}\rfloor-j-\eps_{j,x,x})N^{-2/3}\big) \\
&\E_X\bigg[\mathbf{1}_{\{\forall t:\,N^{-1/3}(N-X(t))\in{\mathcal A}\}} \\
&\prod_{i=1}^{\lfloor T N^{2/3}\rfloor-j-\eps_{j,x,x}} \frac{\sqrt{X(iN^{-2/3})\wedge X((i-1)N^{-2/3})}}{\sqrt{N}} \bigg(1+\frac{\xi\big(X(iN^{-2/3})\wedge X((i-1)N^{-2/3})\big)}{\sqrt{X(iN^{-2/3})\wedge X((i-1)N^{-2/3})}}\bigg)\\
&\qquad\frac{1}{(2\sqrt{N})^j} \sum_{0\le i_1 \le\cdots\le i_j\le \lfloor T N^{2/3}\rfloor-j-\eps_{j,x,x}} \; \prod_{j'=1}^j \aa\big(X(i_{j'}N^{-2/3})\big)
\bigg]\Biggr)\,\mathrm{d}x,
\end{split}
\end{equation}
where $\epsilon_{j,x,y}\in\{0,1\}$ is such that $\lfloor TN^{2/3}\rfloor-j-\epsilon_{j,x,y}$ is even, and $X$ stands for $X^{x,y;N, (\lfloor T N^{2/3}\rfloor-j-\eps_{j,x,y})N^{-2/3}}$ in  \eqref{eq_high_moment_2} and for $X^{x,x;N, (\lfloor
T N^{2/3}\rfloor-j-\eps_{j,x,x})N^{-2/3}}$ in \eqref{eq_trace_approx}. The discussion of Section \ref{Section_combi} shows that ${\mathrm Sc}(N)=(\pi_Nf)'\mathcal M(T,{\mathcal A},N)(\pi_Ng)$ and ${\mathrm Tr}(N)={\rm Trace}\big(\mathcal M(T,\A,N)\big)$ in the notation of \ref{Theorem_Main}. Consequently, Theorem \ref{Theorem_Main} can be rephrased as follows.

\begin{theorem}\label{theorem_Main_restated}
Under Assumption \ref{Assumptions} and as $N\to\infty$, one has the following convergences for all fixed $T>0$, $\mathcal{A}\subset\rr_{\ge0}$, and $f,g\in\Co$:
\begin{equation}\label{eq_Scalar_limit}
\begin{split}
{\mathrm Sc}(N)\to & \frac{1}{\sqrt{2\pi T}}
\int_0^\infty \int_0^\infty f(x)\,
g(y)\,\exp\bigg(-\frac{(x-y)^2}{2T}\bigg)\,
\E_{B^{x,y}}\bigg[\mathbf{1}_{\{\forall t:\,B^{x,y}(t)\in{\mathcal A}\}} \\
&\qquad\qquad\qquad\;\;
\exp\bigg(-\frac{1}{2}\int_0^T B^{x,y}(t)\,\mathrm{d}t+\frac{1}{\sqrt{\beta}}\,\int_0^\infty L_a\big(B^{x,y}\big)\,\mathrm{d}W(a)\bigg)\bigg]\,\mathrm{d}x\,\mathrm{d}y,
\end{split}
\end{equation}
\begin{equation} \label{eq_Trace_limit}
\begin{split}
\mathrm{Tr}(N)\to & \frac{1}{\sqrt{2\pi T}} \int_0^\infty
\E_{B^{x,x}}\bigg[\mathbf{1}_{\{\forall t:\,B^{x,x}(t)\in{\mathcal A}\}} \\
&\qquad\qquad\qquad\quad\;\;
\exp\bigg(-\frac{1}{2}\int_0^T B^{x,x}(t)\,\mathrm{d}t
+\frac{1}{\sqrt{\beta}}\,\int_0^\infty L_a\big(B^{x,x}\big)\,\mathrm{d}W(a)\bigg)\bigg]\,\mathrm{d}x,
\end{split}
\end{equation}
and
\begin{equation}
\label{eq_limit_BM_identification}
\sqrt{\beta}\,N^{-1/6} \sum_{m=N-\lfloor N^{1/3} a \rfloor}^N
  \left(\xi(m)+\frac{\aa(m)}{2}\right)\to W(a).
\end{equation}
The convergences \eqref{eq_Scalar_limit} and \eqref{eq_Trace_limit} hold in distribution and in the
sense of moments, and the same apply jointly to any finite collection of $T$'s, ${\mathcal A}$'s,
$f$'s and $g$'s. The convergence in \eqref{eq_limit_BM_identification} is in distribution with
respect to the Skorokhod topology and takes place jointly with both convergences
\eqref{eq_Scalar_limit}, \eqref{eq_Trace_limit} (also for finitely many $T$'s, ${\mathcal A}$'s,
$f$'s and $g$'s).
\end{theorem}


The proof of Theorem \ref{theorem_Main_restated} is divided into two parts. First, we identify the leading order terms in ${\mathrm Sc}(N)$ and ${\mathrm Tr}(N)$ and compute their limits in Proposition \ref{Proposition_main_single_term}. Then, we analyze the corresponding remainder terms and prove that they vanish asymptotically in Lemmas \ref{Lemma_moment_bound}, \ref{Lemma_moment_bound_upgrade}, and \ref{Lemma_tail_sum}. In
Section \ref{Section_complete_proof}, we explain how these results lead to Theorem \ref{theorem_Main_restated}.

\medskip

We start by defining $\mathrm{Sc}^{(j)}(N;K,\underline R, \overline R)$ and $\mathrm{Tr}^{(j)}(N;K,\underline R, \overline R)$ which will turn out to be the leading order contributions to the $j$-th terms in $\mathrm{Sc}(N)$ and $\mathrm{Tr}(N)$, respectively. For any $K\in[0,\infty),\,\underline R\in [-\infty,0],\,\overline R\in [0,\infty]$ (we allow infinite values for $\underline R,\overline R$), set
\begin{equation}\label{eq_high_moment_term}
\begin{split}
&{\mathrm Sc}^{(j)}(N;K,\underline R, \overline R):=\frac{1}{2}\int_0^K \int_0^K f(x)\,g(y)\,\Xi(x,y;N,T_N)\,
\E_X\Bigg[\mathbf{1}_{\{\forall t:\,N^{-1/3}(N-X(t))\in{\mathcal A}\}} \\
&\underline R \vee\Bigg(
 \prod_{i=1}^{T_NN^{2/3}} \frac{\sqrt{X(iN^{-2/3})\wedge
X((i-1)N^{-2/3})}}{\sqrt{N}} \bigg(1+\frac{\xi(X(iN^{-2/3})\wedge X((i-1)N^{-2/3}))}{\sqrt{X(iN^{-2/3})\wedge X((i-1)N^{-2/3})}}\bigg) \\
& \;\;\qquad\qquad\qquad\qquad\qquad\qquad\qquad\qquad\qquad
\frac{\Big(\sum_{i'=0}^{T_NN^{2/3}} \aa\big(X(i'N^{-2/3})\big)\Big)^j}{j!\,(2\sqrt{N})^j}\Bigg)\wedge\overline R\Bigg]\,\mathrm{d}x\,\mathrm{d}y,
\end{split}
\end{equation}
\begin{equation}\label{eq_trace_term}
\begin{split}
& {\mathrm Tr}^{(j)}(N;K,\underline R, \overline R):=\frac{1}{2}\int_0^K \, \Xi(x,x;N,T_N)\,\E_X
\Bigg[\mathbf{1}_{\{\forall t:\,N^{-1/3}(N-X(t))\in{\mathcal A} \}} \\
&\underline R \vee \Bigg(
 \prod_{i=1}^{T_NN^{2/3}+1} \frac{\sqrt{X(iN^{-2/3})\wedge
X((i-1)N^{-2/3})}}{\sqrt{N}} \bigg(1+\frac{\xi(X(iN^{-2/3})\wedge X((i-1)N^{-2/3})}{\sqrt{X(iN^{-2/3})\wedge X((i-1)N^{-2/3})}}\bigg) \\
&\qquad\qquad\qquad\qquad\qquad\qquad\qquad\qquad\qquad\qquad
\frac{\Big(\sum_{i'=0}^{T_NN^{2/3}} \aa\big(X(i'N^{-2/3})\big) \Big)^j}{j!\,(2\sqrt{N})^j}\Bigg)\wedge\overline R \Bigg]\,
\mathrm{d}x,
\end{split}
\end{equation}
where $X$ stands for $X^{x,y;N,(\lfloor T N^{2/3}\rfloor-j-\eps_{j,x,y})N^{-2/3}}$ in \eqref{eq_high_moment_term} and for $X^{x,x;N,(\lfloor T N^{2/3}\rfloor-j-\eps_{j,x,x})N^{-2/3}}$ in \eqref{eq_trace_term}; $\underline R\vee\cdot\wedge R$ is an abbreviation for $\max(\underline R,\min(\cdot,\overline R))$; and $T_N$ is given by $(\lfloor T N^{2/3}\rfloor-j-\eps_{j,x,y})N^{-2/3}$ in \eqref{eq_high_moment_term} and by $(\lfloor T N^{2/3}\rfloor-j-\eps_{j,x,x})N^{-2/3}$ in \eqref{eq_trace_term}. The exact formulas for $T_N$ are not important for the arguments below:  we will only use that $\sup_N |T_N-T|N^{2/3}<\infty$ for every fixed $j$ in Proposition \ref{Proposition_main_single_term} and that $\sup_N T_N<\infty$ in Lemmas \ref{Lemma_moment_bound}, \ref{Lemma_moment_bound_upgrade}, and \ref{Lemma_tail_sum}.

\begin{proposition} \label{Proposition_main_single_term}
In the situation of Theorem \ref{theorem_Main_restated} the following convergences hold as $N\to\infty$ for any finite collection of $j$'s and any fixed $K\in[0,\infty)$, $\underline R\in[-\infty,0]$, $\overline R\in[0,\infty]$ (the modes of convergence are the same as in Theorem \ref{theorem_Main_restated}):
\begin{equation}\label{eq_Scalar_limit_term}
\begin{split}
& {\mathrm Sc}^{(j)}(N;K,\underline R,\overline R)\to\frac{1}{\sqrt{2\pi T}} \int_0^K \int_0^K f(x)\,g(y)\,
\exp\left(-\frac{(x-y)^2}{2T}\right)\E_{B^{x,y}}\Bigg[\mathbf{1}_{\{\forall t:\,B^{x,y}(t)\in{\mathcal A}\}}\,\underline R \\
& \vee\!\Bigg(\!
\exp\bigg(\!-\frac{1}{2}\int_0^T B^{x,y}(t)\mathrm{d}t
+s_\xi\int_0^\infty L_a\big(B^{x,y}\big)\mathrm{d}W_\xi(a)\!\bigg)
\frac{\big(s_\aa\int_0^\infty L_a\big(B^{x,y}\big)\mathrm{d}W_\aa(a)\big)^j}{2^j\,j!}\Bigg)\!\wedge\!\overline R \Bigg] \\
& \qquad\qquad\qquad\qquad\qquad\qquad\qquad\qquad\qquad\qquad
\qquad\qquad\qquad\qquad\qquad\qquad\qquad\;
\mathrm{d}x\,\mathrm{d}y,
\end{split}
\end{equation}
\begin{equation} \label{eq_Trace_limit_term}
\begin{split}
& {\mathrm Tr}^{(j)}(N;K,\underline R,\overline R)\to\frac{1}{\sqrt{2\pi T}} \int_0^K \E_{B^{x,x}}\bigg[\mathbf{1}_{\{\forall t:\,B^{x,x}(t)\in{\mathcal A}\}}\,\underline R \\
& \vee\!\bigg(\!\exp\bigg(\!-\frac{1}{2}\int_0^T B^{x,x}(t)\mathrm{d}t
+s_\xi\int_0^\infty L_a\big(B^{x,x}\big)\mathrm{d}W_\xi(a)\!\bigg)\frac{\big(s_\aa\int_0^\infty L_a\big(B^{x,x}\big)\mathrm{d}W_\aa(a)\big)^j}{2^j\,j!} \bigg)\!\wedge\!\overline R\bigg]\\
& \qquad\qquad\qquad\qquad\qquad\qquad\qquad\qquad\qquad\qquad
\qquad\qquad\qquad\qquad\qquad\qquad\qquad\quad\;\,\mathrm{d}x,
\end{split}
\end{equation}
\begin{equation}\label{eq_limit_BM_identification_1}
s_\xi\,N^{-1/6} \sum_{h=N-\lfloor a N^{1/3}\rfloor}^N \xi(h)
\to W_\xi(a),\quad\mathrm{and}\quad
s_\aa\,N^{-1/6} \sum_{h=N-\lfloor a N^{1/3}\rfloor}^N \aa(h)
\to W_\aa(a).
\end{equation}
\end{proposition}

\smallskip

We prove Proposition \ref{Proposition_main_single_term} in a series of lemmas. In all of their proofs we only consider ${\mathrm Sc}^{(j)}(N;K,\underline R,\overline R)$, since the arguments for ${\mathrm Tr}^{(j)}(N; K,\underline R,\overline R)$ are exactly the same.


\begin{lemma} \label{Lemma_Binomial_bound}
Let $T_N$, $N\in\nn$ be such that $\sup_N |T_N-T|N^{2/3}<\infty$ and $T_NN^{2/3}\in\nn$, $N\in\nn$. Then, uniformly in $x$, $y$ in any compact set and such that $\lfloor N-N^{1/3}x\rfloor-\lfloor N-N^{1/3}y\rfloor$ has the same parity as $T_NN^{2/3}$, it holds
\begin{equation}
\Xi(x,y;N,T_N)\to\sqrt{\frac{2}{\pi T}}\,e^{-(x-y)^2/(2T)},\quad N\to\infty.
\end{equation}
In addition, there exists a $C>0$ such that
\begin{equation}\label{binom_bound}
\Xi(x,y;N,\tilde{T}) \le C\,e^{-(x-y)^2/(C\tilde{T})}
\end{equation}
for all $x,y\in\rr$, $N\in\nn$, and $\tilde{T}>0$ with $\tilde{T}N^{2/3}\in\nn$ being of the same parity as $\lfloor N-N^{1/3}x\rfloor-\lfloor N-N^{1/3}y\rfloor$.
\end{lemma}
\begin{proof}
The first statement is a special case of the de Moivre-Laplace Theorem in the form of \cite[Section VII.3, Theorem 1]{Fe}. To obtain the second statement we distinguish between the cases $(x-y)^2/\tilde{T}\le \frac{4}{3}\log N$ and $(x-y)^2/\tilde{T}>\frac{4}{3}\log N$. In the first case, the bound \eqref{binom_bound} follows directly from \cite[Section VII.3, Theorem 1, upper bound in (3.14)]{Fe}. In the second case and for $x$, $y$, $N$, $\tilde{T}$ such that $\Xi(x,y;N,\tilde{T})>0$, we apply the non-asymptotic upper bound from the proof of Cram\'er's Theorem (see e.g. \cite[Remark (c) after Theorem 2.2.3]{DZ}) to get
\begin{equation}\label{binom_bound1}
\Xi(x,y;N,\tilde{T})\le 2N^{1/3}\exp\bigg(\!-\tilde{T}N^{2/3}\Lambda^*\bigg(\frac{1}{2}+\frac{\big|\lfloor N-N^{1/3}x\rfloor-\lfloor N-N^{1/3}y\rfloor\big|}{2\tilde{T}N^{2/3}}\bigg)\bigg)
\end{equation}
with $\Lambda^*(z):=\log 2+z\log z+(1-z)\log(1-z)$. It is easy to check
$\Lambda^*(1/2)=(\Lambda^*)'(1/2)=0$ and $(\Lambda^*)''(z)\ge 4$, which together
imply $\Lambda^*(z)\ge 2(z-1/2)^2$. Inserting the latter bound into
\eqref{binom_bound1} gives \eqref{binom_bound}.
\end{proof}

\begin{lemma} \label{Lemma_truncated_theorem} Proposition
\ref{Proposition_main_single_term} holds for
$K\in[0,\infty)$, $\underline R\in(-\infty,0]$, $
\overline R\in[0,\infty)$.
\end{lemma}

\begin{proof}
Without loss of generality we may assume that
\begin{equation*}
f\ge 0,\quad g\ge0,\quad \int_0^K f(x)\,\mathrm{d}x=\int_0^K g(y)\,\mathrm{d}y=1,
\end{equation*}
since otherwise we can decompose $f,\,g$ into their positive and negative parts and
normalize the latter appropriately. Further, we let $\lambda$, $\mu$ be the
probability distributions on $[0,\infty)$ with densities $f\,\mathbf{1}_{[0,K]}$,
$g\,\mathbf{1}_{[0,K]}$, respectively. In addition, for every $N\in\nn$ and
$R\in\nn$, we introduce the independent copies $X^N_r$, $r=1,\,2,\,\ldots,\,R$ of
the random walk bridge $X^{\lambda,\mu;N,T_N}$ (with $T_N$ defined in the paragraph
preceding Proposition \ref{Proposition_main_single_term}). Then, thanks to Fubini's
Theorem, the $R$-th moment of the random variable $Sc^{(j)}(N,K,\underline
R,\overline R)$ can be expressed as
\begin{equation}\label{eq_replica}
\begin{split}
& \E\Bigg[\prod_{r=1}^R \Xi(X^N_r(0),X^N_r(T_N);N,T_N)\,\mathbf{1}_{\{\forall t:\,N^{-1/3}(N-X^N_r(t))\in{\mathcal A}\}}\,  \underline R \\
&\;\vee\!\Bigg(\!\prod_{i=1}^{T_NN^{2/3}} \!\! \frac{\sqrt{
X^N_r(iN^{-2/3})\!\wedge\! X^N_r((i-1)N^{-2/3})}}{\sqrt{N}}
\left(\!1\!+\!\frac{\xi\big(X^N_r(iN^{-2/3})\!\wedge\!
X^N_r((i-1)N^{-2/3})\big)}{\sqrt{X^N_r(iN^{-2/3})\!\wedge\! X^N_r((i-1)N^{-2/3})}}\!\right) \\
&\;\;\frac{\Big(\sum\limits_{i'=0}^{T_NN^{2/3}}
\aa\big(X^N_r(i'N^{-2/3})\big)\Big)^j}{j!\,(2\sqrt{N})^j}\Bigg)\!\wedge\!\overline R\Bigg].
\end{split}
\end{equation}

\smallskip

In order to determine the limit of the latter expectation, we first analyze the random variable inside the expectation.
Applying Proposition \ref{Proposition_basic_convergence_upgrade} to the random walk bridges $X^N_1,X^N_2,\ldots,X^N_R$ and writing
$B_1,B_2,\ldots,B_R$ for the Brownian bridges arising in the limit  we find
\begin{equation}\label{BM_integral_term}
\begin{split}
& \prod_{i=1}^{T_N N^{2/3}} \frac{\sqrt{X^N_r(iN^{-2/3})\wedge
X^N_r((i-1)N^{-2/3})}}{\sqrt{N}} \\
&=\exp\bigg(-\frac{1}{2N^{2/3}}\sum_{i=0}^{T_NN^{2/3}} N^{-1/3}\Big(N-X^N_r\big(iN^{-2/3}\big)\wedge X^N_r\big((i-1)N^{-2/3}\big)\Big)+o(1)\bigg) \\
&=\exp\bigg(-\frac{1}{2}\int_0^T B_r(t)\,\mathrm{d}t+o(1)\bigg),\quad r=1,2,\ldots,R.
\end{split}
\end{equation}

\smallskip

Next, we use employ the uniform convergence of $N^{-1/3}(N-X^N_r)$, $r=1,2,\ldots,R$ (Proposition \ref{Proposition_basic_convergence_upgrade}) to obtain
\begin{equation}\label{xi_ana1}
\frac{\xi\big(\!X^N_r(iN^{-2/3})\!\wedge\!
X^N_r((i-1)N^{-2/3})\!\big)}{\sqrt{\!X^N_r(iN^{-2/3})\!\wedge\! X^N_r((i-1)N^{-2/3})}}\!=\!\frac{\xi\big(\!X^N_r(iN^{-2/3})\!\wedge\!X^N_r((i-1)N^{-2/3})\!\big)}{\sqrt{N}}(1+O(N^{-2/3})\!)
\end{equation}
for all $r=1,2,\ldots,R$. In addition, we note that there exists a real random variable ${\mathcal C}_\xi>0$ such that, with probability one,
\begin{equation}\label{xi_log_bound}
\max_{1 \le m \le N} \xi(m)\le{\mathcal C}_\xi\,\log N,\quad N\in\nn.
\end{equation}
Indeed, for a non-random $C\in\rr$, the probability $\pp\big(\max_{1 \le m \le N} \xi(m)>C\,\log N\big)$ can be written in terms of the tail distribution functions of the $\xi(m)$'s which, in turn, can be estimated using \eqref{new_tail_bound1} and Assumption \ref{Assumptions} (a). A simple application of the Borel-Cantelli Lemma then gives \eqref{xi_log_bound}. Putting together \eqref{xi_ana1}, \eqref{xi_log_bound}, and the elementary inequalities $e^{z-z^2}\le 1+z\le e^z$ valid for all $z$'s close enough to $0$ we conclude that, for each $r=1,2,\ldots,R$, the product
\begin{equation}\label{eq_replace_linear}
\prod_{i=1}^{T_N N^{2/3}} \Bigg(1+\frac{\xi\big(X^N_r(iN^{-2/3})\wedge X^N_r((i-1)N^{-2/3})\big)}{\sqrt{X^N_r(iN^{-2/3})\wedge X^N_r((i-1)N^{-2/3})}}\Bigg)
\end{equation}
behaves asymptotically as
\begin{equation} \label{eq_replace_exponent}
\exp\Bigg(\sum_{i=1}^{T_N N^{2/3}} \frac{\xi\big(X^N_r(iN^{-2/3})\wedge X^N_r((i-1)N^{-2/3})\big)}{\sqrt{N}}\Bigg),
\end{equation}
with an asymptotic multiplicative error term of at most
\begin{equation}\label{multi_error}
\exp\Bigg(\sum_{i=1}^{T_NN^{2/3}} \frac{\xi\big(X^N_r(iN^{-2/3})\wedge X^N_r((i-1)N^{-2/3})\big)^2}{{N}}\Bigg).
\end{equation}
Moreover, an estimate on the fourth moment of the exponent in the latter exponential via Assumption \ref{Assumptions} (c) and the Borel-Cantelli Lemma reveal the almost sure convergence of that exponent to $0$. Consequently, the expression in \eqref{multi_error} tends to $1$ almost surely. Applying  Proposition \ref{Proposition_basic_convergence_upgrade} to the exponent in \eqref{eq_replace_exponent} and combining the result with \eqref{BM_integral_term} we end up with the asymptotic
\begin{equation}\label{second_line_asymp}
\exp\bigg(-\frac{1}{2}\,\int_0^T B_r(t)\,\mathrm{d}t +s_\xi\,\int_0^\infty L_a(B_r)\,\mathrm{d}W_\xi(a)+o(1)\bigg)
\end{equation}
for the second line in \eqref{eq_replica}, for any $r=1,2,\ldots,R$.

\medskip

Next, we apply Lemma \ref{Lemma_Binomial_bound} to $\Xi(X^N_r(0),X^N_r(T_N);N,T_N)$ in the first line of \eqref{eq_replica}, use Proposition \ref{Proposition_basic_convergence_upgrade} for the third
line of \eqref{eq_replica}, and combine the results with \eqref{second_line_asymp} to find that the random variable inside the expectation in \eqref{eq_replica} converges in distribution to
\begin{equation} \label{eq_replicas_lim}
\begin{split}
& \Bigg(\prod_{r=1}^R \sqrt{\frac{2}{\pi T}}\,e^{-(B_r(0)-B_r(T))^2/(2T)} \mathbf{1}_{\{\forall t:\,B_r(t)\in{\mathcal A}\}}\,\underline R \\
& \vee \exp\bigg(-\frac{1}{2}\,\int_0^T B_r(t)\,\mathrm{d}t
+s_\xi\,\int_0^\infty L_a(B_r)\,\mathrm{d}W_\xi(a)\bigg)
\,\frac{\big(s_\aa\,\int_0^\infty L_a\big(B_r\big)\,\mathrm{d}
W_\aa(a)\Big)^j}{j!\,2^j}\Bigg)\wedge\overline R.
\end{split}
\end{equation}
Since the random variables in consideration are uniformly bounded, the expectation in
\eqref{eq_replica} converges to the expectation of the random variable in \eqref{eq_replicas_lim}.
By Fubini's Theorem, the latter is precisely the $R$-th moment of the limit in
\eqref{eq_Scalar_limit_term}. Thanks to the boundedness of the random variables involved the
convergence of moments also yields the convergence in distribution.

\smallskip

The joint convergence for several $T$'s, $\A$'s, $f$'s and $g$'s can be established in exactly the same way, by considering joint moments. Finally, the convergences in \eqref{eq_limit_BM_identification_1} are direct consequences of
\eqref{eq_limit_BM}, \eqref{eq_limit_BM'} in Proposition \ref{Proposition_basic_convergence_upgrade}, and the joint convergence with \eqref{eq_Scalar_limit_term} can be deduced by taking joint moments with bounded continuous functionals of the prelimit expressions in \eqref{eq_limit_BM_identification_1}.
\end{proof}

\begin{lemma} \label{Lemma_theorem_in_moments}
The convergences in Proposition \ref{Proposition_main_single_term} hold in the sense of moments for any $K\in[0,\infty)$, $\underline R\in[-\infty,0]$, $\overline R\in[0,\infty]$.
\end{lemma}

\begin{proof}
Our proof  for the convergence in distribution of the random variable inside the expectation in
\eqref{eq_replica} to the one in \eqref{eq_replicas_lim} remains valid when one of $\underline R$,
$\overline R$ or both of them are infinite. It therefore remains to show the convergence of the
corresponding expectations. To this end, it suffices to establish the uniform integrability of the
prelimit random variables which, in turn, would follow from the uniform boundedness of their second
moments.

\smallskip

Clearly, for any $\underline R\in[-\infty,0]$, $\overline R\in[0,\infty]$, the second moment of the random variable inside the expectation in \eqref{eq_replica} is bounded above by
\begin{equation}\label{eq_replica_second_moment}
\begin{split}
& \E\Bigg[\prod_{r=1}^R \Xi(X^N_r(0),X^N_r(T_N);N,T_N)^2
\prod_{i=1}^{T_NN^{2/3}}
\left(1+\frac{\xi\big(X^N_r(iN^{-2/3})\wedge
X^N_r((i-1)N^{-2/3})\big)}{\sqrt{X^N_r(iN^{-2/3})\wedge X^N_r((i-1)N^{-2/3})}}\right)^2 \\
&\quad\frac{\Big(\sum\limits_{i'=0}^{T_NN^{2/3}}
\aa\big(X^N_r(i'N^{-2/3})\big)\Big)^{2j}}{(j!)^2\,(2\sqrt{N})^{2j}}\Bigg].
\end{split}
\end{equation}
In view of Lemma \ref{Lemma_Binomial_bound}, we can bound the factors $\Xi(X^N_r(0),X^N_r(T_N);N,T_N)^2$, $r=1,2,\ldots,R$ by a  uniform constant. Moreover, we can estimate the expectation with respect to the $\xi$'s of the second product in the first line of \eqref{eq_replica_second_moment} by applying
\begin{equation*}
X^N_r(iN^{-2/3})\wedge X^N_r((i-1)N^{-2/3})\ge\lfloor N-N^{1/3}K\rfloor-T_NN^{2/3}\ge N/2,\;\; i=1,2,\ldots,T_NN^{2/3}
\end{equation*}
to the denominators for all large enough $N$ and then using \eqref{new_tail_bound2}. Finally, to the quantity in the second line of \eqref{eq_replica_second_moment} we apply the chain of elementary inequalities
\begin{equation*}
\bigg(\frac{|z|^j}{j!}\bigg)^2 \le e^{2|z|}\le e^{2z}+e^{-2z}, \quad z\in\rr,\;j\in\nn\cup\{0\}
\end{equation*}
and then bound the expectation with respect to the $\aa$'s by means of \eqref{new_tail_bound1}. All in all, we end up with the estimate
\begin{equation} \label{eq_tail_estimate_1}
\begin{split}
C\,\E\bigg[\!\exp\bigg(C\!\!\sum_{h\in
N^{-1/3}(\zz+1/2)}\!\!\bigg(
\big|\E\big[\xi(N-N^{1/3}h-1/2)\big]\big|\,\frac{\sum_{r=1}^R L_h\big(X^N_r\big)}{N^{1/6}}+\frac{\sum_{r=1}^R L_h\big(X^N_r\big)^2}{N^{1/3}} \\
+\frac{\sum_{r=1}^R L_h\big(X^N_r\big)^{\gamma'}}{N^{\gamma'/6}}\bigg)+C\sum_{h\in N^{-1/3}\zz} \bigg(\big|\E\big[\aa(N-N^{1/3}h)\big]\big|\,\frac{\sum_{r=1}^R L_h\big(X^N_r\big)}{N^{1/6}} \\
+\frac{\sum_{r=1}^R L_h\big(X^N_r\big)^2}{N^{1/3}}
+\frac{\sum_{r=1}^R L_h\big(X^N_r\big)^{\gamma'}}{N^{\gamma'/6}}\bigg)\bigg)\bigg],
\end{split}
\end{equation}
where $C>0$ is a uniform constant and $2<\gamma'<3$ is the same as in Lemma \ref{Lemma_tail_estimates}.

\medskip

Finally, we recall from Assumption \ref{Assumptions} that $\big|\E[\xi(m)]\big|=o(m^{-1/3})$, $\big|\E(\aa(m))\big|=o(m^{-1/3})$ and from above that the bridges $X_r^N$, $r=1,2,\ldots,R$ do not reach $N/2$ for all large enough $N$. Consequently,
\begin{eqnarray*}
&& \sum_{h\in N^{-1/3}(\zz+1/2)} \big|\E\big[\xi(N-N^{1/3}h-1/2)\big]\big|\,\frac{\sum_{r=1}^R L_h\big(X^N_r\big)}{N^{1/6}}=o\bigg( N^{-1/3}\,\frac{ N^{1/3}}{N^{1/6}}\bigg), \\
&&\quad\,\sum_{h\in N^{-1/3}\zz} \big|\E\big[\aa(N-N^{1/3}h)\big]\big|\,\frac{\sum_{r=1}^R L_h\big(X^N_r\big)}{N^{1/6}}=o\bigg(N^{-1/3}\, \frac{N^{1/3}}{N^{1/6}}\bigg),
\end{eqnarray*}
where the error terms are non-random and tend to $0$ in the limit $N\to\infty$. The remaining part of the expression in \eqref{eq_tail_estimate_1} can be controlled by applying the Cauchy-Schwarz inequality and then \eqref{UIlt} (recall Remark \ref{rmk_loc_time}).
\end{proof}

\begin{lemma} \label{Lemma_theorem_in_distribution} The convergences  in Proposition \ref{Proposition_main_single_term} hold in distribution for any $K\in[0,\infty)$, $\underline R\in[-\infty,0]$, $\overline R\in[0,\infty]$.
\end{lemma}

\begin{rmk}
Note that, for infinite $\underline R$, $\overline R$, one should expect the moments of the limits in \eqref{eq_Scalar_limit_term},
 \eqref{eq_Trace_limit_term} to grow very fast, since the latter include exponentials of random variables with Gaussian tails. Therefore, a priori, the just established convergence of moments does not imply the convergence in distribution.
\end{rmk}

\noindent\textit{Proof of Lemma \ref{Lemma_theorem_in_distribution}}.
For finite $\underline R$, $\overline R$, the lemma is a direct consequence of Lemma \ref{Lemma_theorem_in_moments} and the boundedness of the limits in \eqref{eq_Scalar_limit_term},
 \eqref{eq_Trace_limit_term}. We turn to the case that $\underline R$ is finite and $\overline R$ is infinite. As before, we only consider $\mathrm{Sc}^{(j)}(N;K,\underline R,\infty)$ and functions $f\ge0$, $g\ge0$ (otherwise we can write $f$, $g$ as the differences of their positive and negative parts and use the joint convergence in distribution of the associated scalar products). The convergence of moments shows that the random variables $\mathrm{Sc}^{(j)}(N;K,\underline R,\infty)$, $N\in\nn$ form a tight sequence, so it suffices to identify the limit points of the latter with the random variable in \eqref{eq_Scalar_limit_term}. Let $\mathrm{Sc}^{(j)}(\infty;K,\underline R,\infty)$ be such a limit point.

\medskip

Note that, for each $N\in\nn$ and $\overline R\in[0,\infty)$, the random variable
$\mathrm{Sc}^{(j)}(N;K,\underline R,\infty)$ stochastically dominates
$\mathrm{Sc}^{(j)}(N;K,\underline R, \overline R)$. Consequently, the limit point
$\mathrm{Sc}^{(j)}(\infty;K,\underline R,\infty)$ must stochastically dominate the
limit in distribution $\lim_{N\to\infty} \mathrm{Sc}^{(j)}(N;K,\underline R,
\overline R)$ for every $\overline R\in[0,\infty)$. The latter is given by the
expression in \eqref{eq_Scalar_limit_term} and, by the Monotone Convergence Theorem,
tends in the limit $\overline R\uparrow\infty$ to the corresponding expression with
$\overline R=\infty$. We conclude that the limit point
$\mathrm{Sc}^{(j)}(\infty;K,\underline R,\infty)$ and the expression in
\eqref{eq_Scalar_limit_term} with $\overline R=\infty$ are two non-negative random
variables with equal moments such that the first of them stochastically dominates
the second. Clearly, such random variables must have the same distribution.

\medskip

Similarly, the case $\underline R=-\infty$ can be dealt with by using the stochastic domination as $\underline R$ varies. The same argument also applies to any finite collection of $j$'s, $T$'s, $\A$'s, and non-negative $f$'s, $g$'s and gives the joint convergence with \eqref{eq_limit_BM_identification_1} as well.  \ep

\medskip

So far, we have established Proposition \ref{Proposition_main_single_term}, which identifies the leading order contributions in the limits of Theorem \ref{theorem_Main_restated}. It remains to control the associated remainder terms. To this end, for each $N\in\nn$ and $K\in[0,\infty)$, we define
\begin{eqnarray*}
&& \overline{\mathrm{Sc}}^{(j)}(N;K)=\mathrm{Sc}^{(j)}(N;N^{2/3},-\infty,\infty)-\mathrm{Sc}^{(j)}(N;K,-\infty,\infty), \\
&& \overline{\mathrm{Tr}}^{(j)}(N;K)=\mathrm{Tr}^{(j)}(N;N^{2/3},-\infty,\infty)-\mathrm{Tr}^{(j)}(N;K,-\infty,\infty).
\end{eqnarray*}

\begin{lemma} \label{Lemma_moment_bound}
There exist positive constants $C(R,K)$, $R\in\nn$, $K\in[0,\infty)$ (possibly depending on $f$, $g$, $\A$, but not on $N$) such that
\begin{equation*}
\E\Big[\big|\overline{\mathrm{Sc}}^{(j)}(N;K)\big|^R\Big]
\le \frac{C(R,K)}{2^{jR}},\qquad
\E\Big[\big|\overline{\mathrm{Tr}}^{(j)}(N;K)\big|^R\Big]
\le \frac{C(R,K)}{2^{jR}}
\end{equation*}
for all $N\in\nn$, and $\lim_{K\to\infty} C(R,K)=0$, $R\in\nn$.
\end{lemma}

\begin{proof}
We only give the proof for $\overline{\mathrm{Sc}}^{(j)}(N;K)$, since the argument for $\overline{\mathrm{Tr}}^{(j)}(N;K)$ is very similar.
With ${\mathcal R}_{N,K}:=\big[0,N^{2/3}\big]^2\,\backslash[0,K]^2$ and the notation introduced in the proof of Lemma \ref{Lemma_truncated_theorem}, the $R$-th absolute moment of $\overline{\mathrm{Sc}}^{(j)}(N;K)$ is bounded above by
\begin{equation}\label{eq_tail_rv_moment}
\begin{split}
& \E\Bigg[\int_{({\mathcal R}_{N,K})^R} \prod_{r=1}^R \Bigg(|f(x_r)|\,|g(y_r)|\,\Xi(x_r,y_r;N,T_N) \\
&\qquad \E_{X^N_r}\bigg[
 \prod_{i=1}^{T_NN^{2/3}} \frac{\sqrt{X^N_r(iN^{-2/3})\wedge X^N_r((i-1)N^{-2/3})}}{\sqrt{N}} \\
&\qquad \Bigg|1+\frac{\xi\big(X^N_r(iN^{-2/3})\wedge
X((i-1)N^{-2/3})\big)}{\sqrt{X^N_r(iN^{-2/3})\wedge X^N_r((i-1)N^{-2/3})}}\Bigg|
\frac{\big|\sum_{i'=0}^{T_N N^{2/3}} \aa\big(X^N_r(i'N^{-2/3})\big) \big|^j}{j!\,(2\sqrt{N})^j}
\bigg]\mathrm{d}x_r\,\mathrm{d}y_r\Bigg)\!\Bigg]\!.
\end{split}
\end{equation}
Relying on Fubini's Theorem, we can first take all expectations in \eqref{eq_tail_rv_moment} and only then integrate over $({\mathcal R}_{N,K})^R$. Moreover, for fixed $x_r,\,y_r$, $r=1,2,\ldots,R$, we can use H\"older's inequality to bound the expectations inside the integral over $({\mathcal R}_{N,K})^R$ by the product of the $R$-th roots of the expectations
\begin{equation}\label{eq_tail_rv_moment_reduced}
\begin{split}
& \E\Bigg[|f(x_r)|^R|g(y_r)|^R\,\Xi(x_r,y_r;N,T_N)^R
\prod_{i=1}^{T_NN^{\frac{2}{3}}}\bigg(\frac{X^N_r(iN^{-\frac{2}{3}})\wedge X^N_r((i-1)N^{-\frac{2}{3}})}{N}\bigg)^{\frac{R}{2}} \\
&\quad\;\bigg|1+\frac{\xi\big(X^N_r(iN^{-2/3})\wedge
X^N_r((i-1)N^{-2/3})\big)}{\sqrt{X^N_r(iN^{-2/3})\wedge
X^N_r((i-1)N^{-2/3})}}\bigg|^R\;
\frac{\big|\sum_{i'=0}^{T_N N^{2/3}} \aa\big(X^N_r(i'N^{-2/3})\big)\big|^{jR}}{(j!)^R\,(2\sqrt{N})^{jR}}
\Bigg],
\end{split}
\end{equation}
$r=1,2,\ldots,R$.

\medskip

Next, we recall that $f(x)=O(\exp(x^{1-\delta}))$, $g(x)=O(\exp(x^{1-\delta}))$ as
$x\to\infty$ for some $\delta>0$ by assumption. With that $\delta$ and any fixed
$\epsilon\in\big(0,\frac{2\delta}{3}\big)$, we first consider the case that
$x_r,y_r$ in \eqref{eq_tail_rv_moment_reduced} are both at most $N^{2/3-\eps}$. In
this case, we estimate $\Xi(x_r,y_r;N,T_N)$ using the inequality \eqref{binom_bound}
of Lemma \ref{Lemma_Binomial_bound}, bound the expectation with respect to the
$\xi$'s and $\aa$'s as in the derivation of \eqref{eq_tail_estimate_1} in the proof
of Lemma \ref{Lemma_theorem_in_moments}, and employ Assumption \ref{Assumptions}
(a). All in all, we obtain the following upper bound on the expectation in
\eqref{eq_tail_rv_moment_reduced}:
\begin{equation} \label{eq_tail_estimate_2}
\begin{split}
\frac{C}{2^{jR}}\,|f(x_r)|^R |g(y_r)|^R\,e^{-\frac{R}{C}(x_r-y_r)^2}
\E\Bigg[\exp\Bigg(-\frac{R}{2N}\sum_{i=1}^{T_N N^{\frac{2}{3}}} \big(N-X^N_r(iN^{-\frac{2}{3}})\big) \\
+\, C\sum_{h\in N^{-1/3}(\zz+1/2)} \bigg(
N^{-1/3}\,\frac{R\,L_h\big(X^N_r\big)}{N^{1/6}} +
\frac{R^2 L_h\big(X^N_r\big)^2}{N^{1/3}}
+\frac{R^{\gamma'} L_h\big(X^N_r\big)^{\gamma'}}{N^{\gamma'/6}} \bigg) \\
+\, C\sum_{h\in N^{-1/3}\zz} \bigg(
N^{-1/3}\,\frac{R\,L_h\big(X^N_r\big)}{N^{1/6}} +
\frac{R^2 L_h\big(X^N_r\big)^2}{N^{1/3}}+\frac{R^{\gamma'} L_h\big(X^N_r\big)^{\gamma'}}{N^{\gamma'/6}} \bigg)\Bigg)\Bigg],
\end{split}
\end{equation}
where $C>0$ is a uniform constant and $2\le \gamma'<3$ is the same as in Lemma \ref{Lemma_tail_estimates}. An application of the Cauchy-Schwarz inequality to the latter expectation results further in the two factors
\begin{equation}\label{eq_area_exponent}
\E\bigg[\exp\bigg(-\frac{R}{N}\,\sum_{i=1}^{T_N N^{2/3}}
\big(N-X^N_r(iN^{-2/3})\big)\bigg)\bigg]^{1/2},
\end{equation}
\begin{equation} \label{eq_local_times_exponent}
\begin{split}
\E\Bigg[\exp\Bigg(& 2C\sum_{h\in N^{-1/3}(\zz+1/2)} \bigg(
N^{-1/3}\,\frac{R\,L_h\big(X^N_r\big)}{N^{1/6}} +
\frac{R^2 L_h\big(X^N_r\big)^2}{N^{1/3}}+\frac{R^{\gamma'} L_h\big(X^N_r\big)^{\gamma'}}{N^{\gamma'/6}} \bigg) \\
& +2C\sum_{h\in N^{-1/3}\zz} \bigg(
N^{-1/3}\,\frac{R\,L_h\big(X^N_r\big)}{N^{1/6}} +
\frac{R^2 L_h\big(X^N_r\big)^2}{N^{1/3}}+\frac{R^{\gamma'} L_h\big(X^N_r\big)^{\gamma'}}{N^{\gamma'/6}} \bigg)\Bigg)\Bigg]^{1/2}.
\end{split}
\end{equation}

\smallskip

To estimate the expression in \eqref{eq_area_exponent} further we assume first that $x_r \le y_r$. In this case, the random walk bridge $X^N_r$ can be  sampled as follows: sample a random walk bridge connecting $N$ to $N$ in $T_N N^{2/3}-\lfloor N-N^{1/3} x_r\rfloor +\lfloor N-N^{1/3} y_r\rfloor$ steps of size $\pm 1$; next, shift the resulting bridge down by $N-\lfloor N-N^{1/3} x_r\rfloor$; finally, insert $\lfloor N-N^{1/3} x_r \rfloor -\lfloor N-N^{1/3} y_r\rfloor$ additional down steps of size $-1$ uniformly at random. Combining the observation that the latter insertion is only increasing the value of $\sum_{i=1}^{T_N N^{2/3}} \big(N-X^N_r(iN^{-2/3})\big)$ with Proposition
\ref{Proposition_exponential_bound_1} for the random walk bridge connecting $N$ to $N$ we obtain the bound
\begin{equation*}
\E\bigg[\exp\bigg(-\frac{R}{N}\,\sum_{i=1}^{T_N N^{2/3}}
\big(N-X^N_r(iN^{-2/3})\big)\bigg)\bigg]^{1/2}
\le \tilde{C}e^{-RT_Nx_r/2},
\end{equation*}
where $\tilde{C}>0$ is a constant depending only on $R$. A similar argument in the case $x_r>y_r$ leads to the same upper bound, only with $x_r$ replaced by $y_r$.

\medskip

We turn to the expression in \eqref{eq_local_times_exponent}. Our goal is to give bounds on the exponential moments of the six random variables
\begin{equation}\label{eq_three_vars_to_bound}
\begin{split}
& \sum_{h\in N^{-1/3}(\zz+1/2)} \frac{L_h\big(X^N_r\big)}{N^{1/2}},
\quad \sum_{h\in N^{-1/3}(\zz+1/2)} \frac{L_h\big(X^N_r\big)^2}{N^{1/3}},
\quad \sum_{h\in N^{-1/3}(\zz+1/2)} \frac{L_h\big(X^N_r\big)^{\gamma'}}{N^{\gamma'/6}}, \\
& \sum_{h\in N^{-1/3}\zz} \frac{L_h\big(X^N_r\big)}{N^{1/2}},
\quad \sum_{h\in N^{-1/3}\zz} \frac{L_h\big(X^N_r\big)^2}{N^{1/3}},
\quad \sum_{h\in N^{-1/3}\zz} \frac{L_h\big(X^N_r\big)^{\gamma'}}{N^{\gamma'/6}},
\end{split}
\end{equation}
which can be then combined by means of H\"older's inequality. The first and fourth random variables simply equal to $T_N N^{-1/6}$ and $T_N N^{-1/6}+N^{-5/6}$, respectively. For the other four random variables in \eqref{eq_three_vars_to_bound} we use the estimate \eqref{SILTubd} (recalling also Remark \ref{rmk_loc_time} and the elementary inequality \eqref{eq_powers_ineq}) to obtain
\begin{equation*}
\max\bigg(\!\sum_{h\in N^{-1/3}(\zz+1/2)} \frac{L_h\big(X^N_r\big)^p}{N^{1/3}},\!\sum_{h\in N^{-1/3}\zz} \frac{L_h\big(X^N_r\big)^p}{N^{1/3}}\bigg)
\!\le\! C\Big(\!M(N,T_N)^{p-1}+|x_r-y_r|^{p-1}+1\Big),
\end{equation*}
where $p\in\{2,\gamma'\}$, $C>0$ is a uniform constant, and $N^{1/3}M(N,T_N)$ has the same law as the difference between the maximum and the minimum of $X^N_r$ (see the explanation following \eqref{SILTubd}). We also recall from the paragraph after \eqref{SILTubd} that $N^{1/3}M(N,T_N)$ is stochastically dominated by the sum of $|x_r-y_r|$ and the difference between the maximum and the minimum of a simple symmeteric random walk with $T_N N^{2/3}$ steps of size $\pm 1$. At this point, the bound \eqref{eq_SRW_max_bound} shows that
\begin{equation*}
\begin{split}
\max\bigg(\E\bigg[\exp\bigg(\theta\,\sum_{h\in N^{-1/3}(\zz+1/2)} \frac{L_h\big(X^N_r\big)^p}{N^{1/3}}\bigg)\bigg],\,
\E\bigg[\exp\bigg(\theta\,\sum_{h\in N^{-1/3}\zz} \frac{L_h\big(X^N_r\big)^p}{N^{1/3}}\bigg)\bigg]\bigg) \\
\le \tilde{C}\,e^{\tilde{C}\,|x_r-y_r|^{p-1}},
\end{split}
\end{equation*}
where $\theta>0$, $p\in\{2,\gamma'\}$, and $\tilde{C}>0$ is a constant depending only on $\theta$.

\medskip

Putting everything together we get the following estimate on the expression in \eqref{eq_tail_estimate_2}:
\begin{equation} \label{eq_tail_estimate_3}
\frac{\tilde{C}}{2^{jR}}\,|f(x_r)|^R |g(y_r)|^R
\exp\bigg(\!-\frac{R}{C}\,(x_r-y_r)^2-\frac{R\,T_N}{2}\min(x_r,y_r) + \tilde{C}|x_r-y_r|+ \tilde{C}|x_r-y_r|^{\gamma'-1}\!\bigg),
\end{equation}
where $\tilde{C}>0$ is a constant depending only on $R$, and $C>0$ is a uniform constant. Since $f(x)=O(\exp(x^{1-\delta}))$, $g(x)=O(\exp(x^{1-\delta}))$ as $x\to\infty$, and $2<\gamma'<3$, the integral of the $R$-th root of the latter expression (and, hence, also of the expression in  \eqref{eq_tail_rv_moment_reduced}) over the region $\{(x_r,y_r)\in{\mathcal R}_{N,K}:\,x_r,y_r\le N^{2/3-\epsilon}\}$ admits an upper bound of the form $\frac{C(R,K)}{2^j}$ with $\lim_{K\to\infty} C(R,K)=0$ for any fixed $R$.

\medskip

To finish the proof it remains to treat the case of at least one of $x_r,y_r$ is larger than $N^{2/3-\epsilon}$. The additional technical difficulty in this case is that the quantities $\sqrt{X^N_r(iN^{-2/3})\wedge X^N_r((i-1)N^{-2/3})}$ can be anywhere between $1$ and $\sqrt{N-N^{1-\epsilon}+T_NN^{2/3}}\le\sqrt{N-N^{1-\epsilon}/2}$ for all $N$ sufficiently large. To address this issue we estimate the factors involving $X^N_r(iN^{-2/3})\wedge X^N_r((i-1)N^{-2/3})$ in \eqref{eq_tail_rv_moment_reduced} as follows:
\begin{equation*}
\begin{split}
&N^{-R/2}\,\Big|\sqrt{X^N_r(iN^{-2/3})\wedge X^N_r((i-1)N^{-2/3})}+\xi\big(X^N_r(iN^{-2/3})\wedge X^N_r((i-1)N^{-2/3})\big)\Big|^R \\
&\le N^{-R/2}
\Big(\sqrt{N-N^{1-\eps}/2}+\big|\xi\big(X^N_r(iN^{-2/3})\wedge X^N_r((i-1)N^{-2/3})\big)\big|\Big)^R \\
&=\big(1-N^{-\eps}/2\big)^{R/2}\bigg(1+\frac{\big|\xi\big(X^N_r(iN^{-2/3})\wedge X^N_r((i-1)N^{-2/3})\big)\big|}{\sqrt{N-N^{1-\eps}/2}}\bigg)^R.
\end{split}
\end{equation*}
Inserting the latter bound into \eqref{eq_tail_rv_moment_reduced} and proceeding as in the derivation of \eqref{eq_tail_estimate_2} (only replacing all $\xi$'s by their absolute values) we obtain an estimate on the expression in \eqref{eq_tail_rv_moment_reduced} of the form
\begin{equation} \label{eq_tail_estimate_4}
\begin{split}
& \frac{C}{2^{jR}}\,|f(x_r)|^R |g(y_r)|^R\,\exp\Big(-\frac{R}{C}(x_r-y_r)^2-\frac{RT_N}{4}\,N^{2/3-\epsilon}\Big) \\
& \E\Bigg[\exp\Bigg(C\,\sum_{h\in N^{-1/3}(\zz+1/2)} \bigg(\frac{R\,L_h(X^N_r)}{N^{1/6}}
+\frac{R^2 L_h(X^N_r)^2}{N^{1/3}}
+\frac{R^{\gamma'} L_h(X^N_r)}{N^{\gamma'/6}}\bigg) \\
& \qquad\quad\;\;\; +C\,\sum_{h\in N^{-1/3}\zz} \bigg(\frac{R\,L_h(X^N_r)}{N^{1/6}}
+\frac{R^2 L_h(X^N_r)^2}{N^{1/3}}+\frac{R^{\gamma'} L_h(X^N_r)}{N^{\gamma'/6}}\bigg)\Bigg)\Bigg],
\end{split}
\end{equation}
where $C>0$ is a uniform constant and $2\le\gamma'<3$ is the same as in Lemma \ref{Lemma_tail_estimates}. Repeating the argument leading to  \eqref{eq_tail_estimate_3} we arrive at the further upper bound
\begin{equation} \label{eq_tail_estimate_5}
\frac{\tilde{C}}{2^{jR}}|f(x_r)|^R|g(y_r)|^R
\!\exp\!\bigg(\!\!-\frac{R}{C}(x_r-y_r)^2-\frac{RT_N}{4}N^{2/3-\eps}
+\tilde{C}\!\Big(\!T_NN^{1/6}+|x_r-y_r|+|x_r-y_r|^{\gamma'-1}\Big)\!\!\bigg)
\end{equation}
with the constant $C$ of \eqref{eq_tail_estimate_4} and a constant $\tilde{C}>0$ depending only on $R$. Since $f(x)=O(\exp(x^{1-\delta}))$, $g(x)=O(\exp(x^{1-\delta}))$ as $x\to\infty$, and $2<\gamma'<3$, the term $\exp\big(\!\!-\frac{R}{C}(x_r-y_r)^2-\frac{RT_N}{4}N^{2/3-\epsilon}\big)$ dominates on the region of integration $\{(x_r,y_r)\in{\mathcal R}_{N,K}:\,x_r>N^{2/3-\epsilon}\;\mathrm{or}\;y_r>N^{2/3-\epsilon}\}$. The area of the latter is of the order $N^{4/3}$, so that the desired estimate readily follows.
\end{proof}

\begin{rmk} \label{Remark_K_limit}
Lemma \ref{Lemma_moment_bound} shows that the differences between the integrals in \eqref{eq_high_moment_term}, \eqref{eq_trace_term} with $\underline R=-\infty$, $\overline R=\infty$ and a fixed $K$ and those with $\underline R=-\infty$, $\overline R=\infty$ and $K=N^{2/3}$ are bounded uniformly in $N\in\nn$ and tend to $0$ in the limit $K\to\infty$ in the sense of moments. This statement is also true for the sums of such differences over all $j=0,1,\ldots$ (which give rise to exponential functions), since the bounds of Lemma \ref{Lemma_moment_bound} decay exponentially in $j$.
\end{rmk}

Next, we present an extension of Lemma \ref{Lemma_moment_bound} which can be shown in the same way.

\begin{lemma} \label{Lemma_moment_bound_upgrade}
Fix an $N\in\nn$ and let $\Upsilon$ be a random variable given by a deterministic function of the random walk bridge $X$ and the sequence of $\aa$'s. For each $K\in[0,\infty)$, define $\overline{\mathrm{Sc}}^{(j)}(N;K;\Upsilon)$ as the modification of $\overline{\mathrm{Sc}}^{(j)}(N;K)$ obtained by inserting $\Upsilon$ as an additional factor into the expectation in \eqref{eq_high_moment_term}. Suppose that the $(2R)$-th moment of $\Upsilon$ is bounded by $C_1<\infty$. Then,
there exist positive constants $C(R,K)$, $K\in[0,\infty)$ (possibly depending on $f$, $g$, $\A$, but not on $N$, $\Upsilon$) such that
\begin{equation*}
\E\Big[\big|\overline{\mathrm{Sc}}^{(j)}(N;K;\Upsilon)\big|^R\Big] \le \frac{C(R,K)}{2^{jR}}\,C_1^{1/2}, \qquad
\E\Big[\big|\overline{\mathrm{Tr}}^{(j)}(N;K;\Upsilon)\big|^R\Big]
\le \frac{C(R,K)}{2^{jR}}\,C_1^{1/2}
\end{equation*}
for all $N\in\nn$, and $\lim_{K\to\infty} C(R,K)=0$.
\end{lemma}

\smallskip

As a last ingredient for the proof of Theorem \ref{theorem_Main_restated} we need to control the error in replacing the terms involving $\aa$'s in the fourth lines of \eqref{eq_high_moment_2}, \eqref{eq_trace_approx} by the terms involving $\aa$'s in the third lines of \eqref{eq_high_moment_term}, \eqref{eq_trace_term}. To this end, we consider the complete homogeneous symmetric functions and power sums given respectively by
\begin{equation*}
h(j;N):= \sum_{0\le i_1 \le\cdots\le i_j\le T_N N^{2/3}} \, \prod_{j'=1}^j \aa\big(X(i_{j'}N^{-2/3})\big), \quad
p(j;N):=\sum_{i'=0}^{T_NN^{2/3}} \aa\big(X(i'N^{-2/3})\big)^j.
\end{equation*}
The $h(j;N)$'s can be expressed in terms of the $p(j;N)$'s according to the Newton identies. The latter can be written as follows (see e.g. \cite[Chapter 1, Section 2]{Mac}):
\begin{equation} \label{eq_Newton}
h(j;N)=[z^j]\exp\bigg(\sum_{j'=1}^\infty \frac{p(j';N)}{j'} z^{j'}\bigg),
\end{equation}
where $[z^j]$ stands for the coefficient of $z^j$ in the series expansion of what follows. We note further that \eqref{eq_Newton} implies
\begin{equation} \label{eq_sum_arising_for_fixed}
 \frac{h(j;N)}{(2\sqrt{N})^j}=\sum_{l=0}^j \frac{p(1;N)^l}{l!\, (2\sqrt{N})^l} \bigg([z^{j-l}]
\exp\bigg(\sum_{j'=2}^\infty \frac{p(j';N)}{j'(2\sqrt{N})^{j'}}\,z^{j'}\bigg)\!\!\bigg).
\end{equation}
The $l=j$ term in \eqref{eq_sum_arising_for_fixed} is precisely the one appearing in the third lines of \eqref{eq_high_moment_term}, \eqref{eq_trace_term}. The next lemma shows that the other terms become negligible as $N\to\infty$.

\begin{lemma} \label{Lemma_tail_sum}
Under Assumption \ref{Assumptions} and for each $R=1,2,\ldots$, there exist positive constants $\tilde{C}$, $\eps$ such that
\begin{equation} \label{eq_a_tail_estimate}
\E\bigg[\bigg([z^l]\exp\bigg(\sum_{j'=2}^\infty \frac{|p(j';N)|}{j'
(2\sqrt{N})^{j'}}\,z^{j'}\bigg)\!\!\bigg)^R\,\bigg] \le \big(\tilde{C} N^{-\eps}\big)^{lR}
\end{equation}
for all $l=1,2,\ldots,T_NN^{2/3}$ and $N\in\nn$.
\end{lemma}

\begin{proof}
We define $\Gamma_j=\sup_{0\le i_1\le i_2\le\cdots\le i_j} \E\big[ \prod_{j'=1}^j |\aa(i_{j'})|\big]=\sup_{i\in\nn} \E\big[|\aa(i)|^j\big]$, where the second equality is due to H\"older's inequality. By Assumption \ref{Assumptions}, $\Gamma_j\le C^j j^{j\gamma}$ with universal constants $C>0$ and $0<\gamma<3/4$. We also write $k$ for  $T_NN^{2/3}+1$. Now, expanding the random variable inside the expectation in \eqref{eq_a_tail_estimate} into a sum of monomials in the $|\aa\big(X(i'N^{-2/3})\big)|$'s, bounding the expectations of the latter by $\Gamma_{lR}$, and then collecting the terms back we obtain the following upper bound on the left-hand side of
\eqref{eq_a_tail_estimate}:
\begin{equation}
\frac{1}{(2\sqrt{N})^{lR}}\,\Gamma_{lR}\,\bigg([z^l]\exp\bigg(\sum_{j'=2}^{\infty} \frac{k}{j'}\,z^{j'}\bigg)\bigg)^R.
\end{equation}
Using $l\le k$ we deduce further
\begin{equation*}
\begin{split}
& [z^l]\exp\bigg(\sum_{j'=2}^\infty \frac{k}{j'}\,z^{j'}\bigg)
= [z^l]\Bigg(\sum_{l'=0}^{\lfloor l/2\rfloor} \frac{\big(\sum_{j'=2}^l \frac{k}{j'}\,z^{j'}\big)^{l'}}{l'!}\Bigg)
\le \sum_{l'=0}^{\lfloor l/2\rfloor} \frac{\big(\sum_{j'=2}^l \frac{k}{j'}\big)^{l'}}{l'!}
\le \sum_{l'=0}^{\lfloor l/2\rfloor} \frac{(k\log l)^{l'}}{l'!} \\
&  \le \big(\lfloor l/2\rfloor+1\big)\frac{(k\log l)^{\lfloor l/2\rfloor}}{\lfloor l/2\rfloor!}\le C^l\,\bigg(\frac{k\log l}{l}\bigg)^{\lfloor l/2\rfloor},
\end{split}
\end{equation*}
where $C>0$ is a uniform constant. Putting everything together we find that the left-hand side of \eqref{eq_a_tail_estimate} is at most
\begin{equation*}
\tilde{C}^{lR}\,\frac{1}{(2\sqrt{N})^{lR}}\,l^{lR\gamma}\,\bigg(\frac{N^{\frac{2}{3}}\log l}{l}\bigg)^{\lfloor l/2\rfloor R}
\le\bigg(\frac{\tilde{C}\,l^{2\gamma-1}\,N^{\frac{2}{3}}\,\log l}{4N}\bigg)^{lR/2}
\le \bigg(\frac{\tilde{C}\,N^{\frac{4\gamma}{3}}\,\log N^{\frac{2}{3}}}{4N}\bigg)^{lR/2},
\end{equation*}
where $\tilde{C}$ is a constand depending only on $R$, $\gamma$, and $\sup_N T_N$ that varies from expression to expression. The lemma now follows from $\gamma<3/4$ (see Assumption \ref{Assumptions}).
\end{proof}


\subsection{Proof of Theorem \ref{Theorem_Main}}\label{Section_complete_proof}

As discussed in Subsection \ref{subsec_leading}, to prove Theorem \ref{Theorem_Main} it suffices to establish Theorem \ref{theorem_Main_restated}. In the following, we only give the proof of the latter for $\mathrm{Sc}(N)$, since the same proof applies to $\mathrm{Tr}(N)$ as well. To start with, we note
\begin{equation} \label{eq_Sc_expansion}
\mathrm{Sc}(N)=\sum_{j=0}^{\lfloor TN^{2/3}\rfloor} \mathrm{Sc}^{(j)}(N;K,-\infty,\infty) + \sum_{j=0}^{\lfloor TN^{2/3}\rfloor} \overline{\mathrm{Sc}}^{(j)}(N;K)
 + \sum_{j=0}^{\lfloor TN^{2/3}\rfloor} \sum_{l=0}^{j-1} U(N;j,l),
\end{equation}
where
\begin{equation}
\begin{split}
& U(N;j,l):=\frac{1}{2}\int_{0}^\infty \int_{0}^\infty f(x)\,
g(y)\,\Xi(x,y;N,T_N)\,\E_X
\Bigg[\mathbf{1}_{\{\forall t:\,N^{-1/3}(N-X(t))\in{\mathcal A}\}}   \\
&\prod_{i=1}^{T_NN^{2/3}} \frac{\sqrt{X(iN^{-2/3})\wedge
X((i-1)N^{-2/3})}}{\sqrt{N}}\,\bigg(1+\frac{\xi\big(X(iN^{-2/3})\wedge X((i-1)N^{-2/3})\big)}{\sqrt{X(iN^{-2/3})\wedge X((i-1)N^{-2/3})}}\bigg)
\\
&\frac{\Big(\sum_{i'=0}^{T_NN^{2/3}} \aa\big(X(i'N^{-2/3})\big)\Big)^l}{l!\, (2\sqrt{N})^l} \; [z^{j-l}]
\exp\Bigg(\sum_{j'=2}^\infty \,\frac{\sum_{i'=0}^{T_NN^{2/3}} \aa\big(X(i'N^{-2/3})\big)^{j'}}{j'(2\sqrt{N})^{j'}} z^{j'}\Bigg)
\Bigg]\,\mathrm{d}x\,\mathrm{d}y.
\end{split}
\end{equation}
Hereby, with a slight abuse of notation we have allowed the value of $T_N$ to change from term to term with $j$, $l$. However, this is not important, since the results of the previous subsection apply to all occuring values of $T_N$.

\medskip

Now, we take the limits $N\to\infty$, $K\to\infty$ (in this particular order) of the right-hand side in \eqref{eq_Sc_expansion}.
The limits as $N\to\infty$ of the summands in the first sum in \eqref{eq_Sc_expansion} have been identified in  Proposition \ref{Proposition_main_single_term}, so that by invoking the moment bounds of Lemma \ref{Lemma_moment_bound} with $K=0$ we conclude that their sum converges as $N\to\infty$ to
\begin{equation} \label{eq_finite_K_limit}
\begin{split}
& \frac{1}{\sqrt{2\pi T}} \int_0^K \int_0^K f(x)\,
g(y)\,\exp\bigg(-\frac{(x-y)^2}{2T}\bigg) \\
& \E_{B^{x,y}}\bigg[\mathbf{1}_{\{B^{x,y}(t)\in{\mathcal A}\}}
\exp\Big(-\frac{1}{2}\int_0^T B^{x,y}(t)\,\mathrm{d}t
+\sqrt{s_\xi^2+s_\aa^2/4}\int_0^\infty L_a\big(B^{x,y}\big)\,\mathrm{d}W(a)\Big)\bigg]\,\mathrm{d}x\,\mathrm{d}y
\end{split}
\end{equation}
in the sense of moments and in distribution, where $W=\frac{1}{\sqrt{s_\xi^2+s_\aa^2/4}}(s_\xi\,W_\xi+s_\aa\,W_\aa)$. In view of Remark \ref{Remark_K_limit}, the $K\to\infty$ limit of the expression in \eqref{eq_finite_K_limit} in the sense of moments and in distribution is given by the same expression with $K$ replaced by $\infty$. The latter is precisely the limit in \eqref{eq_Scalar_limit} (recall Assumption \ref{Assumptions} (b)).

\medskip

Next, for every fixed $R\in\nn$, we combine Lemma
\ref{Lemma_moment_bound} with the triangle inequality for the $L^R$-norm to conclude that the $R$-th moment of the second sum in \eqref{eq_Sc_expansion} can be bounded above by a constant $C(R,K)$ uniformly in $N\in\nn$, which further satisfies $\lim_{K\to\infty} C(R,K)=0$. Consequently, the second sum in \eqref{eq_Sc_expansion} vanishes in the double limit $N\to\infty$, $K\to\infty$ in the sense of moments and in distribution.

\medskip

To control the third sum in \eqref{eq_Sc_expansion} we apply Lemma \ref{Lemma_moment_bound_upgrade} for $U(N;j,l)$ and a fixed $R\in\nn$
(note that in the case at hand the bound on the $R$-th moment of $\Upsilon$ required in Lemma \ref{Lemma_moment_bound_upgrade} is precisely the content of Lemma \ref{Lemma_tail_sum}) to get
\begin{equation}
\label{eq_U_bound}
\E\big[U(N;j,l)^R\big]\le \frac{\tilde{C}_1}{2^{lR}}\big(\tilde{C}_2 N^{-\eps}\big)^{(j-l)R},\quad j,l\in\nn,\;j\ge l,
\end{equation}
where the positive constants $\tilde{C}_1$, $\tilde{C}_2$, and $\epsilon$ may depend on $R$. Combining this with the triangle inequality for the $L^R$-norm we deduce that the $R$-th moment of the third sum in $\eqref{eq_Sc_expansion}$ tends to $0$ in the limit $N\to\infty$. Since $R\in\nn$ was arbitrary, the third sum in
\eqref{eq_Sc_expansion} must vanish as $N\to\infty$ both in the sense of moments and in distribution. \ep

\begin{rmk} \label{Remark_even_and_odd}
Note that the terms $\mathrm{Sc}^{(j)}(N;K,-\infty,\infty)$ in \eqref{eq_Sc_expansion} with even $j$ (odd $j$ resp.) arise when we take even (odd resp.) powers of the matrix $\frac{M_{N;{\mathcal A}}}{2\sqrt{N}}$. Since Proposition \ref{Proposition_main_single_term} gives the joint asymptotics of such terms for \textit{any} finite collection of $j$'s, we can separate the even and the odd powers in Theorem \ref{Theorem_Main}. By doing so, we obtain the following results. Let $T_N$, $N\in\nn$ and $\tilde{T}_N$, $N\in\nn$ be two sequences of positive numbers such that $T_NN^{2/3}$, $N\in\nn$ are even integers and $\tilde{T}_NN^{2/3}$, $N\in\nn$ are odd integers. Suppose further that $\sup_N |T_N-T|N^{2/3}<\infty$ and $\sup_N |\tilde{T}_N-T|N^{2/3}<\infty$. Then, under the same assumptions and in the same sense as in Theorem \ref{Theorem_Main} the limit
\begin{equation*}
\lim_{N\to\infty}\;
\frac{1}{2}\bigg(\frac{M_{N;{\mathcal A}}}{2\sqrt N}\bigg)^{T_NN^{2/3}}
\end{equation*}
is given by the integral operator with the kernel
\begin{equation} \label{eq_even_part}
\begin{split}
& K_{\A}^{\rm even}(x,y;T):=
 \frac{1}{\sqrt{2\pi T}}
\exp\bigg(-\frac{(x-y)^2}{2T}\bigg)\,\E_{B^{x,y}}\Bigg[\mathbf{1}_{\{\forall t:\,B^{x,y}(t)\in{\mathcal A}\}}\\
& \exp\bigg(-\frac{1}{2}\int_0^T B^{x,y}(t)\,\mathrm{d}t
+s_\xi\int_0^\infty L_a\big(B^{x,y}\big)\,\mathrm{d}W_\xi(a)\bigg)
\!\sum_{j\,\mathrm{even}} \frac{\big(s_\aa\int_0^\infty L_a\big(B^{x,y}\big)\,\mathrm{d}W_\aa(a)\big)^j}{2^j j!} \Biggr].
\end{split}
\end{equation}
Similarly, the $N\to\infty$ limit of $\frac{1}{2}\bigg(\frac{M_{N;{\mathcal A}}}{2\sqrt N}\bigg)^{\tilde{T}_NN^{2/3}}$ is the integral operator with the kernel
\begin{equation} \label{eq_even_part}
\begin{split}
& K_{\A}^{\rm even}(x,y;T):=
 \frac{1}{\sqrt{2\pi T}}
\exp\bigg(-\frac{(x-y)^2}{2T}\bigg)\,\E_{B^{x,y}}\Bigg[\mathbf{1}_{\{\forall t:\,B^{x,y}(t)\in{\mathcal A}\}}\\
& \exp\bigg(-\frac{1}{2}\int_0^T B^{x,y}(t)\,\mathrm{d}t
+s_\xi\int_0^\infty L_a\big(B^{x,y}\big)\,\mathrm{d}W_\xi(a)\bigg)
\!\sum_{j\,\mathrm{odd}} \frac{\big(s_\aa\int_0^\infty L_a\big(B^{x,y}\big)\,\mathrm{d}W_\aa(a)\big)^j}{2^j j!} \Biggr].
\end{split}
\end{equation}
Here, $W_\xi$ and $W_\aa$ are independent standard Brownian motions as in Proposition \ref{Proposition_main_single_term}.
\end{rmk}



\section{Properties of the stochastic Airy semigroup I}
\label{Section_properties}

In this section, we study the properties of operators $\U_\A(T)$, $T>0$ in more detail. The present section contains the statements that we prove directly from the definition of $\U_\A(T)$, $T>0$, while in Sections \ref{Section_extreme_convergence},  \ref{Section_properties_2} we prove the remaining properties by relying on the asymptotic results of Section \ref{Section_estimates}.

\begin{lemma}\label{Lemma_HS}
For each $T>0$, $\U_\A(T)$ is a Hilbert-Schmidt operator on $L^2(\rr_{\ge0})$ with probability one.
\end{lemma}

\begin{proof}
Fix a $T>0$. We need to prove that the integral kernel $K_\A(x,y;T)$ of $\U_\A(T)$ satisfies
\begin{equation}\label{eq_HS}
\int_{\rr_{\ge0}} \int_{\rr_{\ge0}} K_\A(x,y;T)^2\,\mathrm{d}x\,\mathrm{d}y < \infty
\end{equation}
with probability one. To this end, we will show that the expectation of the latter
double integral is finite. Moving the square inside the expectation in the
definition of $K_\A(x,y;T)^2$, dropping the indicator function therein, applying
Fubini's Theorem to take the expectation with respect to $W$, and observing that the
latter expectation boils down to the exponential moment of a Gaussian random
variable, we obtain  the following upper bound on the expectation of \eqref{eq_HS}:
\begin{equation}\label{eq_HS2}
\frac{1}{2\pi T}\,\int_{\rr_{\ge0}} \int_{\rr_{\ge0}} e^{-\frac{(x-y)^2}{T}}\,\E_{B^{x,y}}\bigg[
\exp\bigg(-\int_0^T B^{x,y}(t)\,\mathrm{d}t+\frac{2}{\beta}\int_0^\infty L_a(B^{x,y})^2\,\mathrm{d}a\bigg)\bigg]\,\mathrm{d}x\,\mathrm{d}y.
\end{equation}
Next, we note that $B^{x,y}(t)-\big(1-\frac{t}{T}\big)x-\frac{t}{T}\,y$, $t\in[0,T]$
is a copy of $B^{0,0}$, as well as
$\big(1-\frac{t}{T}\big)x+\frac{t}{T}\,y\ge\min(x,y)$. These and the Cauchy-Schwarz
inequality allow to estimate the integrand in \eqref{eq_HS2} by the product of
$e^{-\frac{(x-y)^2}{T}-\min(x,y)}$ with
\begin{equation}\label{eq_HS3}
\E_{B^{0,0}}\bigg[\exp\bigg(-2\int_0^T B^{0,0}(t)\,\mathrm{d}t\bigg)\bigg]^{1/2}\E_{B^{x,y}}\bigg[\exp\bigg(\frac{4}{\beta}\int_0^\infty L_a(B^{x,y})^2\,\mathrm{d}a\bigg)\bigg]^{1/2}.
\end{equation}
Further, it is well-known that the first expectation in \eqref{eq_HS3} is given by a \textit{finite} constant $C$, see e.g. \cite[proof of Theorem 3.1, final display]{CSY} for a significantly stronger statement. For the second expectation in \eqref{eq_HS3}, we recall from the proof of Proposition \ref{Proposition_basic_conv} that $\int_0^\infty L_a(B^{x,y})^2\,\mathrm{d}a$ is the limit in distribution of $N^{-1/3}\sum_{h\in N^{-1/3}(\zz_{\ge0}+1/2)} L_h(X^{x,y;N,T_N})^2$, $N\in\nn$. Moreover, the latter random variables are stochastically dominated by $8T\big(N^{-1/3}J_{TN^{2/3}}+N^{-1/3}\tilde{J}_{TN^{2/3}}+2|x-y|+2N^{-1/3}\big)$, $N\in\nn$, respectively, where $J_{TN^{2/3}}$, $\tilde{J}_{TN^{2/3}}$ are the maxima of two independent simple random walks with $TN^{2/3}$ steps of size $\pm 1$ (see \eqref{SILTubd} and the paragraph following it). Passing to the limit $N\to\infty$ we conclude that $\int_0^\infty L_a(B^{x,y})^2\,\mathrm{d}a$ is stochastically dominated by $8T(J+\tilde{J}+2|x-y|)$, where $J$, $\tilde{J}$ are the maxima of two independent standard Brownian motions on $[0,T]$. All in all, the expression in \eqref{eq_HS2} is bounded above by
\begin{equation*}
\frac{C^{1/2}}{2\pi T} \int_{\rr_{\ge0}} \int_{\rr_{\ge0}} e^{-\frac{(x-y)^2}{T}-\min(x,y)}\,\E\bigg[\exp\bigg(\frac{32T}{\beta}\big(J+\tilde{J}+2|x-y|\big)\bigg)\bigg]^{1/2}\,\mathrm{d}x\,\mathrm{d}y.
\end{equation*}
It remains to note the integrability of
$e^{-\frac{(x-y)^2}{T}-\min(x,y)+\frac{32T}{\beta}|x-y|}$ over
$\big(\rr_{\ge0}\big)^2$.
\end{proof}

Lemma \ref{Lemma_HS} shows that the operators $\U_\A(T)$ are well-defined, and we
can now prove Propositions \ref{Proposition_trace_class},
\ref{Proposition_semigroup}, and \ref{Proposition_continuous}.

\begin{proof}[Proof of Proposition \ref{Proposition_semigroup}]
We need to show that the event
\begin{equation*}
\{\forall\,f\in L^2(\rr_{\ge0}):\,\U_\A(T_1)\,\U_\A(T_2)f=\U_\A(T_1+T_2)f\}
\end{equation*}
has probability one for any $T_1,T_2\ge0$. Moreover, according to Lemma
\ref{Lemma_HS}, the operators $\U_\A(T_1)$, $\U_\A(T_2)$, $\U_\A(T_1+T_2)$ are
continuous with probability one, so that in the latter event we may replace
$L^2(\rr_{\ge0})$ by a countable dense subset. Consequently, it suffices to prove
that $\U_\A(T_1)\,\U_\A(T_2)f=\U_\A(T_1+T_2)f$ almost surely for a \textit{fixed}
$f\in L^2(\rr_{\ge0})$. In addition, we may assume that $f\ge0$, since otherwise we
can decompose $f$ into its positive and negative parts. In this case, Fubini's
Theorem reveals that it is enough to show that almost surely
\begin{equation}\label{eq_kernel_semi}
\forall\,x,y\in\rr_{\ge0}:\quad \int_{\rr_{\ge0}} K_\A(x,z;T_1)\,K_\A(z,y;T_2)\,\mathrm{d}z=K_\A(x,y;T_1+T_2).
\end{equation}
Further, elementary manipulations allow to rewrite the latter integral as
\begin{equation}\label{eq_conc_semi}
\begin{split}
\frac{1}{\sqrt{2\pi(T_1+T_2)}}\,\int_\rr \sqrt{\frac{T_1+T_2}{2\pi\, T_1 T_2}}\,\exp\bigg(-\bigg(z-\frac{T_1}{T_1+T_2}y-\frac{T_2}{T_1+T_2}x\bigg)^2\Big/\bigg(2\frac{T_1T_2}{T_1+T_2}\bigg)\bigg) \\
\E_{B^{x,z},B^{z,y}}\bigg[\!\mathbf{1}_{\{\forall t:B^{x,z}(t)\in\A,\forall t:B^{z,y}(t)\in\A\}}\!\exp\!\bigg(\!\!-\!\frac{(x-y)^2}{2(T_1+T_2)}\!-\!\frac{1}{2}\!\int_0^{T_1}\!B^{x,z}(t)\,\mathrm{d}t\!-\!\frac{1}{2}\!\int_0^{T_2}\! B^{z,y}(t)\,\mathrm{d}t\\
+\frac{1}{\sqrt{\beta}}\int_0^\infty \big(L_a(B^{x,z})+L_a(B^{z,y})\big)\,\mathrm{d}W(a)\bigg)\bigg]\,\mathrm{d}z.
\end{split}
\end{equation}
We now make the following observation: the process obtained by sampling a point $z$
according to the normal distribution with mean
$\frac{T_1}{T_1+T_2}y+\frac{T_2}{T_1+T_2}x$ and variance $\frac{T_1T_2}{T_1+T_2}$
and then concatenating a standard Brownian bridge connecting $x$ to $z$ in time
$T_1$ with a conditionally independent (given $z$) standard Brownian bridge
connecting $z$ to $y$ in time $T_2$ is a standard Brownian bridge connecting $x$ to
$y$ in time $T_1+T_2$. As a result, we end up with $K_\A(x,y;T_1+T_2)$.
\end{proof}

\begin{proof}[Proof of Proposition \ref{Proposition_trace_class}]
The almost sure symmetry of $\U_\A(T)$ amounts to showing that the event
\begin{equation*}
\bigg\{\forall\,f,g\in L^2(\rr_{\ge0}):\;\int_{\rr_{\ge0}} \big(\U_\A(T)f\big)(x) g(x)\,\mathrm{d}x=\int_{\rr_{\ge0}} f(x)\big(\U_\A(T)g\big)(x)\,\mathrm{d}x\bigg\}
\end{equation*}
has probability one. Thanks to the almost sure continuity of the operator $\U_\A(T)$
(Lemma \ref{Lemma_HS}) it is further sufficient to consider $f$, $g$ from a
countable dense subset of $L^2(\rr_{\ge0})$. Hence, the symmetry property of
$\U_\A(T)$ reduces to the almost sure equality
\begin{equation*}
\int_{\rr_{\ge0}} \big(\U_\A(T)f\big)(x) g(x)\,\mathrm{d}x=\int_{\rr_{\ge0}} f(x)\big(\U_\A(T)g\big)(x)\,\mathrm{d}x
\end{equation*}
for \textit{fixed} $f,g\in L^2(\rr_{\ge0})$. In addition, we may assume $f\ge0$,
$g\ge0$, since otherwise we can decompose $f$, $g$ into their positive and negative
parts. For such functions, we may use Fubini's Theorem to reduce the statement
further to: almost surely,
\begin{equation}
\forall\,x,y\in\rr_{\ge0}:\quad K_\A(x,y;T)=K_\A(y,x;T).
\end{equation}
The latter follows from the definition of the kernel $K_\A(\cdot,\cdot;T)$ and the
fact that the time reversal of a standard Brownian bridge connecting $x$ to $y$ in
time $T$ is a standard Brownian bridge connecting $y$ to $x$ in time $T$.

\medskip

At this point, the almost sure non-negativity of $\U_\A(T)$ is a direct consequence
of the almost sure identity $\U_\A(T)=\U_\A(T/2)\,\U_\A(T/2)$ (Proposition
\ref{Proposition_semigroup}) and the almost sure symmetry of $\U_\A(T/2)$.

\medskip

So far, we have established that $\U_\A(T)$ is a symmetric positive Hilbert-Schmidt
operator with probability one. In particular, the Spectral Theorem for symmetric
compact operators reveals that $L^2(\rr_{\ge0})$ almost surely admits an orthonormal
basis comprised of eigenfunctions of $\,\U_\A(T)$ with the corresponding nonnegative
eigenvalues $e^1_\A(T)\ge e^2_\A(T)\ge\cdots$. By definition, $\U_\A(T)$ is trace
class iff $\sum_{i=1}^\infty e^i_\A(T)<\infty$, and in this case
\begin{equation}\label{eq_U_trace1}
\mathrm{Trace}\big(\U_\A(T)\big)=\sum_{i=1}^\infty e^i_\A(T).
\end{equation}
Moreover, $\U_\A(T)\,\U_\A(T/2)=\U_\A(T/2)\,\U_\A(T)$ and \cite[Chapter 28, Theorem
7]{Lax} show that $L^2(\rr_{\ge0})$ almost surely admits an orthonormal basis
comprised of \textit{common} eigenfunctions of $\U_\A(T)$ and $\U_\A(T/2)$.
Consequently, $\sum_{i=1}^\infty e^i_\A(T)$ can be rewritten as $\sum_{i=1}^\infty
e^i_\A(T/2)^2$. The latter expression gives the square of the Hilbert-Schmidt norm
of $\U_\A(T/2)$ (see \cite[Chapter 30, Exercise 11]{Lax}) and can be therefore
rewritten further as
\begin{equation*}
\int_{\rr_{\ge0}} \int_{\rr_{\ge0}} K_\A(x,y;T/2)^2\,\mathrm{d}x\,\mathrm{d}y
\end{equation*}
(see \cite[Theorem 2.11]{Simon}). The latter expression is finite with probability one by Lemma
\ref{Lemma_HS}, so that $\U_\A(T)$ is trace class. Finally, an application of the semigroup
equation \eqref{eq_kernel_semi} to
\begin{equation*}
\int_{\rr_{\ge0}} \int_{\rr_{\ge0}} K_\A(x,y;T/2)^2\,\mathrm{d}x\,\mathrm{d}y
=\int_{\rr_{\ge0}} \int_{\rr_{\ge0}} K_\A(y,x;T/2)\,K_\A(x,y;T/2)\,\mathrm{d}x\,\mathrm{d}y
\end{equation*}
gives the trace formula \eqref{trace_formula}.
\end{proof}

\smallskip

\begin{proof}[Proof of Proposition \ref{Proposition_continuous}] 
We will show the stronger statement that, for each \emph{even} integer $p\ge 2$, it holds
\begin{equation}
\label{eq_stronger_continuity}
\lim_{t\to T} \E\big[\|\U_\A(T)f-\U_\A(t)f\|^p\big]=0.
\end{equation}
With the notation $\|\U_\A(t)\|$ for the spectral norm (i.e. the largest eigenvalue) of $\U_\A(t)$ we can use the Cauchy-Schwarz inequality to find
\begin{equation*}
\begin{split}
\E \| \U_\A(T)f-\U_\A(t)f\|^p &\le \E\big[\|\U_\A(\min(T,t))\|^p\,\|
\U_\A(|T-t|)f-f\|^p\big] \\ 
&\le \E\big[\|\U_\A(\min(T,t))\|^{2p}\big]^{1/2}\,\E\big[\|\U_\A(|T-t|)f-f\|^{2p}\big]^{1/2} \\ 
&\le \E\big[\mathrm{Trace}(\U_\A(2p\min(t,T)))\big]^{1/2}\,\E\big[\|\U_\A(|T-t|)f-f\|^{2p}\big]^{1/2}.
\end{split}
\end{equation*}
The reduction of the trace to a Hilbert-Schmidt norm as in the proof of Proposition \ref{Proposition_trace_class} and the arguments in the proof of Lemma \ref{Lemma_HS} show that the quantity $\E\big[\mathrm{Trace}(\U_\A(2p\min(t,T)))\big]$ can be bounded uniformly for all $t$ in a neighborhood of a fixed $T>0$. Consequently, we only need to show \eqref{eq_stronger_continuity} for $T=0$.

\medskip

Write $\G_\A(T)$ for the integral operator with kernel
\begin{equation}\label{def:kernel_Gauss}
K^{\G}_\A(x,y;T)=\frac{1}{\sqrt{2\pi T}} \exp\!\bigg(\!-\frac{(x-y)^2}{2T}\bigg)
\E_{B^{x,y}} \bigg[\mathbf{1}_{\{\forall t:\,B^{x,y}(t)\in\A\}}\exp\bigg(
-\frac{1}{2}\int_0^T B^{x,y}(t)\,\mathrm{d} t\bigg) \bigg].
\end{equation}
As $T\to 0$, the operators $\G_\A(T)$ converge strongly to $\U_\A(0)$ (the orthogonal projector from $L^2(\rr_{\ge 0})$ onto $L^2(\A)$), see e.g. \cite[Section 2, Step 6]{Mc}. Therefore, it is enough to prove
$$
 \lim_{T\to 0} \E\big[\|\U_\A(T)f-\G_\A(T)f\|^p\big]=0
$$
or, in other words, the convergence to $0$ as $T\to 0$ of
\begin{equation} \label{eq_cont_small_term}
\begin{split}
\E\Bigg[\!\Bigg(\!\int_{\mathbb R_{\ge 0}}\!\bigg(\int_{\mathbb R_{\ge 0}}  \frac{1}{\sqrt{2\pi T}} \exp\!\bigg(\!-\frac{(x-y)^2}{2T}\bigg)\E_{B^{x,y}}\!\bigg[\mathbf{1}_{\{\forall t:\,B^{x,y}(t)\in\A\}} 
\exp\!\bigg(\!\!-\frac{1}{2}\int_0^T B^{x,y}(t)\,\mathrm{d}t\!\bigg) \\ 
\bigg(\exp\bigg(\frac{1}{\sqrt{\beta}}\,\int_0^\infty
L_a(B^{x,y})\,\mathrm{d}W(a)\bigg)-1\bigg)\bigg]\,f(y)\,\mathrm{d}y \bigg)^2\,\mathrm{d}x\Bigg)^{p/2}\Bigg].
\end{split}
\end{equation}

\smallskip

Next, we apply the elementary inequality $|e^S-1|\le |S|\,e^{|S|}$, $S\in\rr$ with
$$
S:=\frac{1}{\sqrt{\beta}}\,\int_0^\infty L_a(B^{x,y})\,\mathrm{d}W(a)
$$
to estimate the expression in \eqref{eq_cont_small_term} by
\begin{equation}
\begin{split}
\E\Bigg[\!\Bigg(\!\int_{\mathbb R_{\ge 0}}\!\bigg(\int_{\mathbb R_{\ge 0}}  &\frac{1}{\sqrt{2\pi T}} \exp\!\bigg(\!-\frac{(x-y)^2}{2T}\bigg) \\
& \E_{B^{x,y}}\!\bigg[\mathbf{1}_{\{\forall t:\,B^{x,y}(t)\in\A\}} 
\!\exp\!\bigg(\!\!-\frac{1}{2}\int_0^T\!\!B^{x,y}(t)\,\mathrm{d}t\!\bigg) 
|S|\,e^{|S|}\bigg]|f(y)|\,\mathrm{d}y \bigg)^2\mathrm{d}x\!\Bigg)^{p/2}\Bigg].
\end{split}
\end{equation}
Moreover, writing the square of the inner integral and then the $p/2$ power of the outer integral as products of integrals in indepedendent variables and applying Fubini's Theorem together with H\"older's inequality we end up with the bound
\begin{equation}
\begin{split}
& \Bigg(\!\int_{\mathbb R_{\ge 0}}\!\bigg(\int_{\mathbb R_{\ge 0}}  \frac{1}{\sqrt{2\pi T}} \exp\!\bigg(\!-\frac{(x-y)^2}{2T}\bigg) \\
&\qquad \E_{B^{x,y}}\!\bigg[\mathbf{1}_{\{\forall t:\,B^{x,y}(t)\in\A\}} \!\exp\!\bigg(\!\!-\frac{1}{2}\int_0^T\!\!B^{x,y}(t)\,\mathrm{d}t\!\bigg) 
\E_W\Big[|S|^p\,e^{p|S|}\Big]^{1/p}\bigg]|f(y)|\,\mathrm{d}y \bigg)^2\mathrm{d}x\!\Bigg)^{p/2}\!.
\end{split}
\end{equation}
Dropping the indicator function and using the Cauchy-Schwarz inequality we obtain the further estimate
\begin{equation}
\begin{split}
\Bigg(\!\int_{\mathbb R_{\ge 0}}\!\bigg(\int_{\mathbb R_{\ge 0}}  & \frac{1}{\sqrt{2\pi T}} \exp\!\bigg(\!-\frac{(x-y)^2}{2T}\bigg)\,|f(y)| \\
&\E_{B^{x,y}}\bigg[\!\exp\bigg(\!\!-\!\int_0^T\!\! B^{x,y}(t)\,\mathrm{d}t\bigg)\bigg]^{1/2}\,
\E_{B^{x,y}}\bigg[\E_W\Big[|S|^p\,e^{p|S|}\Big]^{2/p}\bigg]^{1/2}\mathrm{d}y \bigg)^2\mathrm{d}x\!\Bigg)^{p/2}\!.
\end{split}
\end{equation}

\smallskip

Now, the same argument as in the proof of Lemma
\ref{Lemma_HS} shows that the first expectation with respect to $B^{x,y}$ is bounded by a uniform constant for all $x,y\in\rr_{\ge0}$ and $T\in(0,1]$. To control the second expectation with respect to $B^{x,y}$ we combine the representation  
$$
B^{x,y}(t)=\sqrt{T}\,\hat{B}^{x/\sqrt{T}, y/\sqrt{T}}(t/T),\quad t\in[0,T]
$$
with a standard Brownian bridge $\hat B^{x/\sqrt{T}, y/\sqrt{T}}$ connecting $x/\sqrt{T}$ to $y/\sqrt{T}$ in time $1$, the fact that for a given trajectory of $\hat B^{x/\sqrt{T}, y/\sqrt{T}}$ the random variable $S$ has the normal distribution with mean $0$ and variance
$$
\frac{T^{3/2}}{\beta} \int_0^\infty L_a(\hat B^{x/\sqrt{T},y/\sqrt{T}})^2\,\mathrm{d}a,
$$
the elementary inequality $|s|^p\le p!\,e^{|s|}$, $s\in\rr$, and an estimate on the integral of the squared local times as in the proof of Lemma \ref{Lemma_HS} to obtain a bound of the form
$$
\E_{B^{x,y}}\bigg[\E_W\Big[|S|^p\,e^{p|S|}\Big]^{2/p}\bigg]^{1/2}
\le C\,T^{3/4}\,e^{C\,|x-y|/\sqrt{T}}.
$$
Here, $C<\infty$ is a constant depending only on $p$. 

\medskip

All in all, we end up with the following estimate on the expression in \eqref{eq_cont_small_term}:
\begin{equation}
C\,T^{p/4} \Bigg(\int_{\mathbb R_{\ge 0}} \bigg(\int_{\mathbb R_{\ge 0}} 
\exp\bigg(-\frac{(x-y)^2}{2T}+C\,\frac{|x-y|}{\sqrt T}\bigg)\,|f(y)|\,\mathrm{d}y\bigg)^2\,\mathrm{d}x\Bigg)^{p/2}.
\end{equation}
By Young's inequality for convolution the latter is at most 
$$
C\,T^{p/4} \|f\|^p \bigg(\int_\rr \exp\bigg(-\frac{x^2}{2T}+C\,\frac{|x|}{\sqrt T}\bigg)\,\mathrm{d}x\bigg)^p \le C\,T^{3p/4}\|f\|^p,
$$
which finishes the proof.
\end{proof}

\section{Convergence of extreme eigenvalues} \label{Section_extreme_convergence}

The aim of this section is to obtain Corollary \ref{Corollary_spectrum} from Theorem
\ref{Theorem_Main}. Such proofs are standard in the literature on the moments method for random matrices, see e.g. \cite{Soshnikov}, \cite[Section 5]{Sodin_edge}, \cite{Sodin_ICM}. Nevertheless, we give a sketch of the proof for the sake of completeness. In addition, it yields Proposition \ref{Proposition_spectrum_of_semigroup}. We start with the following statement.

\begin{lemma} \label{Lemma_Trace_conv}
The convergence in distribution
\begin{equation}\label{eq_trace_exp_sum}
\sum_{i=1}^N e^{T\lambda^i_{N;{\mathcal
A}}/2}\longrightarrow_{N\to\infty}\mathrm{Trace}\big(\U_\A(T)\big)
\end{equation}
holds jointly for any finitely many $T$'s and $\A$s.
\end{lemma}

\begin{proof}
We write $\mu^{1,+}_{N;\A}\ge \mu^{2,+}_{N;\A}\ge\cdots$ and
$\mu^{1,-}_{N;\A}\le\mu^{2,-}_{N;\A}\le\cdots$ for the positive and the negative
eigenvalues of $M_{N;\A}$, respectively. In addition, we define their rescaled
versions
\begin{equation*}
\lambda^{i,+}_{N;\A}\!=\!N^{1/6}\big(\mu^{i,+}_{N;\A}-2\sqrt{N}\big),\,i=1,2,\ldots\;\mathrm{and}\;
\lambda^{i,-}_{N;\A}\!=\!-N^{1/6}\big(\mu^{i,-}_{N;\A}+2\sqrt{N}\big),\,i=1,2,\ldots
\end{equation*}
Clearly,
\begin{equation} \label{eq_Trace_sum}
\begin{split}
& \mathrm{Trace}({\mathcal M}(T,\A,N)) = \frac{1}{2}\,\sum_i
\bigg(1+\frac{\lambda^{i,+}_{N;\A}}{2N^{2/3}} \bigg)^{\lfloor T N^{2/3}\rfloor} +
\frac{1}{2}\,\sum_i \bigg(
1+\frac{\lambda^{i,+}_{N;\A}}{2 N^{2/3}} \bigg)^{\lfloor T N^{2/3}\rfloor-1} \\
& +\frac{(-1)^{\lfloor T N^{2/3}\rfloor}}{2}\,\sum_i
\bigg(1+\frac{\lambda^{i,-}_{N;\A}}{2 N^{2/3}}\bigg)^{\lfloor T N^{2/3}\rfloor}
+\frac{(-1)^{\lfloor T N^{2/3}\rfloor-1}}{2}\,\sum_i \bigg(
1+\frac{\lambda^{i,-}_{N;\A}}{2 N^{2/3}} \bigg)^{\lfloor T N^{2/3}\rfloor-1}.
\end{split}
\end{equation}
From Theorem \ref{Theorem_Main} we know that both sides of \eqref{eq_Trace_sum}
converge in distribution to the right-hand side of \eqref{eq_trace_exp_sum} in the
limit $N\to\infty$. Consequently, it suffices to show the convergence to $0$ in
probability as $N\to\infty$ of the difference between the right-hand side of
\eqref{eq_Trace_sum} and
\begin{equation} \label{eq_Trace_with_exponents}
\begin{split}
\sum_{i=1}^N e^{T\lambda^i_{N;{\mathcal A}}/2}= \frac{1}{2}\sum_{i=1}^N
e^{T\lambda^i_{N;\A}/2} + \frac{1}{2} \sum_{i=1}^N e^{T\lambda^i_{N;\A}/2} +
\frac{(-1)^{\lfloor T N^{2/3}\rfloor}}{2} \sum_{i=1}^N e^{
T \lambda^i_{N;\A}/2} \\
+ \frac{(-1)^{\lfloor T N^{2/3}\rfloor-1}}{2} \sum_{i=1}^N e^{T\lambda^i_{N;\A}/2}.
\end{split}
\end{equation}

\smallskip

To analyze the difference between the right-hand sides of \eqref{eq_Trace_sum} and
\eqref{eq_Trace_with_exponents} we take $\eps=1/100$ and distinguish between four
types of (rescaled) eigenvalues:
\begin{enumerate}
\item ``bulk eigenvalues'': $\lambda^{i,+}_{N;\A}$'s and $\lambda^{i,-}_{N;\A}$'s which are less or equal to $-N^\eps$,
\item ``outliers'': $\lambda^{i,+}_{N;\A}$'s and $\lambda^{i,-}_{N;\A}$'s which are greater or equal to $N^\eps$,
\item ``right edge eigevalues'': $\lambda^{i,+}_{N;\A}$'s in $(-N^\eps,N^\eps)$,
\item ``left edge eigevalues'': $\lambda^{i,-}_{N;\A}$'s in $(-N^\eps,N^\eps)$.
\end{enumerate}
The contribution of the bulk eigenvalues to the right-hand side of
\eqref{eq_Trace_with_exponents} becomes negligible in the limit $N\to\infty$, since
there are at most $N$ of them and each contributes at most $e^{-TN^\eps/2}$. Due to
the elementary inequality $1+a\le e^a$, $a\in\rr$, the same is true for the
right-hand side of \eqref{eq_Trace_sum}.

\medskip

Next, we show that with probability tending to $1$ as $N\to\infty$  there are no
outliers. To this end, we consider a sequence $T_N$, $N\in\nn$ of positive numbers
such that $T_NN^{2/3}$, $N\in\nn$ are even integers and $\sup_N
|T_N-T|N^{2/3}<\infty$. Then,
\begin{equation*}
\mathrm{Trace}\big((M_{N;\A})^{T_NN^{2/3}}\big) = \sum_i \bigg(
1+\frac{\lambda^{i,+}_{N;\A}}{2 N^{2/3}} \bigg)^{T_NN^{2/3}} + \sum_i \bigg(
1+\frac{\lambda^{i,-}_{N;\A}}{2 N^{2/3}}\bigg)^{T_NN^{2/3}}.
\end{equation*}
Hence, if there is an outlier, then
\begin{equation}\label{eq_even_trace}
\mathrm{Trace}\big((M_{N;\A})^{T_NN^{2/3}}\big)\ge
\bigg(1+\frac{N^\eps}{2N^{2/3}}\bigg)^{T_NN^{2/3}}.
\end{equation}
According to Remark \ref{Remark_even_and_odd}, the left-hand side of
\eqref{eq_even_trace} converges in distribution as $N\to\infty$. On the other hand,
the right-hand side of \eqref{eq_even_trace} becomes arbitrarily large in the same
limit. Consequently, the probability that the inequality \eqref{eq_even_trace} takes
place tends to $0$ as $N\to\infty$.

\medskip

Finally, to the summands involving the edge eigenvalues we can apply the
approximations
\begin{equation*}
\bigg(1+\frac{\lambda^{i,\pm}_{N;\A}}{2 N^{2/3}} \bigg)^{\lfloor T
N^{2/3}\rfloor}\approx e^{T\lambda^{i,\pm}_{N;\A}/2},\quad
\bigg(1+\frac{\lambda^{i,\pm}_{N;\A}}{2 N^{2/3}} \bigg)^{\lfloor T
N^{2/3}\rfloor-1}\approx e^{T\lambda^{i,\pm}_{N;\A}/2}
\end{equation*}
with additive errors of at most
\begin{equation*}
\big(e^{N^{-2/3+2\eps}/2}-1\big)\bigg(1+\frac{\lambda^{i,\pm}_{N;\A}}{2
N^{2/3}}\bigg)^{\lfloor T N^{2/3}\rfloor},\quad
\big(e^{N^{-2/3+2\eps}/2}-1\big)\bigg(1+\frac{\lambda^{i,\pm}_{N;\A}}{2
N^{2/3}}\bigg)^{\lfloor T N^{2/3}\rfloor-1},
\end{equation*}
respectively. It remains to show that the sums of the latter errors over all the
edge eigenvalues tend to $0$ in probability. To this end, it suffices to prove that
the expressions
\begin{equation*}
\big(e^{N^{-2/3+2\eps}/2}-1\big) \sum_{i=1}^N \bigg(1+\frac{\lambda^i_{N;\A}}{2
N^{2/3}}\bigg)^{\lfloor T N^{2/3}\rfloor},\quad \big(e^{N^{-2/3+2\eps}/2}-1\big)
\sum_{i=1}^N \bigg(1+\frac{\lambda^i_{N;\A}}{2 N^{2/3}}\bigg)^{\lfloor T
N^{2/3}\rfloor-1}
\end{equation*}
converge to $0$ in probability, since the contributions of the non-edge eigenvalues
to them have been shown to be negligible before. The latter convergences are direct
consequences of Remark \ref{Remark_even_and_odd}.
\end{proof}

We can now deduce Corollary \ref{Corollary_spectrum} and Proposition
\ref{Proposition_spectrum_of_semigroup}.  According to Proposition
\ref{Proposition_trace_class}, for each $T>0$, $\U_\A(T)$ is a symmetric trace class
operator with probability one. In particular, $\U_\A(T)$ is compact (see e.g. \cite[Section 30.8, Exercise 11 (h), (g)]{Lax}), and therefore its spectrum is discrete. Moreover, the almost sure commutativity of operators $\U_\A(T)$, $T>0$ shows that
there exists an orthonormal basis $\vv^1_\A,\vv^2_\A,\dots$ of $L^2(\A)$ consisting of eigenfunctions common to all $\U_\A(T)$, $T\in(0,\infty)\cap\qq$ (see e.g. \cite[Chapter 28, Theorem 7]{Lax}). Next, we order the eigenvectors in such a way
that the corresponding eigenvalues of $\U_\A(1)$ satisfy $e^1_\A\ge e^2_\A\ge\cdots$, note that all of the latter are non-negative ($\U_\A(1)$ is a non-negative operator), and define $\eta^1_\A\ge\eta^2_\A\ge\cdots$ by $\eta^i_\A= 2 \log e^i_\A$, $i\in\nn$. 

\medskip

The semigroup property shows further that $\U_\A(T) \vv^i_\A =
\exp(T \eta^i_\A/2) \vv^i_\A$, $i\in\nn$ for all $T\in(0,\infty)\cap\qq$ and, in particular, $\mathrm{Trace}(\U_\A(T))=\sum_{i=1}^{\infty} \exp(T \eta^i_\A/2)$. We remark that a priori  Proposition \ref{Proposition_trace_class} allows some of the eigenvalues $e^i_\A$ of $\U_\A(1)$ to vanish. However, if that were the case, all of the operators $\U_\A(T)$, $T\in(0,\infty)\cap\qq$ would evaluate to $0$ on the corresponding eigenvectors, contradicting the continuity of the semigroup at $T=0$ established in Proposition \ref{Proposition_continuous}.

\medskip

At this point, Lemma \ref{Lemma_Trace_conv} implies the convergences
\begin{equation}\label{eq_Laplace_conv}
\sum_{i=1}^N e^{T\lambda^i_{N;{\mathcal
A}}/2}\longrightarrow_{N\to\infty}\sum_{i=1}^\infty e^{T\eta^i_\A/2}
\end{equation}
in distribution, jointly for any finitely many $T$'s in $(0,\infty)\cap\qq$ and $\A$'s. Applying the Skorokhod Representation Theorem in
the form of \cite[Theorem 3.5.1]{Dud} we find a new probability space, on which the convergence of \eqref{eq_Laplace_conv} holds almost surely for all $T\in (0,\infty)\cap\qq$ and a fixed finite collection of $\A$'s. Then, the one-to-one property of Laplace transforms implies that, for all $i\in\nn$, we must have the almost sure convergence $\lim_{N\to\infty} \lambda^i_{N;\mathcal A}=\eta^i_\A$, see e.g. \cite[Section 5]{Sodin_edge} for more details. This proves \eqref{eq_edge_limit}. Finally, Proposition \ref{Proposition_spectrum_of_semigroup} for irrational $T>0$ follows from the continuity of the semigroup (Proposition \ref{Proposition_spectrum_of_semigroup}), which then implies  
$\mathrm{Trace}(\U_\A(T))=\sum_{i=1}^{\infty} \exp(T \eta^i_\A/2)$ and \eqref{eq_Laplace_conv} via Lemma \ref{Lemma_Trace_conv} as before.

\section{Properties of the Stochastic Airy semigroup II}
\label{Section_properties_2}

\begin{proof}[Proof of Corollary \ref{Corollary_generator}]
Note that the matrix $M_N$ can be interpreted as an operator on $L^2(\rr_{\ge0})$:
for a function $f\in L^2(\rr_{\ge0})$, one can define the action of $M_N$ on $f$ by
interpreting the $N$-dimensional vector $M_N(\pi_Nf)$ as a piecewise constant
function on the intervals
$[0,N^{-1/3}),\,[N^{-1/3},2N^{-1/3}),\,\ldots,\,[N^{2/3}-N^{-1/3},N^{2/3})$ whose
values are given by the $N^{1/6}$ multiples of the entries of $M_N(\pi_Nf)$. The
main result of \cite{RRV} (see also \cite{KRV}) provides a coupling of the matrices
$M_N$, $N\in\nn$ (viewed as operators on $L^2(\rr_{\ge0})$) such that their scaled
largest eigenvalues and the corresponding eigenfunctions converge to those of $-\frac{1}{2}SAO_\beta$ almost surely. Simultaneously, the Brownian motion $W$ arises via the limit transition as in \eqref{eq_limit_BM_appearance}.

\medskip

It is not hard to see that under the coupling of \cite{RRV} the largest eigenvalues and the corresponding eigenfunctions of the matrices $\mathcal M(T,\rr_{\ge0},N)$, $N\in\nn$ (viewed as operators on $L^2(\rr_{\ge0})$) converge to those of $e^{-\frac{T}{2}SAO_\beta}$ almost surely. Indeed, for the convergence of eigenfunctions it suffices to observe that the eigenvectors of each $\mathcal M(T,\rr_{\ge0},N)$ are the same as those of the corresponding matrix $M_N$, and the eigenfunctions of $e^{-\frac{T}{2}SAO_\beta}$ are the same as those of $-\frac{1}{2}SAO_\beta$. On the other hand, for the convergence of eigenvalues we can use the same approximation of the normalized high powers of eigenvalues by the exponentials of their scaled versions as in the proof of Lemma \ref{Lemma_Trace_conv}.

\medskip

Moreover, the almost sure convergence of the eigenvalues of $-\frac{1}{2}SAO_\beta$ to $-\infty$ (\cite[Proposition 3.5]{RRV}) implies the almost sure strong convergence of the matrices $\mathcal M(T,\rr_{\ge0},N)$, $N\in\nn$ (viewed as operators on $L^2(\rr_{\ge0})$) to $e^{-\frac{T}{2}SAO_\beta}$. Since strong convergence implies weak convergence, by comparing with Theorem
\ref{Theorem_Main} we conclude that the finite-domensional distributions of the families $\int_{\mathbb R_{\ge 0}} f(x)\,(e^{-\frac{T}{2}SAO_\beta}g)(x)\,\mathrm{d}x$, $f,g\in L^2(\rr_{\ge0})$
and $\int_{\mathbb R_{\ge 0}} f(x)\,(\U_\A(T) g)(x)\,\mathrm{d}x$, $f,g\in L^2(\rr_{\ge0})$ coincide, as well as their joint distributions with the Brownian motions $W$ in the definitions of
$SAO_\beta$ and $\U_\A(T)$, respectively.

\medskip

In other words, the laws of the pairs $(W,e^{-\frac{T}{2}SAO_\beta})$ and $(W,\U_\A(T))$ are the same, where the second component is endowed with the Borel $\sigma$-algebra associated with the weak operator topology. Putting together the measurability of both $e^{-\frac{T}{2}SAO_\beta}$ and $\U_\A(T)$ with respect to the
$\sigma$-algebra generated by $W$ with \cite[Theorem 5.3]{Ka} we conclude that there is a unique up to null sets of $W$ deterministic functional $F$ such that $(W,F(W))$ has the same law as each of the pairs $(W,e^{-\frac{T}{2}SAO_\beta})$ and $(W,\U_\A(T))$. Consequently, if we choose the Brownian motions $W$ in the
respective definitions of $SAO_\beta$ and $\U_\A(T)$ to be the same, the almost sure identities $e^{-\frac{T}{2}SAO_\beta}=F(W)=\U_\A(T)$ will hold.
\end{proof}

We conclude the section with the proof of Proposition \ref{Proposition_expectation} and Corollary \ref{Cor_excursion_identity}.

\begin{proof}[Proof of Proposition \ref{Proposition_expectation} and Corollary \ref{Cor_excursion_identity}] Fix a $T>0$. From Corollary \ref{Corollary_generator} we know that the expectation $\E\big[\sum_{i=1}^\infty e^{T\eta^i/2}\big]$ is given by $\E\big[\mathrm{Trace}(\U(T))\big]$ which, in turn, can be computed to $\E\big[\int_{\rr_{\ge0}} K(x,x;T)\,\mathrm{d}x\big]$ thanks to the trace formula \eqref{trace_formula}. To simplify the latter expectation we apply Fubini's Theorem and take the expectation with respect to $W$ first using Novikov's condition (see e.g. \cite[Proposition 3.5.12]{KS}) to obtain
\begin{equation*}
\frac{1}{\sqrt{2\pi T}}\,\int_{\rr_{\ge0}} \E_{B^{x,x}}\bigg[\mathbf{1}_{\{\forall t:\,B^{x,x}(t)\ge0\}}
\exp\bigg(-\frac{1}{2}\int_0^T B^{x,x}(t)\,\mathrm{d}t+\frac{1}{2\beta}\int_0^\infty L_a(B^{x,x})^2\,\mathrm{d}a\bigg)\bigg]\,\mathrm{d}x.
\end{equation*}
Next, we note that $B^{0,0}:=B^{x,x}-x$ is a standard Brownian bridge connecting $0$ to $0$ in time $T$, in terms of which the latter expression reads
\begin{equation*}
\frac{1}{\sqrt{2\pi T}}\int_{\rr_{\ge0}} \!\!\E_{B^{0,0}}\bigg[\!\mathbf{1}_{\{\forall t:\,x\ge-B^{0,0}(t)\}}
\exp\bigg(\!\!-\frac{1}{2}\int_0^T \!\! \big(B^{0,0}(t)+x\big)\mathrm{d}t+\frac{1}{2\beta}\int_{-\infty}^\infty L_a(B^{0,0})^2\,\mathrm{d}a\!\bigg)\!\bigg]\mathrm{d}x.
\end{equation*}
At this point, we rewrite the event $\{\forall t:\,x\ge-B^{0,0}(t)\}$ as $\{x\ge-\min(B^{0,0})\}$ and compute the integral with respect to $x$ before taking the expectation to get
\begin{equation}\label{eq_BB_simplified}
\sqrt{\frac{2}{\pi}}\,T^{-3/2}\,\E_{B^{0,0}}\bigg[
\exp\bigg(-\frac{1}{2}\int_0^T \big(B^{0,0}(t)-\min(B^{0,0})\big)\mathrm{d}t+\frac{1}{2\beta}\int_{-\infty}^\infty L_a(B^{0,0})^2\,\mathrm{d}a\bigg)\bigg].
\end{equation}

\smallskip

Lastly, we consider the Vervaat transform $V(B^{0,0})(t):=B^{0,0}\big((t+t^*)\,\mathrm{mod}\,T\big)-\min(B^{0,0})$, $t\in[0,T]$, where $t^*$ is the almost surely unique time at which $B^{0,0}$ attains its global minimum $\min(B^{0,0})$. By Vervaat's Theorem (\cite[Theorem 1]{Ver}), $V(B^{0,0})$ is a standard Brownian excursion on $[0,T]$, so that $\e(t):=T^{-1/2}V(B^{0,0})(Tt)$, $t\in[0,1]$ is a standard Brownian excursion on $[0,1]$. The quantity of \eqref{eq_BB_simplified} can be now reexpressed in terms of $\e$ as
\begin{equation*}
\sqrt{\frac{2}{\pi}}\,T^{-3/2}\,\E\bigg[
\exp\bigg(-\frac{T^{3/2}}{2}\int_0^1 \e(t)\,\mathrm{d}t+\frac{T^{3/2}}{2\beta}\int_0^\infty (l_y)^2\,\mathrm{d}y\bigg)\bigg],
\end{equation*}
where each $l_y$ is the total local time of $\e$ at level $y$. Proposition \ref{Proposition_expectation} readily follows.

\medskip

In addition, for $\beta=2$, we can equate the results of Propositions \ref{Proposition_expectation} and \ref{Proposition_expectation_Ok} to obtain
\begin{equation}
\E\bigg[
\exp\bigg(-\frac{T^{3/2}}{2}\bigg(\int_0^1 \e(t)\,\mathrm{d}t-\frac{1}{2}\int_0^\infty (l_y)^2\,\mathrm{d}y\bigg)\bigg)\bigg]=e^{T^3/96}.
\end{equation}
Since $T>0$ was arbitrary, we have identified the moment generating function on the
negative half-line of the random variable $\int_0^1
\e(t)\,\mathrm{d}t-\frac{1}{2}\int_0^\infty (l_y)^2\,\mathrm{d}y$ with that of a
Gaussian random variable of mean $0$ and variance $\frac{1}{12}$. The usual analytic
continuation technique via Morera's Theorem allows to extend such identity to the
open left complex half-plane. Corollary \ref{Cor_excursion_identity} now follows
from the Uniqueness Theorem for moment generating functions (see e.g. \cite[Chapter
VI, Theorem 6a]{Wi}).
\end{proof}

\appendix

\section{Global asymptotics}

In this section, we study the fluctuations of the empirical measure
$\rho_N$ defined in \eqref{eq_empirical}. The following two theorems give the Law of Large Numbers and the Central Limit Theorem for $\rho_N$, respectively.

\begin{theorem} \label{Theorem_LLN_moments}
Instead of Assumption \ref{Assumptions} suppose that $\aa(m)$, $m\in\nn$ and $\xi(m)$, $m\in\nn$ are mutually independent random variables with finite moments of all orders satisfying
\begin{equation}\label{app_moment_cond1}
\forall\,k\in\nn:\quad \lim_{m\to\infty} \frac{\E[|\aa(m)|^k]}{m^{k/2}}=\lim_{m\to\infty}
\frac{\E[|\xi(m)|^k]}{m^{k/2}}=0.
\end{equation}
Then, for each $k\in\nn$, one has the convergences in probability
\begin{equation} \label{eq_moments_limit}
\lim_{N\to\infty} \int_{\mathbb R} x^k \rho_N(\mathrm{d}x)
=\lim_{N\to\infty} \frac{1}{N}\,{\rm Trace}\bigg(\frac{M_N}{\sqrt N}
\bigg)^k
=\begin{cases} \frac{1}{k/2+1} {k\choose {k/2}}& \text{if }k\text{ is even,}\\ 0&
\text{if }k\text{ is odd.}
  \end{cases}
 \end{equation}
\end{theorem}

\smallskip

\begin{rmk}
The limits in \eqref{eq_moments_limit} are precisely the moments of the semicircle distribution, so that Theorem \ref{Theorem_LLN_moments} implies Proposition \ref{Prop_semicircle}.
\end{rmk}

\begin{theorem} \label{Theorem_CLT_moments}
Instead of Assumption \ref{Assumptions} suppose that $\aa(m)$, $m\in\nn$ and $\xi(m)$, $m\in\nn$ are mutually independent random variables satisfying
\begin{equation}\label{app_moment_cond2}
\begin{split}
& \forall\,k\in\nn:\quad \sup_{m\in\nn}\;\E\big[|\aa(m)|^k\big]<\infty,\;\;\; \sup_{m\in\nn}\;\E\big[|\xi(m)|^k\big]<\infty, \\
& \exists\,\mu_\aa,\mu_\xi\in\rr:\quad
\lim_{m\to\infty} \sqrt{m}\,\E[\aa(m)]=\mu_\aa,\;\;\;
\lim_{m\to\infty} \sqrt{m}\,\E[\xi(m)]=\mu_\xi, \\
& \exists\,s_\aa,\,s_\xi\in(0,\infty):\quad
\lim_{m\to\infty} \E\big[\aa(m)^2\big]=s_\aa^2,\;\;\;
\lim_{m\to\infty} \E\big[\xi(m)^2\big]=s_\xi^2.
\end{split}
\end{equation}
Then, every finite subfamily of
\begin{equation*}
\bigg\{{\rm Trace}\bigg(\frac{M_{\lfloor \alpha N \rfloor}}{\sqrt N}
\bigg)^k -\E\bigg[{\rm Trace}\bigg(\frac{M_{\lfloor \alpha N \rfloor}}{\sqrt N}\bigg)^k\,\bigg]:\;\alpha\in(0,1],\;k\in\nn\bigg\}
\end{equation*}
converges in distribution to a centered Gaussian random vector whose covariance matrix has the entries
\begin{equation} \label{eq_limit_covariance}
(\alpha \wedge \alpha')^{\frac{k+k'}{2}}\begin{cases} s_\aa^2\, \frac{2kk'}{k+k'}
  {{k-1}\choose (k-1)/2}\,{{k'-1} \choose (k'-1)/2}& \text{if }k\text{ and } k'\text{ are both odd,}\\
s_\xi^2\,\frac{2kk'}{k+k'}   {k\choose k/2}{k'\choose k'/2}& \text{if }k\text{ and } k'\text{ are both even,}\\
  0& \text{if }k\text{ and }k'\text{ are of different parity.}
  \end{cases}
\end{equation}
\end{theorem}

\begin{rmk}
Note that in the covariance structure of \eqref{eq_limit_covariance} the $\alpha$-direction plays a very different role from the $k$-direction. In the $\alpha$-direction, the random variables change as the values of a Brownian motion. This is due to the tridiagonal structure of the matrix $M_N$ which makes large submatrices of the powers of $M_N$ asymptotically uncorrelated. On the other hand, in the $k$-direction, the fluctuations are governed by a log-correlated Gaussian field known from the case of the Gaussian $\beta$ ensemble. In contrast, the fluctuations of the corners of Wigner matrices
\cite{Bor_GFF}, sample covariance matrices \cite{DP_GFF}, adjacency matrices of random graphs \cite{GanPal_GFF}, \cite{Johnson}, and multilevel $\beta$-Jacobi ensembles \cite{BG_GFF} are governed by the two-dimensional Gaussian free field ($2$D GFF), in which both coordinates behave the same. An intuitive explanation for this difference is that the amount of randomness in $M_N$ (two sequences of mutually independent random variables) is not enough to lead to such a complicated object as the $2$D GFF.
\end{rmk}

\smallskip

The rest of the section is devoted to the proof of Theorems \ref{Theorem_LLN_moments} and
\ref{Theorem_CLT_moments}. For similar proofs see  \cite{DE2}, \cite{Duy}.

\medskip

\noindent\textit{Proof of Theorems \ref{Theorem_LLN_moments} and \ref{Theorem_CLT_moments}.}
To obtain Theorem \ref{Theorem_LLN_moments} we need to consider traces of powers of $M_N$. By definition,
\begin{equation}
\label{eq_Trace_as_sum}
\mathrm{Trace}\big((M_N)^k\big)=\sum M_N[i_0,i_1]\,M_N[i_1,i_2]\cdots M_N[i_{k-1},i_k]
\end{equation}
where the sum is taken over all sequences of integers $i_0,i_1,\ldots,i_k$ in $\{1,2,\ldots,N\}$ such that $i_0=i_k$ and $|i_l-i_{l-1}|\le 1$ for all $l=1,2,\ldots,k$. In the latter sum, each factor of the form $M_N[m,m]$ is of order $o(\sqrt{m})$ in the sense of condition \eqref{app_moment_cond1}, whereas each factor of the form $M_N[m,m+1]$ or $M_N[m+1,m]$ is given by $\sqrt{m}+\xi(m)$ with $\xi(m)$ being of order $o(\sqrt{m})$ in the sense of condition \eqref{app_moment_cond1}. We immediately conclude that the sum in \eqref{eq_Trace_as_sum} is dominated by paths with $|i_l-i_{l-1}|=1$ for all $l=1,2,\ldots,k$ and that the paths with $i_0\le k$ or $i_0\ge N-k+1$ are asymptotically negligible. Consequently, for even $k\in\nn$,
\begin{equation*}
\begin{split}
\frac{1}{N}\,{\rm Trace}\bigg(\frac{M_N}{\sqrt N}\bigg)^k
\!\approx\!\frac{1}{N^{k/2+1}}\E\big[{\rm Trace}\big((M_N)^k\big)\big]
\!=\!\bigg(\frac{1}{N^{k/2+1}}\sum_{m=1}^N m^{k/2} {k \choose {k/2}}\bigg)(1+o(1)) \\
\longrightarrow_{N\to\infty} \int_0^1 x^{k/2}\,\mathrm{d}x\,
{k \choose {k/2}}=\frac{1}{k/2+1} {k \choose {k/2}},
\end{split}
\end{equation*}
where ``$\approx$'' means that the difference of the two terms tends to $0$ in probability (which here can be easily verified by using Chebyshev's inequality and condition \eqref{app_moment_cond1}). Very similar considerations yield for odd $k\in\nn$ that $\frac{1}{N}\,{\rm Trace}\big(\frac{M_N}{\sqrt N}\big)^k\approx \frac{1}{N^{k/2+1}}\E\big[{\rm Trace}\big((M_N)^k\big)\big]$ which, in turn, tends to $0$ in the limit $N\to\infty$. This gives Theorem \ref{Theorem_LLN_moments}.

\medskip

We proceed to the proof of Theorem \ref{Theorem_CLT_moments} and need to study the random variables
\begin{equation}\label{app_centered_rv}
N^{-k/2}\Big({\rm Trace}\big((M_{\lfloor \alpha N\rfloor})^k\big)-\E\big[{\rm Trace}\big((M_{\lfloor \alpha N\rfloor})^k\big)\big]\Big),\;\;\alpha\in(0,1],\;k\in\nn.
\end{equation}
To this end, for each $\alpha\in(0,1]$ and $k\in\nn$, we introduce the quantity
\begin{equation}
\mu_{\alpha,k}:=\sum\,\prod_{l=1}^k \sqrt{i_l\wedge i_{l-1}}
\end{equation}
where the sum is taken over sequences of integers $i_0,i_1,\ldots,i_k$ in $\{1,2,\ldots,\lfloor \alpha N\rfloor\}$ such that $i_0=i_k$ and $|i_l-i_{l-1}|=1$ for all $l=1,2,\ldots,k$. We refer to such sequences as sequences of type 0. In addition, we call sequences with $i_L=i_{L-1}$ for some $L$ and $|i_l-i_{l-1}|=1$ for all $l\neq L$ sequences of type 1 and sequences with $i_L=i_{L-1}$, $i_{L'}=i_{L'-1}$ for some $L\neq L'$ and $|i_l-i_{l-1}|=1$ for all $l\neq L,L'$ sequences of type 2.

\medskip

Next, we consider ${\rm Trace}\big((M_{\lfloor \alpha N\rfloor})^k\big)-\mu_{\alpha,k}$ and group the terms with the same sequence $i_0,i_1,\ldots,i_k$ together. For each sequence of type 0, we expand the resulting difference to get
\begin{equation} \label{eq_centered_moment}
\begin{split}
 N^{k/2}\bigg(\!\sum_{l_1=1}^k \frac{\xi(i_{l_1}\wedge i_{l_1-1})}{\sqrt N} \prod_{l\ne l_1} \frac{\sqrt{i_l\wedge i_{l-1}}}{\sqrt{N}}
+\!\!\sum_{\substack{l_1,l_2=1\\ l_1\neq l_2}}^k \frac{\xi(i_{l_1}\wedge i_{l_1-1})}{\sqrt N} \frac{\xi(i_{l_2}\wedge i_{l_2-1})}{\sqrt N}\prod_{l\ne l_1,l_2}
\frac{\sqrt{i_l\wedge i_{l-1}}}{\sqrt{N}}\\
+r(i_0,i_1,\ldots,i_k)\!\bigg),
\end{split}
\end{equation}
where the expectation of the remainder $|r(i_0,i_1,\ldots,i_k)|$ is of order $N^{-3/2}$ uniformly in the sequence $i_0,i_1,\ldots,i_k$ thanks to the first part of condition \eqref{app_moment_cond2}. On the other hand, for each sequence of type 1, say with $i_L=i_{L-1}$, we use the expansion
\begin{equation}\label{eq_centered_moment_2}
 N^{k/2}\bigg(\frac{\aa(i_L)}{\sqrt N}\,\prod_{l\ne L} \frac{\sqrt{i_l\wedge i_{l-1}}}{\sqrt{N}}+\frac{\aa(i_L)}{\sqrt N}\,\sum_{l_1\neq L} \frac{\xi(i_{l_1}\wedge i_{l_1-1})}{\sqrt{N}}\,\prod_{l\ne l_1,L} \frac{\sqrt{i_l\wedge i_{l-1}}}{\sqrt{N}}+r(i_0,i_1,\ldots,i_k)\bigg),
\end{equation}
again with a remainder $|r(i_0,i_1,\ldots,i_k)|$ whose expectation is of order $N^{-3/2}$ uniformly in the sequence. Lastly, for a sequence of type 2, say with $i_L=i_{L-1}$ and $i_{L'}=i_{L'-1}$, we end up with the expansion
\begin{equation}\label{eq_centered_moment_3}
N^{k/2}\bigg(\frac{\aa(i_L)}{\sqrt{N}}\,\frac{\aa(i_{L'})}{\sqrt{N}}\,\prod_{l\neq L,L'} \frac{\sqrt{i_l\wedge i_{l-1}}}{\sqrt{N}}+ r(i_0,i_1,\ldots,i_k)\bigg),
\end{equation}
yet again with a remainder $|r(i_0,i_1,\ldots,i_k)|$ with expectation uniformly of order $N^{-3/2}$.

\medskip

At this point, we multiply by $N^{-k/2}$ (thus, cancelling the prefactors in \eqref{eq_centered_moment}-\eqref{eq_centered_moment_3}) and sum over the sequences $i_0,i_1,\ldots,i_k$. We note that all sequences except for the ones of types 0, 1, and 2, as well as the sums of the remainders in \eqref{eq_centered_moment}-\eqref{eq_centered_moment_3} become negligible in $L^1$ sense in the limit $N\to\infty$ and therefore can be dropped. The remaining terms are of the following types:
\begin{equation} \label{eq_xi_coeff}
\sum_{m=1}^{\lfloor \alpha N\rfloor -1} \frac{\xi(m)}{\sqrt N}\,\bigg(\sum\,\prod_{l\neq l_1} \frac{\sqrt{i_l\wedge i_{l-1}}}{\sqrt{N}}\bigg)
\end{equation}
where the second sum is over sequences of type 0 such that $i_{l_1}\wedge i_{l_1-1}=m$ for some $l_1$ (if the event $i_{l_1}\wedge i_{l_1-1}=m$ occurs multiple times in a sequence, the sequence is counted that number of times);
\begin{equation}\label{eq_alpha_coeff}
\sum_{m=1}^{\lfloor \alpha N\rfloor} \frac{\aa(m)}{\sqrt N}\,\bigg(\sum\,\prod_{l\ne L} \frac{\sqrt{i_l\wedge i_{l-1}}}{\sqrt{N}}\bigg)
\end{equation}
where the second sum is over sequences of type 1 such that $i_L=i_{L-1}=m$;
\begin{equation}\label{eq_xi_xi_coeff}
\sum_{m_1,m_2=1}^{\lfloor \alpha N\rfloor-1}
\frac{\xi(m_1)\,\xi(m_2)}{N}\,\bigg(\sum\,
\prod_{l\ne l_1,l_2} \frac{\sqrt{i_l\wedge i_{l-1}}}{\sqrt{N}}\bigg)
\end{equation}
where the second sum is over sequences of type 0 such that $i_{l_1}\wedge i_{l_1-1}=m_1$, $i_{l_2}\wedge i_{l_2-1}=m_2$ for some $l_1\neq l_2$ (counted with the appropriate multiplicity);
\begin{equation}\label{eq_alpha_xi_coeff}
\sum_{m_1=1}^{\lfloor \alpha N\rfloor} \sum_{m_2=1}^{\lfloor \alpha N\rfloor-1} \frac{\aa(m_1)\,\xi(m_2)}{N} \,\bigg(\sum\,\prod_{l\ne l_1,L} \frac{\sqrt{i_l\wedge i_{l-1}}}{\sqrt{N}}\bigg)
\end{equation}
where the third sum is over sequences of type 1 such that $i_L=i_{L-1}=m_1$ and $i_{l_1}\wedge i_{l_1-1}=m_2$ for some $l_1\neq L$ (counted with the appropriate multiplicity);
\begin{equation}\label{eq_alpha_alpha_coeff}
\sum_{m_1,m_2=1}^{\lfloor \alpha N\rfloor} \frac{\aa(m_1)\,\aa(m_2)}{N}\,\bigg(\sum\,\prod_{l\neq L,L'} \frac{\sqrt{i_l\wedge i_{l-1}}}{\sqrt{N}}\bigg)
\end{equation}
where the second sum is over sequences of type 2 such that $i_L=i_{L-1}=m_1$, $i_{L'}=i_{L'-1}=m_2$.

\medskip

Condition \eqref{app_moment_cond2} now implies the following convergences as $N\to\infty$: the expressions in \eqref{eq_xi_coeff} and \eqref{eq_alpha_coeff} converge jointly in distribution to indepedent  nondegenerate Gaussian random variables (note that the coefficients of $\frac{\xi(m)}{\sqrt{N}}$ and $\frac{\aa(m)}{\sqrt{N}}$ tend to constants and apply e.g. the Central Limit Theorem of \cite[Theorem 8.4]{BR}); the sum of the terms with $m_1\neq m_2$ in \eqref{eq_xi_xi_coeff} can be decomposed into subsums according to the value of $|m_2-m_1|$ and the latter tend to $0$ in $L^1$ by the Law of Large Numbers (note also that there are less than $k/2$ subsums with nonzero coefficients of $\frac{\xi(m_1)\,\xi(m_2)}{N}$); the sum of the terms with $m_1=m_2$ in \eqref{eq_xi_xi_coeff} converges to a constant in $L^1$ by the Law of Large Numbers (note that coefficients of $\frac{\xi(m)^2}{N}$ tend to a constant); the expression in \eqref{eq_alpha_xi_coeff} can be decomposed into subsums according to the value of $m_2-m_1$ and the latter tend to $0$ in $L^1$ by the Law of Large Numbers (note also that there are less than $k$ subsums with nonzero coefficients of $\frac{\aa(m_1)\,\xi(m_2)}{N}$); and the expression in \eqref{eq_alpha_alpha_coeff} converges to a constant in $L^1$ by the same argument as for the expression in \eqref{eq_xi_xi_coeff}.

\medskip

At this point, we turn to the centered random variables of \eqref{app_centered_rv}. From the discussion so far it is clear that the limits in distribution of finite families of such are centered Gaussian random vectors, whose components arise from the limits in distribution of the expressions in \eqref{eq_xi_coeff} and \eqref{eq_alpha_coeff} centered according to their expectations. It remains to derive the limiting covariance structure of \eqref{eq_limit_covariance} from the latter.

\medskip

Since the expression of \eqref{eq_xi_coeff} is only present for even $k$ and the expression of \eqref{eq_alpha_coeff} is only present for odd $k$, the limiting covariance between random variables of \eqref{app_centered_rv} with $k,\,k'$ of different parities must vanish. Next, pick two random variables as in \eqref{app_centered_rv}  with even $k,\,k'$ and arbitrary $\alpha,\,\alpha'$. The covariance between such behaves asymptotically as
\begin{equation}\label{eq_even_cor}
\sum_{m=1}^{\lfloor (\alpha\wedge\alpha')N\rfloor-1} \frac{s_\xi^2}{N}\,\bigg(\sum\,\prod_{l\neq l_1} \frac{\sqrt{i_l\wedge i_{l-1}}}{\sqrt{N}}\bigg)\,\bigg(\sum\,\prod_{l\neq l_1} \frac{\sqrt{i_l\wedge i_{l-1}}}{\sqrt{N}}\bigg)
\end{equation}
where the second sum is over sequences $i_0,i_1,\ldots,i_k$ of type 0
such that $i_{l_1}\wedge i_{l_1-1}=m$ for some $l_1$ and the third sum is over sequences $i_0,i_1,\ldots,i_{k'}$ of type 0 such that $i_{l_1}\wedge i_{l_1-1}=m$ for some $l_1$ (as before, we count sequences with several occurences of the event $i_{l_1}\wedge i_{l_1-1}=m$ multiple times).

\medskip

Now, we note that the quantities $i_l\wedge i_{l-1}$ in \eqref{eq_even_cor} deviate from $m$ by less than $k/2$, the number of summands in the second sum in \eqref{eq_even_cor} is $k\,{k \choose k/2}$ (the number of sequences $i_0,i_1,\ldots,i_k$ of type 0  with a labeled step up to translates) when $k/2\le m\le \lfloor (\alpha\wedge\alpha')N\rfloor-k/2$, and the number of summands in the third sum in \eqref{eq_even_cor} is $k'\,{k' \choose k'/2}$ when $k'/2\le m\le \lfloor (\alpha\wedge\alpha')N\rfloor-k'/2$. As a result, the $N\to\infty$ limit of the expression in \eqref{eq_even_cor} is given by
\begin{equation}\label{eq_even_cor_int}
s_\xi^2\,\bigg(\!\int_0^{\alpha\wedge\alpha'} x^{(k-1+k'-1)/2}\,\mathrm{d}x\!\bigg)\,k{k \choose k/2}\,k'{k' \choose k'/2}
=s_\xi^2\,(\alpha\wedge\alpha')^{\frac{k+k'}{2}}\,\frac{2kk'}{k+k'}{k \choose k/2}{k' \choose k'/2}.
\end{equation}

\medskip

Similarly, for odd $k,\,k'$ and arbitrary $\alpha,\,\alpha'$, the covariance of the corresponding random variables in \eqref{app_centered_rv} behaves asymptotically as
\begin{equation}\label{eq_odd_cor}
\sum_{m=1}^{\lfloor (\alpha\wedge\alpha')N\rfloor} \frac{s_\aa^2}{N}\,\bigg(\sum\,\prod_{l\neq L} \frac{\sqrt{i_l\wedge i_{l-1}}}{\sqrt{N}}\bigg)\,\bigg(\sum\,\prod_{l\neq L} \frac{\sqrt{i_l\wedge i_{l-1}}}{\sqrt{N}}\bigg)
\end{equation}
where the second sum is over sequences $i_0,i_1,\ldots,i_k$ of type 1 such that $i_L=i_{L-1}=m$ and the third sum is over sequences $i_0,i_1,\ldots,i_{k'}$ of type 1 such that $i_L=i_{L-1}=m$. An argument as for even $k,k'$ allows to evaluate the $N\to\infty$ limit of the expression in \eqref{eq_odd_cor} to
\begin{equation}\label{eq_odd_cor_int}
\begin{split}
& s_\aa^2\,\bigg(\!\int_0^{\alpha\wedge\alpha'} x^{(k-1+k'-1)/2}\,\mathrm{d}x\!\bigg)\,k{k-1 \choose (k-1)/2}\,k'{k'-1 \choose (k'-1)/2} \\
& =s_\aa^2\,(\alpha\wedge\alpha')^{\frac{k+k'}{2}}\,\frac{2kk'}{k+k'}{k-1 \choose (k-1)/2}{k'-1 \choose (k'-1)/2}.
\end{split}
\end{equation}
This finishes the proof of Theorem \ref{Theorem_CLT_moments}. \ep

\section{Coupling with Brownian bridge local times}
\label{Section_appendix_coupling}

This section is devoted to the proof of Proposition \ref{lemma_coupling}. The starting point is the following immediate consequence of \cite[Theorem 6.4]{LTF}.

\begin{proposition}\label{Lawler_coupling}
For every $C>0$, one can find a $\tilde{C}>0$ such that, for every $N\in\nn$, there exists a probability space supporting a random walk bridge $X^{x,y;N,T_N}$ and a standard Brownian bridge $\tilde{B}^{x,y;N}$ connecting $\lfloor N-N^{1/3}x\rfloor$ to $\lfloor N-N^{1/3}y\rfloor$ in time $T_N N^{2/3}$ satisfying
\begin{equation}
\pp\Big(\sup_{0\le t\le T_N} \big|X^{x,y;N,T_N}(t)-\tilde{B}^{x,y;N}(tN^{2/3})\big|\ge\tilde{C}\,\log N\Big)\le \tilde{C}N^{-C}.
\end{equation}
\end{proposition}

\smallskip

Note that, for each $N\in\nn$, the process $B^{x,y;N}(t):=N^{-1/3}\big(N-\tilde{B}^{x,y;N}(tN^{2/3})\big)$, $t\in[0,T_N]$ is a standard Brownian bridge connecting $N^{-1/3}\big(N-\lfloor N-N^{1/3}x\rfloor\big)$ to $N^{-1/3}\big(N-\lfloor N-N^{1/3}y\rfloor\big)$. Moreover, under the coupling of Proposition \ref{Lawler_coupling},
\begin{equation}\label{coupling_LDbound}
\pp\Big(\sup_{0\le t\le T_N} \big|N^{-1/3}\big(N-X^{x,y;N,T_N}(t)\big)-B^{x,y;N}(t)\big|\ge \tilde{C}N^{-1/3}\log N\Big)\le \tilde{C}N^{-C}.
\end{equation}

\smallskip

We are now ready to describe our coupling. We start with a probability space supporting a standard Brownian bridge $B^{x,y}$ connecting $x$ to $y$ in time $T$. Thanks to the Brownian scaling property, for every $N\in\nn$, the law of the process
\begin{equation}
\begin{split}
& \bigg(B^{x,y}\bigg(t\,\frac{T}{T_N}\bigg)-\bigg(1-\frac{t}{T_N}\bigg)\,x-\frac{t}{T_N}\,y\bigg)\,\bigg(\frac{T}{T_N}\bigg)^{1/2} \\
& +\!\bigg(\!1-\frac{t}{T_N}\!\bigg)\Big(\!N^{-1/3}\big(N-\lfloor N-N^{1/3}x\rfloor\big)\!\Big)\!+\!\frac{t}{T_N}\Big(\!N^{-1/3}\big(N-\lfloor N-N^{1/3}y\rfloor\!\Big),\;t\in[0,T_N]
\end{split}
\end{equation}
is that of the Brownian bridge $B^{x,y;N}$ defined above. Therefore, we can use the regular
conditional distributions of $X^{x,y;N,T_N}$ given $B^{x,y;N}$ (see e.g. \cite[Theorem 5.3]{Ka}) to
construct copies of the random walk bridges $X^{x,y;N,T_N}$, $N\in\nn$ (that we also denote by
$X^{x,y;N,T_N}$, $N\in\nn$) on an enlargement of the probability space on which $B^{x,y}$ is
defined.

\medskip

We need to show that the described coupling has the properties \eqref{XtoB0} and \eqref{XtoB}. To see \eqref{XtoB} one only needs to combine \eqref{coupling_LDbound}, the Borel-Cantelli Lemma, and the L\'{e}vy Modulus of Continuity Theorem:
\begin{equation}
\pp\Bigg(\limsup_{\epsilon\to0} \;\frac{\sup_{\substack{0\le t_1<t_2\le T \\ t_2-t_1\le\epsilon}} \big|B^{x,y}(t_2)-B^{x,y}(t_1)\big|}{\sqrt{2\epsilon\log(1/\epsilon)}}=1\Bigg)=1
\end{equation}
(see e.g. \cite[Theorem 2.9.25]{KS} and recall that $B^{x,y}$ can be obtained from a Brownian motion by subtracting an affine function of its terminal value).

\medskip

Our proof of the property \eqref{XtoB0} relies on \cite[Lemma 3.1]{BK2} which is stated next for the time interval $[0,T]$ instead of $[0,1]$ as in \cite{BK2}.

\begin{lemma}
Suppose $f_1$ and $f_2$ are measurable functions on $[0,T]$ which possess local times. Then, for every $\epsilon>0$,
\begin{equation}\label{BK_lemma}
\sup_{h\in\rr}\,\big|L_h(f_1)-L_h(f_2)\big|\le \frac{T}{\epsilon^2}\sup_{0\le t\le T} |f_1(t)-f_2(t)|
+\sum_{i=1}^2 \,\sup_{\substack{h_1,h_2\in\rr\\|h_1-h_2|\le\epsilon}}\big|L_{h_1}(f_i)-L_{h_2}(f_i)\big|.
\end{equation}
\end{lemma}

\smallskip

We aim to apply the lemma for every fixed $N\in\nn$ with $f_1:=B^{x,y}$, $f_2:=X^{x,y;N,T_N}$, and $\epsilon:=N^{-2/15}$. To this end, we recall that there exists a random variable $\mathcal C$ such that
\begin{equation}\label{BB_mod_cont}
\sup_{\substack{h_1,h_2\in\rr\\|h_1-h_2|\le N^{-2/15}}}\big|L_{h_1}(B^{x,y})-L_{h_2}(B^{x,y})\big|\le\mathcal C\,N^{-1/15}\,(\log N)^{1/2},\quad N\in\nn
\end{equation}
with probability one (note that the laws of $B^{x,y}(t)-x$, $t\in[0,T/2]$ and
$B^{x,y}(T-t)-y$, $t\in[0,T/2]$ are absolutely continuous with respect to that of a standard Brownian motion on $[0,T/2]$ and apply \cite[estimate (2.1)]{Tr} for standard Brownian motion). In addition, we have the following lemma.

\begin{lemma}\label{lemma_RW_cont}
 For every
$\tilde{\epsilon}>0$ there exists a random variable $\mathcal C_{\tilde \epsilon}$
such that
\begin{equation}\label{RWB_mod_cont}
\sup_{\substack{h_1,h_2\in\rr\\|h_1-h_2|\le
N^{-2/15}}}\big|L_{h_1}(X^{x,y;N,T_N})-L_{h_2}(X^{x,y;N,T_N})\big|\le\mathcal
C_{\tilde \epsilon}\,N^{-1/15+\tilde{\epsilon}},\quad N\in\nn
\end{equation}
with probability one.
\end{lemma}

Given the lemma, \eqref{XtoB0} follows by inserting \eqref{XtoB},
\eqref{BB_mod_cont}, and \eqref{RWB_mod_cont} into \eqref{BK_lemma} and choosing $\tilde{\epsilon}>0$ small enough.

\medskip

\noindent\textit{Proof of Lemma \ref{lemma_RW_cont}.} 
We only consider even $\lfloor N- N^{1/3}x\rfloor$, $\lfloor N-N^{1/3}y\rfloor$, and $T_N N^{2/3}$, since the proof in all other cases is analogous. In addition, let us first study $h_1$, $h_2$ of the form $N^{-1/3}(2a)$ for some $a\in\zz$. Then, we can use \cite[Proposition 3.1]{BK1} for a simple symmetric random walk with $T_N N^{2/3}$ steps started at $\lfloor N-N^{1/3}x\rfloor$ and restricted to even times (such restriction is needed to satisfy the strong aperiodicity assumption of \cite{BK1}). Conditioning the random walk on ending at $\lfloor N-N^{1/3}y\rfloor$ we get the estimate
\begin{equation} \label{eq_Bound_of_BK}
\begin{split}
& \pp\Bigg(\sup_{\substack{h_1,h_2\in N^{-1/3} (2\mathbb Z)\\|h_1-h_2|\le N^{-2/15}}}\big|L_{h_1}(X^{x,y;N,T_N})-L_{h_2}(X^{x,y;N,T_N})\big| \ge N^{-(1-\tilde{\epsilon})/15} \lambda\Bigg) \\ 
& \le C N^{1/3} \big(e^{-\lambda/C}+ N^{-14/3}\big),
\end{split}
\end{equation}
with a constant $C<\infty$ depending only on $\tilde{\epsilon}$ and where the factor $N^{1/3}$ on the right-hand side results from conditioning. Taking $\lambda=N^{\tilde \eps/2}$ and using the Borel-Cantelli Lemma we end up with \eqref{RWB_mod_cont} for $h_1$, $h_2$ restricted to $N^{-1/3} (2\mathbb Z)$.

\medskip

We remove the restriction on $h_1$, $h_2$ in two steps. First, we allow $h_1$, $h_2$ to be also of the form $N^{-1/3}(2a+e)$ for some $a\in\zz$, $e\in(-1,1)\backslash\{0\}$. Without loss of generality, we can take $e=\pm\frac{1}{2}$. In addition, we note that, given $L_{N^{-1/3}(2a)}(X^{x,y;N,T_N})$, each excursion of $X^{x,y;N,T_N}$ between two consecutive visits to $N-2a$ is upward or downward, each with probability $\frac{1}{2}$ independently of the other such excursions. It follows that conditionally on the event $L_{N^{-1/3}(2a)}(X^{x,y;N,T_N})=u$, the number of visits $N^{1/3}L_{N^{-1/3}(2a+e)}(X^{x,y;N,T_N})$ is twice a  binomial random variable with parameters $N^{1/3}u$, $\frac{1}{2}$, up to an error of at most $4$ (possibly arising before the first or after the last visit to $N-2a$). Therefore, a standard tail bound for the binomial distribution (see e.g. \cite[Corollary 1]{Ma} for a much stronger result) shows that the conditional probability
\begin{equation}
\pp\Big(\big|L_{N^{-1/3}(2a+e)}(X^{x,y;N,T_N})-u\big|>
N^{-1/15}u^{1/2}\,\big|\,L_{N^{-1/3}(2a)}(X^{x,y;N,T_N})=u\Big)
\end{equation}
decays faster than polynomially as $N$ tends to infinity, uniformly in $u$. Since we know from the above that $L_{N^{-1/3}(2a)}(X^{x,y;N,T_N})=O(1)$ as $N\to\infty$ for all $a\in\zz$ with probability one, we can condition on this event and then apply the Borel-Cantelli Lemma to include $h_1$, $h_2$ of the form $N^{-1/3} (2a+e)$ in \eqref{RWB_mod_cont}.

\medskip

Finally, to incorporate $h_1$, $h_2$ of the form $N^{-1/3}(2a+1)$ for some $a\in\mathbb Z$ it suffices to use the combinatorial identity
$$
\bigg| L_{N^{-1/3}(2a+1)}(X^{x,y;N,T_N})-\frac{
 L_{N^{-1/3}(2a+1/2)}(X^{x,y;N,T_N})\!+\!L_{N^{-1/3}(2a+3/2)}(X^{x,y;N,T_N})}{2} \bigg|\!\le\! N^{-1/3}.
$$
This finishes the proof of the lemma and thereby of Proposition \ref{lemma_coupling}. \ep

\bigskip


\end{document}